\documentclass[a4paper,11pt]{amsart}
\usepackage{fullpage}
\usepackage[english]{babel}
 \usepackage{tikz}
\definecolor{liens}{rgb}{1,0,0}
\usepackage[colorlinks=true, linkcolor=blue, 
hyperfootnotes=true,citecolor=blue,urlcolor=black]{hyperref}
\newtheorem{thm}{Theorem}[section]
\newtheorem{lem}[thm]{Lemma}
\newtheorem{prop}[thm]{Proposition}
\newtheorem{prop/def}[thm]{Proposition/Definition}
\newtheorem{cor}[thm]{Corollary}

\theoremstyle{definition}
\newtheorem{defn}[thm]{Definition}

\theoremstyle{remark}
 
\newtheorem{rem}[thm]{Remark}

\newcommand{\bZZ}{\boldsymbol{Z}}
\newcommand{\Q}{\mathbb{Q}}
\newcommand{\Z}{\mathbb{Z}}
\newcommand{\T}{\vartheta}
\newcommand{\R}{\mathbb{R}}
\newcommand{\N}{\mathbb{N}}
\newcommand{\C}{\mathbb{C}}

\newcommand{\ZZ}{Z}
\newcommand{\y}{\boldsymbol{y}}

\newcommand{\zJS}{\boldsymbol{z}}
 
 \newcommand{\longversion}[1]{ {\color{teal}  
{\small  Version longue de la preuve disponible en d\'ecommentant dans la pr\'eambule.}
 }}    
 
  \newcommand{\Itwo}{\mathrm{I}_2}     
  \newcommand{\qOka}[1]{q\mathrm{Oka}_{#1}}
   \newcommand{\qOkaMod}[1]{q\widetilde{\mathrm{Oka}}_{#1}}

 \newcommand{\qDom}{\mathfrak{Q}}

\author[T. Dreyfus]{Thomas Dreyfus}
\address{Institut de Recherche Math\'ematique Avanc\'ee, U.M.R. 7501 Universit\'e de Strasbourg et C.N.R.S. 7, rue Ren\'e Descartes 67084 Strasbourg, FRANCE}
\email{dreyfus@math.unistra.fr}
\author[V. Heu]{Viktoria Heu}
\address{ IRMA\\
 7 rue Ren\'e Descartes\\
 67084 Strasbourg\\France}
\email{heu@math.unistra.fr}
\begin{document}
\title{Degeneration from difference to differential Okamoto spaces for the sixth Painlev\'e equation}
\subjclass[2010]{14D05, 14F35, 34M56, 39A13}
 \sloppy
\begin{abstract} 
In the current paper we study the $q$-analogue introduced by Jimbo and Sakai of the well known Painlev\'e VI differential equation. We explain how it can be deduced from a $q$-analogue of Schlesinger equations and show that for a convenient change of variables and auxiliary parameters, it admits a $q$-analogue of Hamiltonian formulation. This allows us to show that Sakai's $q$-analogue of Okamoto space of initial conditions for $qP_\mathrm{VI}$ admits the differential Okamoto space  \emph{via} some natural limit process.  \end{abstract}
  \thanks{This work took place at IRMA. It was supported by ANR-13-JS01- 0002-01 and ANR-19-CE40-0008.}
\maketitle
\tableofcontents

 \section*{Introduction}
 
In \cite{jimbo1996q}, Jimbo and Sakai have introduced a $q$-analogue of the Painlev\'e VI equation, namely the following system of $q$-difference equations: 
$$(qP_{\mathrm{JS,VI}}) : \left\{\begin{array}{rcl}
\displaystyle  \frac{\y \cdot \sigma_{q,t} \y }{a_3 a_4}  &=&\displaystyle  \frac{\left( \sigma_{q,t}  \zJS - t \frac{a_1 a_2}{\vartheta_1} \right)\left(\sigma_{q,t}  \zJS - t\frac{a_1 a_2}{\vartheta_2} \right)}{\left(\sigma_{q,t}  \zJS -\frac{1}{q\kappa_1}  \right)\left(\sigma_{q,t}  \zJS - \frac{1}{\kappa_2} \right)} 
\vspace{.2cm} \\
\displaystyle   \frac{\zJS \cdot \sigma_{q,t}\zJS }{\frac{1}{q\kappa_1\kappa_2}} &=&\displaystyle  \frac{\left(\y -ta_1 \right)\left(\y -ta_2 \right)}{\left(\y -a_3\right)\left(\y -a_4 \right)}\, ,  
\end{array}\right.$$
where $\kappa_1, \kappa_2, \vartheta_1, \vartheta_2, a_1,a_2,a_3,a_4 \in \C^*$ are parameters subject to the relation 
$$a_1 a_2 a_3 a_4 \kappa_1 \kappa_2 =\vartheta_1 \vartheta_2 .$$ 
Here $q$ is a  complex parameter that is neither zero nor one, and  $\sigma_{q,t}$ is the operator which to a function $f(t)$ associates $f(q\cdot t)$. The $q$-derivative $$\partial_{q,t}:=\frac{\sigma_{q,t}-1}{(q-1)t}$$ formally converges, when $q\to 1$, to the classical derivative $\partial_t$ (differentiation with respect to $t$). It has been shown in \cite{jimbo1996q} that the classical Painlev\'e~VI equation may be obtained by some limit process, when $q$ goes to $1$, from its $q$-analogue. More precisely, by a series of changes of variables and parameters, $qP_{\mathrm{JS,VI}}$ formally yields a certain system of differential equations with eight complex parameters, subject to one relation. As one can easily check, one can then further normalize these parameters to a quadruple $\boldsymbol{\theta}=(\theta_0,\theta_1,\theta_t,\theta_\infty)$ of complex parameters  such that this  system of differential equations  is the non-autonomous Hamiltonian system 
$$(P_{\mathrm{VI}}) : \left\{\begin{array}{rcl}
\displaystyle  \partial_t \y &=& \displaystyle  \phantom{+}\partial_{\bZZ} H_{\mathrm{VI}}^{\boldsymbol{\theta}}(\y,\bZZ,t)
\vspace{.2cm} \\
\displaystyle   \partial_t \bZZ &=&\displaystyle  -\partial_{\y} H_{\mathrm{VI}}^{\boldsymbol{\theta}}(\y,\bZZ,t)\, ,
\end{array}\right.$$ 
where $H_{\mathrm{VI}}^{\boldsymbol{\theta}}(\y,\bZZ,t)$ is given by 
$$\begin{array}{l}\displaystyle   \frac{\y (\y -1)(\y -t)}{t(t-1)}\left(\bZZ^2+\frac{\bZZ}{\y -t}\right)-\frac{1}{4}\left(\frac{(\theta_\infty -1)^2-1}{t(t-1)}\y +\frac{\theta_0^2}{(t-1)\y }+\frac{\theta_t^2 }{\y -t}-\frac{\theta_1^2}{t(\y -1)}\right)\, .\end{array}$$
This non-autonomous Hamiltonian system $(P_{\mathrm{VI}})$ is actually the one discovered in \cite{okamoto1986studies} that, when reformulated as a single second order differential equation in $\y$, yields the sixth Painlev\'e differential equation with auxiliary parameters
  $\boldsymbol{\theta}$. Given a generic initial condition, \emph{i.e.} $\y_{0}\in \C\setminus \{0,1,t_0\}$, and $\bZZ_{0} \in \C$, Cauchy's theorem implies the existence and uniqueness of a germ at $t_0\neq 0,1$ of associated holomorphic solution of $(P_{\mathrm{VI}})$. As shown in  \cite{okamoto1986studies}, it is moreover possible to give a meaning  to solutions including non-generic initial conditions. More precisely, 
Okamoto's space of initial conditions  at a fixed time $t_0\in \C\setminus \{0,1\}$ is the second Hirzebruch surface $\mathbb{F}_2$ blown up in eight points, whose position is encoded by $\boldsymbol{\theta}$ and $t_0$, minus a divisor formed by five irreducible components of self-intersection number $(-2)$ related to each other according to the following intersection diagram:

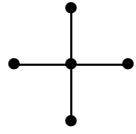
\begin{center}
\begin{tikzpicture}[scale=0.25]
 
\draw[thick] (0,-3)--(0,3);
\draw[thick] (-3,0)--(3,0);
  \draw (0,3) node {$\bullet$} ;
    \draw (0,0) node {$\bullet$} ;
    \draw (0,-3) node {$\bullet$} ;
      \draw (3,0) node {$\bullet$} ;
        \draw (-3,0) node {$\bullet$} ;
          \end{tikzpicture}
\end{center}

Here each node represents an irreducible component, and nodes share a common edge if and only if they intersect each other. For each point in Okamoto's space of initial conditions at $t_0$, there exists a unique associated germ (at $t_0$) of meromorphic solution of 
$(P_{\mathrm{VI}})$. 
A $q$-analogue of Okamoto's space for $(qP_{\mathrm{JS,VI}})$ for some fixed generic time $t_0$ was found in \cite{sakai2001rational}. It is given by $\mathbb{F}_0=\mathbb{P}^1\times \mathbb{P}^1$, blown up in eight points, whose positions are encoded by $a_1, a_2 , a_3, a_4 , \kappa_1,  \kappa_2 , \vartheta_1,  \vartheta_2$ and $t_0$, minus a divisor formed of four irreducible components of self-intersection number $(-2)$, arranged according to the following intersection diagram:
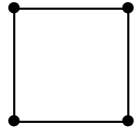
\begin{center}
\begin{tikzpicture}[scale=0.25]
 
\draw[thick] (-3,-3)--(-3,3);
\draw[thick] (3,-3)--(3,3);
\draw[thick] (-3,-3)--(3,-3);
\draw[thick] (-3,3)--(3,3);
  \draw (3,3) node {$\bullet$} ;
    \draw (3,-3) node {$\bullet$} ;
      \draw (-3,3) node {$\bullet$} ;
        \draw (-3,-3) node {$\bullet$} ;
          \end{tikzpicture}
\end{center}
For each point in Sakai's $q$-analogue of Okamoto's space of initial conditions at $t_0$, there exists a unique \emph{discrete} solution of $(qP_{\mathrm{JS,VI}})$, which, roughly speaking, encodes the values at $q^\Z t_0$ that a meromorphic solution with prescribed value at $t_0$, if it exists,  should interpolate. The  questions adressed in the present paper are the following. 
\begin{enumerate}
\item[Q1)] How can Okamoto's space of initial values at $t_0$ for $(P_{\mathrm{VI}})$ be obtained \emph{via} a natural limit process from its discrete analogue?\vspace{.2cm}
\item[Q2)] How can meromorphic solutions of $(P_{\mathrm{VI}})$ be obtained  \emph{via} a natural limit process from their discrete analogue?\vspace{.2cm}
\end{enumerate}
Let us first answer question $(Q1)$ informally. What we will obtain in Section \ref{sec:4} is that one of the four irreducible components of the boundary of the $q$-Okamoto space at $t_0$ does degenerate at the limit $q\to 1$: for $q=1$, it is no longer irreducible, but is itself the union of three irreducible components of self-intersection $(-2)$, one of which coincides with the limit of a non-degenerating one, and the other two of which intersect only the latter. 
So in terms of intersection diagrams, the informal answer to question $(Q1)$ is the following. 
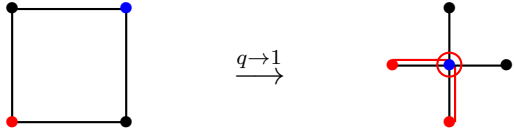
\begin{center}
\begin{tikzpicture}[scale=0.25]
 
\draw[thick] (-13,-3)--(-13,3);
\draw[thick] (-7,-3)--(-7,3);
\draw[thick] (-13,-3)--(-7,-3);
\draw[thick] (-13,3)--(-7,3);
  \draw (-7,-3) node {$\bullet$} ;
    \draw (-7,3) node {\textcolor{blue}{$\bullet$}} ;
      \draw (-13,3) node {$\bullet$} ;
        \draw (-13,-3) node {\textcolor{red}{$\bullet$}} ;
        \draw(0,0) node {\large $\stackrel{q\to 1}{\longrightarrow }$} ;
        \draw[thick] (10,-3)--(10,3);
            \draw[thick, red] (10.3,-3)--(10.3,0);
\draw[thick] (7,0)--(13,0);
\draw[thick, red] (7,0.3)--(10,0.3);
  \draw (10,3) node {$\bullet$} ;
    \draw (10,0) node {\textcolor{blue}{$\bullet$}} ;
\draw[thick,red] (10,0.05) circle (18pt);
    \draw (10,-3) node {\textcolor{red}{$\bullet$}} ;
      \draw (13,0) node {$\bullet$} ;
        \draw (7,0) node {\textcolor{red}{$\bullet$}} ;

          \end{tikzpicture}
\end{center}
The key to this result (see Section \ref{sec:ConfOkaSpaces} for a precise formulation) is to coveniently identify normalizations and changes of variables in $(qP_{\mathrm{JS,VI}})$ \emph{before} considering the limit process, such that the limit when $q\to 1$ \emph{is} $(P_{\mathrm{VI}})$. To this end, we retrace, with some alterations, the method by which in \cite{jimbo1996q}, the $q$-difference equation  $(qP_{\mathrm{JS,VI}})$ has been obtained from a $q$-analogue of isomonodromic deformations. 
 Here  the isomonodromy condition to be considered concerns certain families, parametrized by a time variable $t$, of $q$-Fuchsian systems  of rank $2$, which for fixed $q$  are of the form
$$\sigma_{q,x} Y(x,t)=\mathfrak{A}(x,t)Y(x,t),\quad  \textrm{ with }\quad   \mathfrak{A}(x,t)=\mathfrak{A}_{0}(t)+x\frac{\mathfrak{A}_{1}(t)}{x-1}+x\frac{\mathfrak{A}_{t}(t)}{t(x-t)}\, ,$$
  where $x$ is the standard coordinate on $\C\subset \mathbb{P}^1.$ As shown in \cite{jimbo1996q}, under certain generic conditions, for given spectral parameters  $\kappa_1, \kappa_2, \vartheta_1, \vartheta_2, a_1,a_2,a_3,a_4 \in \C^*$ as above, such a family can be encoded by a triple of functions $(w(t),\y(t),\zJS(t))$. We show in Section \ref{sec:confFuchsSys} that a convenient change of variables and parameters (including a normalization) that is both compatible with the definition of the considered $q$-Fuchsian systems in \cite{jimbo1996q} and that yields tracefree differential Fuchsian systems at the limit when $q\to 1$ is the following setting:
\begin{equation}\label{introeq1}
\boldsymbol{\lambda}=\frac{\kappa_2w}{q-1}\, , \quad \quad \bZZ=\frac{\frac{(\y-ta_1)(\y-ta_2)}{q(\y-1)(\y-t)\zJS}-1}{(q-1)\y}\, , 
\end{equation}
\begin{equation}\label{introeq2}
\left\{ \begin{array}{rclcrclcrclcrclc}\displaystyle 
a_1&=&1+(q-1)\frac{\theta_t}{2}, &&a_2&=&\frac{1}{a_1},&&
a_3&=&1+(q-1)\frac{\theta_1}{2}, &&\vspace{.2cm}\\ \displaystyle  a_4&=&\frac{1}{a_3},&&
\vartheta_1&=&1+(q-1)\frac{\theta_0}{2}, &&\vartheta_2&=&\frac{1}{\vartheta_1},&&\vspace{.2cm}\\
\displaystyle \kappa_1&=&\frac{1}{\kappa_2},&&\kappa_2&=&1+(q-1)\frac{\theta_\infty}{2} \, .
\end{array}\right.
\end{equation}
In Section \ref{sec1}, we both simplify and generalize the definition of $q$-isomonodromy in \cite{jimbo1996q} into the requirement that there exists a certain matrix 
$\mathfrak{B}(x,t)$, depending rationally on $x$, such that the system of $q$-difference equations
$$\left\{ \begin{array}{rcl} \sigma_{q,x}Y(x,t)&=&\mathfrak{A}(x,t)Y(x,t)\vspace{.2cm}\\
\sigma_{q,t}Y(x,t)&=&\mathfrak{B}(x,t)Y(x,t)\end{array}\right.$$
satisfies the $q$-analogue of integrability, namely $\sigma_{q,t}\mathfrak{A}\cdot \mathfrak{B}=\sigma_{q,x}\mathfrak{B}\cdot \mathfrak{A}\, .$ We further introduce a sufficient condition for $q$-isomonodromy, which we call $q$-Schlesinger isomonodromy. We show that under some generic conditions such as the non-resonancy condition $\frac{\vartheta_1}{\vartheta_2}\not \in q^{\Z^*}$ and $\frac{\kappa_1}{\kappa_2}\not \in q^{\Z^*}$,  this $q$-Schlesinger isomonodromy of $\sigma_{q,x} Y=\mathfrak{A}Y$ is equivalent to the following \emph{$q$-Schlesinger equations}
    $$\left\{ \begin{array}{rcl}
   \sigma_{q,t}\mathfrak{A}_0&=&B_0\mathfrak{A}_0 B_0^{-1} \vspace{.2cm}\\
    \sigma_{q,t}\mathfrak{A}_1&=&\frac{t-1}{qt-1}\frac{(qta_1-1)(qta_2-1)}{q(ta_1-1)(ta_2-1)}\cdot (q\Itwo +B_0) \, \mathfrak{A}_1\left(\Itwo +B_0 \right)^{-1} \vspace{.2cm}\\
     \sigma_{q,t}\mathfrak{A}_t&=& - B_0\mathfrak{A}_0  \left(\frac{1}{a_1a_2}\Itwo +qtB_0^{-1}\right)\vspace{.2cm}\\
     &&-\frac{t(t-1)}{(ta_1-1)(ta_2-1)}\cdot (q\Itwo +B_0) \, \mathfrak{A}_1 \left( \Itwo +\frac{(qta_1-1)(qta_2-1)}{qt-1}\left(\Itwo +B_0 \right)^{-1}\right) \, ,
   \end{array} \right.$$
where   $$B_0:=-qt\left(\mathfrak{A}_0+\mathfrak{A}_1+\frac{1}{t}\mathfrak{A}_t+\frac{t-1}{(ta_1-1)(ta_2-1)}\mathfrak{A}_1\right)\left(\frac{1}{a_1a_2}\mathfrak{A}_0+\frac{t(t-1)}{(ta_1-1)(ta_2-1)} \mathfrak{A}_1\right)^{-1}\, .$$
Moreover, we show in Section \ref{sec:confSchlesinger}  that when the spectral values are functions of $q$ given by \eqref{introeq2}, then these $q$-Schlesinger equations yield the usual (differential) Schlesinger equation at the limit $q\to 1$. When $\mathfrak{A}$ is expressed with respect to $(w(t),\y(t),\zJS(t))$  and the spectral parameters, then these  $q$-Schlesinger equations are (generically) equivalent to a system of $q$-difference equations given by $(qP_{\mathrm{JS,VI}})$ and an additional equation for $\sigma_{q,t}w$. However, our weaker definition of $q$-isomonodromy is actually (generically) equivalent to $(qP_{\mathrm{JS,VI}})$.  We conclude that the change of variables and parameters \eqref{introeq1} combined with \eqref{introeq2} is a natural setting for the study of confluence of the $q$-Painlev\'e VI equation. And indeed, for generic values of $q$, the change of variable \eqref{introeq1}  defines a biregular transformation of Sakai's $q$-Okamoto space (see Section \ref{sec:qoka}), such that the obtained modified $q$-Okamoto space yields, for spectral values \eqref{introeq2},  the differential Okamoto space when $q \to 1$ (see Section \ref{sec:ConfOkaSpaces}). 

Under convenient assumptions on $q$ and the spectral data, some meromorphic solutions of $(qP_{\mathrm{JS,VI}})$ defined in a convenient sectorial neighborhood of $t=0$ have been constructed in \cite{mano2010asymptotic,Ohyama}. On the other hand, contrary to the differential setting, the following question remains open:  for a given generic value of $t_{0}$ and generic initial condition $(\y(t_0),\zJS(t_0))$,  does there exist an associated meromorphic solution of $(qP_{\mathrm{JS,VI}})$, defined on a connected subset of $\C^*$ stable under multiplication by $q^{\pm 1}$? Of course the answer is likely to depend on particular choices of $q$ and the spectral parameters. For example, when $q$ is a $n$-th root of unity, then a necessary condition for the existence of meromorphic solution is that the spectral parameters are chosen in a way such that the $n$-th iterate of  $(qP_{\mathrm{JS,VI}})$  is the identity. The question whether meromorphic solutions of $(qP_{\mathrm{JS,VI}})$ over  convenient set, parametrized by $q$, admit a limit when $q\to 1$ also seems difficult and remain open.
\par 
Much more abordable is the question of the existence of discrete solutions, essentially solved in \cite[Prop. 1]{sakai2001rational}, which, roughly speaking, encode the values at $q^\Z t_0$ that  meromorphic solutions with prescribed value at $t_0$, if they exist,  should interpolate. More precisely, a discrete solution with initial value $(\y_0, \zJS_0)\in \C^*\times \C^*$ at $t_0\in \C^*$ is a sequence $(  \y_{\ell} ,  \zJS_{\ell}, t_\ell)_{\ell \in \Z}$ of points in $ \mathbb{P}^{1}\times \mathbb{P}^{1}\times\C^*$ such that for $\ell \neq 0$, we have $t_{\ell}=q^{\ell } t_0$ and  $( \y_{\ell} ,  \zJS_{\ell})=\left(f_\ell(\y_0, \zJS_0,t_0,q),g_\ell(\y_0, \zJS_0,t_0,q)\right)$, where $f_\ell,g_\ell$ are the rational functions invariables $\y,\zJS,t,q$ such that the $\ell$-th iterate of $(qP_{\mathrm{JS,VI}})$ is of the form 
$$\left\{ \begin{array}{rcl}
  \sigma_{q,t}^\ell \y &=&  f_\ell(\y, \zJS,t,q) \vspace{.2cm}\\
 \sigma_{q,t}^\ell \zJS &=&  g_\ell(\y, \zJS,t,q)\, .
  \end{array} \right.$$

We refer to Section \ref{sec:32} for more details. It is shown in \cite[Prop. 1]{sakai2001rational} that discrete solutions with initial value in $\C^*\times \C^*$ at $t_0$ are well defined, in particular they exist and are unique, if $t_0$ and the spectral values are generic. Moreover, under this assumption, one can consider a space of initial values bigger than $\C^*\times \C^*$, namely the $q$-Okamoto space. We specify in Section \ref{sec:qoka} which are the special values for $t_0$ and the spectral parameters that need to be excluded here. Note that a discrete solution, as a sequence, does make sense even if $q$ is a root of unity. 

Of course there is an analogous notion of discrete solution  for the \emph{modified} $q$-Painlev\'e VI equation obtained by applying the change of variables and parameters \eqref{introeq1}, \eqref{introeq2} to $(qP_{\mathrm{JS,VI}})$. We prove in Section \ref{sec:confPainl} that the therby obtained system of $q$-difference equations is a $q$-analogue of Hamiltonian system. More precisely, it is given by 
$$ (q\widetilde{P}_{\mathrm{VI}})  : \left\{ \begin{array}{rcl}
  \partial_{q,t}\y &=& \phantom{+}\partial_{q,\bZZ}H^{\boldsymbol{\theta}}_{\mathrm{VI}}(\y,\bZZ,t)+  (q-1) \mathcal{R}^{\boldsymbol{\theta}}_1(\y,\bZZ,t,q) \vspace{.2cm}\\
 \partial_{q,t}\bZZ &=&  \, -\partial_{q,\y}H^{\boldsymbol{\theta}}_{\mathrm{VI}}(\y,\bZZ,t)+  (q-1) \mathcal{R}^{\boldsymbol{\theta}}_2(\y,\bZZ,t,q)\, , 
    \,    \end{array} \right. $$
    where $H^{\boldsymbol{\theta}}_{\mathrm{VI}}$ is the Hamiltonian from the (differential) $({P}_{\mathrm{VI}})$ and for $i\in \{1,2\}$,  $\mathcal{R}^{\boldsymbol{\theta}}_i$ is some rational function such that  
    $\mathcal{R}^{\boldsymbol{\theta}}_i|_{q=1}$ is well defined and does not have poles outside the polar locus of~$H^{\boldsymbol{\theta}}_{\mathrm{VI}}$. 
    
    From this we deduce, also in Section \ref{sec:confPainl}, the answers to question $(Q2)$.
For $t_0\in \C^*$ and $(\y_0,\bZZ_0)\in \left(\C\setminus \{0,1,t_0\}\right)\times \C$, the sequence $(  \y_{\ell}(q) ,  \bZZ_{\ell}(q), q^\ell t)_{\ell \in \N}$ of triples defining the corresponding discrete solution, but seen as rational functions of $q$, is well defined, and encodes in some precise manner the Taylor series coefficients of the unique solution of $(P_{\mathrm{VI}})$ with initial condition $(\y_0,\bZZ_0)$ at $t_0$.\\  \par 

Each of the four sections following this introduction is decomposed into three parts. Each time, in the first part we briefly recall  some notions and known results in the differential case. In the second part, their $q$-analogues are discussed, and in the third part, confluence is adressed. Concerning the $q$-analogues, we usually recall some results from \cite{jimbo1996q} and \cite{sakai2001rational}, complemented by some precisions that we deemed helpful, and to which we add  new results.  We finish by an  appendix, see Section \ref{app:qFundsols}, explaining how our notion of $q$-isomonodromy is related to the one in \cite{jimbo1996q}.\\ \par 
 
\begin{center}
In this paper, we adopt the following (standard) notation.\end{center}
\begin{center}
\begin{tabular}{ll}
$\mathrm{GL}_2(R)$& the ring of invertible $2\times 2$ matrices with coefficients in a ring $R$.\\
$\mathrm{SL}_2(R)$& the ring of invertible $2\times 2$ matrices of determinant $1$ with coefficients in $R$.\\
$\mathrm{M}_2(R)$& the algebra
of $2\times 2$ matrices with coefficients in a ring $R$. \\
$\mathfrak{sl}_2(M)$& the vector space  
of $2\times 2$ matrices of trace $0$ with coefficients in $M$, where \\ & $M$ is a $\C$-module.\\
$\Itwo $ & the identity matrix in $\mathrm{M}_2(R)$\\
$A^{(i,j)}$& the $(i,j)$-entry of a matrix $A$.\\
$\mathcal{O}(U)$& the ring of holomorphic functions on some complex domain $U\subset \C^n$.\\
$\mathcal{M}(U)$& the field of meromorphic functions on $U$.\\
$k[x_1, \dots, x_n]$ & the ring of polynomials in $n$ variables, named $x_1, \dots, x_n$, with\\
&coefficients in a field $k$.\\
$k(x_1, \dots, x_n)$& the fraction field of $k[x_1, \dots, x_n]$. \\  $\phantom{I}$\end{tabular}
\\
\end{center}

We would like to emphasize that some results in the sequel require stronger assumptions on the complex variable $q$ than $q\neq 0,1$.  These assumptions will of course be duly specified when needed. 
We choose not to accumulate these requirements along the way towards the $q$-Painlev\'e  VI  equation (which in and by itself is well-defined for $q\neq 0,1$), in order to get the full picture of possible $q$-Okamoto spaces. 
 
\section{Fuchsian systems}
\subsection{Differential case}\label{setup}

Let $\boldsymbol{\theta}=(\theta_0,\theta_1,\theta_t,\theta_\infty)\in  \mathbb{C}^4$ with $\theta_\infty \neq 0$. We say that $\boldsymbol{\theta}$ satisfies the   \emph{non-resonancy condition} if 
\begin{equation}\label{eq:resonance}
\forall i\in \{0,1,t,\infty\}\, , \quad \theta_i \not \in \Z^*\, .\end{equation}

We consider a linear partial differential equation of the form 
\begin{equation}\label{sl2System}\partial_x  Y(x,t)=A(x,t) Y(x,t)\, , \quad \textrm{with} \quad A(x,t)=\frac{A_0(t)}{x}+\frac{A_1(t)}{x-1}+\frac{A_t(t)}{x-t} \, .\end{equation}
Here $x$ is the standard coordinate on $\C$, seen as a subset of $\mathbb{P}^1=\C\cup\{\infty\}$, and $t$ is the standard coordinate on an open connected subset $U \subset \mathbb{C}\setminus \{0,1\}$.    
  
  \begin{defn}\label{defi1} We shall say that \eqref{sl2System} is a  \emph{family of $\mathfrak{sl}_2$-Fuchsian systems  with spectral data} $\boldsymbol{\theta}$ if the following hold:\vspace{.2cm} 
   \begin{itemize}
   \item[$\bullet$] for each $i\in \{0,1,t\}$, $A_i\in \mathfrak{sl}_2(\mathcal{O}(U))$, 
\item[$\bullet$] for all $i\in \{0,1,t\}$ and all $t\in U $, we have $$\mathrm{Spec}(A_i(t))=\left\{\frac{1}{2}\theta_i,-\frac{1}{2}\theta_i\right\}\, , $$
\item[$\bullet$] the residue $A_\infty:=-A_0-A_1-A_t$ at infinity is constant and  normalized as follows: $$ A_\infty\equiv \begin{pmatrix} \frac{\theta_\infty}{2}& 0 \\0 &- \frac{\theta_\infty}{2}\end{pmatrix}\, .$$
   \end{itemize}
   \end{defn}
  With this normalization, the $(1,2)$ entry of $x(x-1)(x-t)A$, seen as an element of $\mathcal{O}(U)[x]$, is a polynomial of degree at most one. Let us assume that it has degree one and define a non-zero holomorphic function $\lambda(t) \in \mathcal{O}(U )$ and a meromorphic function 
  $y(t)  \in \mathcal{M}(U )$ by 
\begin{equation}\label{defyZ} A^{(1,2)}(x,t)=\frac{\lambda(t)(x-y(t))}{x(x-1)(x-t)}\,  .\end{equation}
Assuming moreover that $y(t)\not \equiv 0,1,t$,  we may define a meromorphic function 
  $\ZZ(t)  \in \mathcal{M}(U )$ by 
\begin{equation}\label{defZ}   A^{(1,1)}(y(t),t)=\ZZ(t)\, .\end{equation}
The next Lemma shows that the matrix $A$ is determined by  the triple $(\lambda,y,\ZZ)$.
\begin{lem}\label{AparlamyZ} If a family of $\mathfrak{sl}_2$-Fuchsian systems \eqref{sl2System} with spectral data $\boldsymbol{\theta}$, with $\theta_\infty \neq 0$, gives rise to 
\begin{equation}\label{lamyZcond} 
(\lambda,y,\ZZ)\in  (\mathcal{O}(U )\setminus \{0\})\times (\mathcal{M}(U )\setminus\{0,1,\mathrm{id}\}) \times \mathcal{M}(U )
\end{equation}
as above, then the coefficients of the matrix $A$ necessarily are the following functions of $\lambda,y,\ZZ, x,t$ and $\boldsymbol{\theta}$: 
 $$\begin{array}{rccl} {A}^{(1,1)}&=&& \frac{(y-x)}{4\theta_\infty x(x-1)(x-t)}\left[y(y-1)(y-t)\left(2Z+\frac{\theta_\infty}{y-x}\right)^2-a\right]-\frac{\theta_\infty }{4}\left( \frac{1}{x}+\frac{1}{x-1}+\frac{1}{x-t}+\frac{1}{y-x}\right).\vspace{.2cm}\\
{A}^{(1,2)}&=&&\frac{\lambda(x-y)}{x(x-1)(x-t)}
\vspace{.2cm}\\
{A}^{(2,2)}&=&-&{A}^{(1,1)}\vspace{.2cm}\\
{A}^{(2,1)}&=&&\frac{1}{\lambda}\left(\frac{c_0}{x}+\frac{c_1}{x-1}-\frac{c_0+c_1}{x-t} \right) 
\vspace{.2cm}\\
c_0&=&&\frac{y}{t\theta_\infty^2}\left( {(y-1)(y-t)} (y\ZZ +\theta_\infty)\ZZ+\frac{(y-1-t)\theta_\infty^2-a}{4 } \right)^2- \frac{t\theta_0^2}{4y}
\vspace{.2cm}\\
c_1&=&& \frac{y-1}{(1-t)\theta_\infty^2}\left( {y(y-t)} ((y-1)\ZZ +\theta_\infty)\ZZ+\frac{(y+1-t)\theta_\infty^2-a}{4 } \right)^2- \frac{(1-t)\theta_1^2}{4(y-1)}.
\end{array}$$
Here we denote $$a:=\frac{t\theta_0^2}{y}-\frac{(t-1)\theta_1^2}{y-1}+\frac{t(t-1)\theta_t^2}{y-t}\, .$$
 \end{lem}
\begin{proof} This lemma can be deduced from the formulae in \cite[p. 443-444]{jimbo1981monodromy} by considering the tensor product of $\partial_xY=AY$ with $\partial_x\zeta = \left(\frac{\theta_0}{2x}+\frac{\theta_1}{2(x-1)}+\frac{\theta_t}{2(x-t)}\right)\zeta$.   \end{proof}
  \begin{rem} If we have an arbitrary meromorphic triple $(\lambda,y,\ZZ)$ as in \eqref{lamyZcond}, then \emph{via} the formulae in Lemma \ref{AparlamyZ} we can associate a family of Fuchsian systems. Note however that the coefficients of the matrix functions $A_0, A_1$ and $A_t$ then are meromorphic functions of $t$. If one wants to obtain holomorphic coefficients, one might have to restrict to the complement of a discrete subset in  $U$. Indeed, for example the product $\lambda y$ needs to be holomorphic. 
 \end{rem}
 \begin{rem}As explained in \cite[Sec. 4]{loray2016isomonodromic}, the condition $\theta_\infty \neq 0$ can actually be overcome if one works in a (conjugated) setting where $A_\infty$ is normalized to $\left( \begin{smallmatrix} \theta_\infty/2 &0\\1&-\theta_\infty/2\end{smallmatrix}\right)$.
 \end{rem}

\subsection{A discrete analogue}\label{AboutAfrak}

Let $q\in \C\setminus \{0,1\}.$
Let $\boldsymbol{\Theta}:=(\Theta_0,\Theta_1,\Theta_t,\Theta_\infty) \in (\C^*)^4$ and let $\overline{\boldsymbol{\Theta}}:=(\overline{\Theta}_0, \overline{\Theta}_1,\overline{\Theta}_t,\overline{\Theta}_\infty) \in (\C^*)^4$ with $\Theta_\infty\neq \overline{\Theta}_\infty$
 be two quadrupels subject to the following relation:
\begin{equation}\label{thetarel}
\Theta_0\overline{\Theta}_0=\Theta_\infty\overline{\Theta}_\infty\Theta_t\overline{\Theta}_t\Theta_1\overline{\Theta}_1
 \, .\end{equation}
 
 \begin{rem}  
 In order to motivate our choice of notation, let us indicate that with respect to confluence, will be led to consider $\Theta_i$ satisfying some relation with the $\theta_i$ from the differential context, and $\overline{\Theta}_i$ satisfying   $\overline{\Theta}_i=1/\Theta_i$ (see Section \ref{sec:confFuchsSys}). \end{rem} 
We say that $(\boldsymbol{\Theta},\overline{\boldsymbol{\Theta}})$ satisfies the  \emph{non-resonancy condition} if 
\begin{equation}\label{eq:qresonance}
 \forall i\in \{0,1,t,\infty\}\, , \quad \frac{\Theta_i}{\overline{\Theta}_i} \not \in q^{\Z^{*}} .\end{equation}
If $\zeta$ is a standard coordinate in a complex domain that is stable under multiplication by $q$ and $\frac{1}{q}$, then we define the following operators on functions of $\zeta$:
$${\sigma_{q,\zeta}:f(\zeta)\mapsto f(q\zeta)}\, \, , \quad \partial_{q,\zeta} :f(\zeta)\mapsto \frac{ f(q\zeta) -f(\zeta)}{(q-1)\zeta }\, .$$
 We consider families of linear $q$-difference systems of the form
\begin{equation}\label{eqq1}
\sigma_{q,x} Y(x,t)=\mathfrak{A}(x,t)Y(x,t),\quad  \textrm{ with }\quad   \mathfrak{A}(x,t)=\mathfrak{A}_{0}(t)+x\frac{\mathfrak{A}_{1}(t)}{x-1}+x\frac{\mathfrak{A}_{t}(t)}{t(x-t)}\, ,
\end{equation}
or, equivalently, by setting $\widetilde{\mathfrak{A}}_0=\frac{\mathfrak{A}_0-\Itwo }{q-1}\,,  \widetilde{\mathfrak{A}}_1=\frac{\mathfrak{A}_1}{q-1}\,, \widetilde{\mathfrak{A}}_t=\frac{\mathfrak{A}_t}{t(q-1)}$, 
$$
\partial_{q,x} Y(x,t)=\widetilde{\mathfrak{A}}(x,t)Y(x,t),\quad  \textrm{ with }\quad  \widetilde{\mathfrak{A}}(x,t)=\frac{\widetilde{\mathfrak{A}}_0(t)}{x}+\frac{\widetilde{\mathfrak{A}}_1(t)}{x-1}+\frac{\widetilde{\mathfrak{A}}_t(t)}{x-t}
\, .
$$
Here $x$ is again the standard coordinate on $\C$, and $t$ is the standard coordinate on  an open connected  subset $\mathfrak{D}$ of $\C^{*}$.  

\begin{defn}\label{qfuchsian}
We shall say that \eqref{eqq1} is a family of $q$-Fuchsian systems with spectral data $(\boldsymbol{\Theta}\, , \overline{\boldsymbol{\Theta}})$ if the following hold:
\begin{itemize}
\item[$\bullet$] for each $i\in \{0,1,t\}$, $\mathfrak{A}_i\in \mathrm{M}_{2}(\mathcal{O}(\mathfrak{D}))$, 
\item[$\bullet$] for all $t\in \mathfrak{D}$, we have 
$$\mathrm{Spec}(\mathfrak{A}_0(t))=\left\{\Theta_0, \overline{\Theta}_0\right\}, $$
\item[$\bullet$] we have $\det(\mathfrak{A}) =  \frac{\Theta_\infty\overline{\Theta}_\infty\mathfrak{p}(x,t)}{(x-1)^2(x-t)^2}$, where 
 \begin{equation}\label{pfrakeq}
 \mathfrak{p}(x,t)=(x-t\Theta_t)(x-t\overline{\Theta}_t)(x-\Theta_1)(x-\overline{\Theta}_1)\, ,\end{equation}
\item[$\bullet$]  the matrix $\mathfrak{A}_\infty:=\mathfrak{A}_0+\mathfrak{A}_1+\frac{1}{t}\mathfrak{A}_t$ is constant and normalized as follows:
\begin{equation}\label{eq10}
\mathfrak{A}_{\infty}=\begin{pmatrix}\overline{\Theta}_\infty & 0\\0& \Theta_\infty\end{pmatrix}\, .
\end{equation}
\end{itemize}
\end{defn}
\begin{rem} The entries of the matrix
$(x-1)(x-t)\mathfrak{A}(x,t)\in  \mathrm{M}_2(\mathcal{O}(\mathfrak{D})[x])$ have degree two. Then, the determinant of the latter, is a degree four polynomial. The $x^{4}$ coefficient is ${\det(\mathfrak{A}_{\infty})=\Theta_\infty\overline{\Theta}_\infty}$, which is coherent with $\det((x-1)(x-t)\mathfrak{A}(x,t))=\Theta_\infty\overline{\Theta}_\infty\mathfrak{p}(x,t)$.
On the other hand, the constant coefficient is 
 $t\det(\mathfrak{A}_0(t))=t^{2}\Theta_0  \overline{\Theta}_0$.  Note that \eqref{thetarel} is then equivalent to $\Theta_\infty\overline{\Theta}_{\infty}\mathfrak{p}(0,t)=t^{2}\Theta_0  \overline{\Theta}_0$.
The assumption that that two zeros of  $\mathfrak{p}$ are proportional to $t$, and the two others are independent of $t$ will be needed for instance in the proof of Proposition \ref{qLaxPair}. 
\end{rem}
With the normalization  \eqref{eq10}, the $(1,2)$ entry of $(x-1)(x-t)\mathfrak{A}$ is a polynomial of degree at most one in $x$. Let us assume that is has degree one and define a non-zero holomorphic function $\boldsymbol{\lambda}(t) \in \mathcal{O}(\mathfrak{D})$ and a meromorphic function 
  $\y(t)  \in \mathcal{M}(\mathfrak{D})$ by 
\begin{equation}\label{defqyZ} \mathfrak{A}^{(1,2)}(x,t)=\frac{(q-1)\boldsymbol{\lambda}(t)(x-\y(t))}{(x-1)(x-t)}\,  ,\end{equation}
 so that
$\mathfrak{A}^{(1,2)}(\y (t),t)=0.$ 
Assuming moreover that $\y(t)\not \equiv 0,1,t$,  we may define a meromorphic function 
  $\bZZ(t)  \in \mathcal{M}(\mathfrak{D})$ by 
 \begin{equation}\label{eqdefbZZ}\mathfrak{A}^{(1,1)}(\y (t),t) =1+(q-1)\y(t)\bZZ(t).\end{equation}

\begin{rem}
We chose here to  slighlty  modify the notation from \cite{jimbo1996q} because we are mainly interested at the limit when $q\to 1$. 
Towards this goal, it is worth mentioning that our variables satisfy $\widetilde{\mathfrak{A}}^{(1,2)}(x,t)=\frac{\boldsymbol{\lambda}(t)(x-\y(t))}{(x-1)(x-t)}$, and $\widetilde{\mathfrak{A}}^{(1,1)}(\y (t),t)=\bZZ(t)$. More details are given in Section~\ref{sec:confFuchsSys}.
\end{rem}
Analogously to the differential case of  families of $\mathfrak{sl}_2$-Fuchsian  systems, we have the following lemma, which is a slight adaptation of the formulas in \cite[p. 4]{jimbo1996q}.

 \begin{lem}\label{AfrakfromTriple} 
If a family of $q$-Fuchsian systems \eqref{eqq1}  with spectral data $(\boldsymbol{\Theta}\, , \overline{\boldsymbol{\Theta}})$, with $\Theta_\infty \neq \overline{\Theta}_\infty$, gives rise to  $$(\boldsymbol{\lambda},\y ,\bZZ)\in (\mathcal{O}(\mathfrak{D})\setminus \{0\})\times (\mathcal{M}(\mathfrak{D})\setminus \{0,1,\mathrm{id}\})\times\mathcal{M}(\mathfrak{D})$$ as above, and if $1+(q-1)\y(t)\bZZ(t)$ does not vanish identically, then the coefficients of the matrix $\mathfrak{A}$ are necessarily the following functions of $\boldsymbol{\lambda}, \y, \bZZ, x, t$ and $(\boldsymbol{\Theta}\, , \overline{\boldsymbol{\Theta}})$:  \begin{equation}\label{Afrak}\mathfrak{A}(x) =\frac{1}{(x-1)(x-t)}\left(\begin{array}{cc}
\overline{\Theta}_\infty \left((x-\y )(x-\alpha)+z_1\right) &(q-1)\boldsymbol{\lambda}(x-\y )\vspace{.2cm}\\
\Theta_\infty\overline{\Theta}_\infty  \frac{(\gamma x+\delta)}{(q-1)\boldsymbol{\lambda}} &  \Theta_\infty((x-\y )(x-\beta)+z_2)
 \end{array}\right)\end{equation}
with 
\begin{equation}\label{AfrakEntries}
\left\{
\begin{array}{rcl}
z_1&=& \frac{ (\y-1)(\y-t)(1+(q-1)\y\bZZ)}{\overline{\Theta}_\infty}\vspace{.2cm}\\
z_2 &=& \frac{\mathfrak{p}(\y )}{z_1}\vspace{.2cm}\\
\beta &=&-(\alpha-1)-(\y -t)+\frac{\mathfrak{p}(0)-\mathfrak{p}(\y )}{t\y } +\frac{\mathfrak{p}(t)-\mathfrak{p}(\y )}{t(t-1)(\y -t)}-\frac{\mathfrak{p}(1)-\mathfrak{p}(\y )}{(t-1)(\y -1)}  \vspace{.2cm}\\
\gamma &=&z_1+z_2+\alpha\beta+(\alpha-1+\beta)(\y -1)-\frac{\mathfrak{p}(0)-\mathfrak{p}(\y )}{\y } + \frac{\mathfrak{p}(1)-\mathfrak{p}(\y )}{ \y -1}\vspace{.2cm}\\
 \delta &=& \frac{\mathfrak{p}(0)-(\alpha \y +z_1)(\beta \y +z_2)}{\y },
 \end{array}
 \right.
  \end{equation}
  and 
  \begin{equation}\label{alphaval}
\begin{array}{rcl}
\alpha &=& -\frac{t(\Theta_0+\overline{\Theta}_0)-(\overline{\Theta}_\infty z_1+\Theta_\infty z_2)}{(\Theta_\infty-\overline{\Theta}_\infty)\y }\vspace{.2cm}\\&&+\frac{ \Theta_\infty  }{\Theta_\infty-\overline{\Theta}_\infty}\left(1+t-\y +\frac{\mathfrak{p}(0)-\mathfrak{p}(\y )}{t\y }+\frac{\mathfrak{p}(t)-\mathfrak{p}(\y )}{t(t-1)(\y -t)} -\frac{\mathfrak{p}(1)-\mathfrak{p}(\y )}{(t-1)(\y -1)}\right).\end{array}
  \end{equation}
Recall that $\mathfrak{p}$ is defined in \eqref{pfrakeq}. 
  Here we dropped the dependence on $t$ in order to simplify the formulas, \emph{i.e.} we write $\mathfrak{p}(x)$ instead of $\mathfrak{p}(x,t)$ and similarly  $\y =\y (t),\bZZ=\bZZ(t),\boldsymbol{\lambda}=\boldsymbol{\lambda}(t), \mathfrak{A}(x)=\mathfrak{A}(x,t)$.
\end{lem}

\begin{proof} The general form of $\mathfrak{A}$ in \eqref{Afrak}, together with the equation for $z_1$ in \eqref{AfrakEntries}, its precisely what is needed in order for $\mathfrak{A}_\infty$ to be of the required normalized form, and for $\boldsymbol{\lambda}, \y, \bZZ$ to satisfy the equalities  
\eqref{defqyZ} and  \eqref{eqdefbZZ}.  Via evaluation at $x=0,1,t,\y $, the equation $$\det\left( (x-1)(x-t)\mathfrak{A}(x) \right)=\Theta_\infty \overline{\Theta}_\infty\mathfrak{p}(x)$$ with arbitrary $\alpha$ is equivalent to the remaining equations in   \eqref{AfrakEntries}. More precisely, the successive  evaluations at $\y,0,1,t$ give the lines $2,5,4,$ and $3$ of \eqref{AfrakEntries}.  In particular, using \eqref{thetarel}, we have $$\det\mathfrak{A}_0=\Theta_\infty\overline{\Theta}_\infty\Theta_t\overline{\Theta}_t\Theta_1\overline{\Theta}_1=\Theta_0\overline{\Theta}_0\, .$$ Equation \eqref{alphaval} then is equivalent to $\mathrm{trace}(\mathfrak{A}_0)=\Theta_0+\overline{\Theta}_0$.
\end{proof}

 \subsection{Confluence}\label{sec:confFuchsSys}
 The heuristic equality
 $$\lim_{q\to 1}\partial_{q,x}=\partial_x$$
 motivates the following definition. 
 \begin{defn}
 Let $f\in \C(g,x,q)$ such that $\{q=1\}$ is not an irreducible component of the polar divisor of $f$, i.e. $f(g,x,1)$ is a well defined rational function. Then we say that the $q$-difference equation 
  $\partial_{q,x}g=f(g,x,q)$ \emph{discretises} the differential equation $\partial_xg=f(g,x,1)$. 
 \end{defn}
 This definition generalizes in the obvious way to the case of systems of rational $q$-difference equations in several variables. It can also be generalized, in a more subtle way, to the case when the base field is not $\C$, but for example the field of meromorphic functions on some domain. The term \emph{confluence} is used when the inverse phenomenon occurs: when objects associated to a discretized differential equation (most importantly, solutions), yield the corresponding object of the differential equation by some limit process as $q\to 1$. Confluence is widely studied, see  for instance \cite{sauloy2000regular,zhang2002sommation,
di2009q,dreyfus2015confluence,dreyfus2017isomonodromic}. Before confluence can even be adressed, one of course needs to identify the appropriate discretization. The aim of the current section is to do so for families of $\mathfrak{sl}_2$-Fuchsian systems $\partial_xY=A(x,t)Y$ as in Definition~\ref{defi1}. Note that the naive approach of setting $\mathfrak{A}(x,t,q)=\Itwo +(q-1)xA(x,t)$ does in general not yield a $q$-Fuchsian system as in Definition~\ref{qfuchsian}. 
Instead, we will consider $\mathfrak{A}(x,t,q)$ given by a triple of meromorphic functions as in Lemma~\ref{AfrakfromTriple}, but with an  additional parameter $q$, and study when $\widetilde{\mathfrak{A}}=\frac{\mathfrak{A}-\Itwo }{(q-1)x}$ admits a limit as $q\to 1$. 
Here, in a first step, we ignore the difficulty of the coefficients of $\widetilde{\mathfrak{A}}$ being meromorphic functions with respect to $t$, by simply considering $\boldsymbol{\lambda},\y,\bZZ$ as additional variables. 
\begin{prop}\label{propThetaconf}  Let $\boldsymbol{\theta}=(\theta_0,\theta_1,\theta_t,\theta_\infty)\in  \mathbb{C}^4$ with $\theta_\infty \neq 0$. Let $\boldsymbol{\Theta}(q)=(\Theta_0, \Theta_1,\Theta_t,\Theta_\infty)(q)$ and $\overline{\boldsymbol{\Theta}}(q)=(\overline{\Theta}_0, \overline{\Theta}_1,\overline{\Theta}_t,\overline{\Theta}_\infty)(q)$ be two quadrupels of elements of $\C(q)$, \emph{i.e.} rational functions in a complex variable $q$, such that $\Theta_\infty \neq \overline{\Theta}_\infty$ and such that that \eqref{thetarel} holds.
Let $$\mathfrak{A}\in \mathrm{M}_2(\C(x,\boldsymbol{\lambda},\y,\bZZ,t,q))$$
be the $2\times 2$-matrix with coefficients in $\C(x,\boldsymbol{\lambda},\y,\bZZ,t,q)$ (\emph{i.e.} the set of rational functions in six complex variables named $x,\boldsymbol{\lambda}, \y,\bZZ,t,q$) defined by the formulae in Lemma \ref{AfrakfromTriple}.  
Let $A\in \mathfrak{sl}_2(\C(x,\lambda, y,Z,t))$ be defined by formulae in Lemma \ref{AparlamyZ}. Denote $\widetilde{\mathfrak{A}}=\frac{\mathfrak{A}-\Itwo }{(q-1)x}$. The following are equivalent.
\begin{enumerate}
\item The divisor $\{q=1\}$  in $\C^6_{x,\boldsymbol{\lambda},\y,\bZZ,t,q}$ is not an irreducible component of the polar divisor of $\widetilde{\mathfrak{A}}$ and the therefore well-defined rational matrix function $\widetilde{\mathfrak{A}}|_{q=1}$ equals  
$A(x,\boldsymbol{\lambda},\y,\bZZ,t)$. In other words, 
$$\lim_{q\to 1} \widetilde{\mathfrak{A}}(x,\boldsymbol{\lambda},\y,\bZZ,t,q)=A(x,\boldsymbol{\lambda},\y,\bZZ,t)\, .$$
\item Up to permuting the roles of $\Theta_i$ and $\overline{\Theta}_i$ for $i\in \{0,1,t\}$,  the following holds as $q\to 1$:\\
\begin{equation}\label{ThetaconfcondGen}\left\{\begin{array}{rcll}
 \Theta_i(q)&=& 1+(q-1)\frac{\theta_i}{2}+O(q-1)^2&\forall i\in \{0,1,t,\infty\}
\vspace{.2cm}\\
\overline{\Theta}_i(q)&=& 1-(q-1)\frac{\theta_i}{2}+O(q-1)^2&\forall i\in \{0,1,t,\infty\}
\, .
\end{array}\right.\end{equation}
\end{enumerate}
\end{prop}
  \begin{proof}
  From the particular form of $\mathfrak{A}$ it follows that $\mathfrak{A}$ can be decomposed as $\mathfrak{A}_0+\frac{x}{x-1}\mathfrak{A}_1+\frac{x\mathfrak{A}_t}{t(x-t)}$ where $\mathfrak{A}_0, \mathfrak{A}_1,\mathfrak{A}_t$ do not depend on $x$. It follows that  $\widetilde{\mathfrak{A}}$ can be decomposed as $\frac{\widetilde{\mathfrak{A}}_{0}}{x}+\frac{\widetilde{\mathfrak{A}}_{1}}{x-1}+\frac{\widetilde{\mathfrak{A}}_{t}}{x-t}$ where $\widetilde{\mathfrak{A}}_{0},\widetilde{\mathfrak{A}}_{1},\widetilde{\mathfrak{A}}_{t}$ do not depend on $x$. Similarly, we may denote by $A_0,A_1,A_t$ the residues of $A$ with respect to $x=0,1,t$. We denote $\widetilde{\mathfrak{A}}_\infty:=-\lim_{x\to \infty}x\widetilde{\mathfrak{A}}=\frac{\Itwo -\mathfrak{A}_\infty}{q-1}\, $ and $A_\infty=-A_0-A_1-A_t$. Then $(1)$ holds if and only if, for each $i\in \{0,1,t,\infty\}$, we have 
  $$\lim_{q\to 1} \widetilde{\mathfrak{A}}_i(x,\boldsymbol{\lambda},\y,\bZZ,t,q)=A_i(x,\boldsymbol{\lambda},\y,\bZZ,t).$$
  Assume $(1)$ holds. We have $$\widetilde{\mathfrak{A}}_\infty =\begin{pmatrix}\frac{1-\overline{\Theta}_\infty}{q-1  }&0\\0&\frac{1-\Theta_\infty}{q-1}\end{pmatrix}\, . $$
  Since $A_\infty$ is of normal form, we deduce the estimates for $\Theta_\infty$ and $\overline{\Theta}_\infty.$ 
 Since $\mathrm{Spec}(\mathfrak{A}_{0})=\{\Theta_0, \overline{\Theta}_0\}$, we have 
 $$\mathrm{Spec}\left(\widetilde{\mathfrak{A}}_{0}\right)=\left\{\frac{\Theta_0-1}{q-1}, \frac{\overline{\Theta}_0-1}{q-1}\right\}\, .$$
 Since $\mathrm{Spec}(A_0)=\{\theta_0/2,-\theta_0/2\}$, we deduce the estimates for $\Theta_0$ and $\overline{\Theta}_0$ as in the statement (up to interchanging their roles). Recall that $\det(\mathfrak{A}) =  \frac{\Theta_\infty\overline{\Theta}_\infty(x-t\Theta_t)(x-t\overline{\Theta}_t)(x-\Theta_1)(x-\overline{\Theta}_1)}{(x-1)^2(x-t)^2}$.
 We therefore have
 $$\begin{array}{l}
 \mathrm{det}\left(\widetilde{\mathfrak{A}}_1\right)=\frac{\Theta_\infty \overline{\Theta}_\infty (1-t\Theta_t)(1-t\overline{\Theta}_t)(1-\Theta_1)(1-\overline{\Theta}_1)}{(q-1)^2(1-t)^2}\, , \vspace{.2cm} \\
 \mathrm{det}\left(\widetilde{\mathfrak{A}}_t\right)=\frac{\Theta_\infty \overline{\Theta}_\infty t^2 (1-\Theta_t)(1-\overline{\Theta}_t)(t-\Theta_1)(t-\overline{\Theta}_1)}{(q-1)^2(t-1)^2}\, .\end{array}$$
 Since $\mathrm{det}(A_i)=-\theta_i^2/4$ for $i\in \{1,t\}$, we deduce  the estimates for $\Theta_i$ and $\overline{\Theta}_i$ as in the statement (up to interchanging their roles).   Hence $(1)\Rightarrow (2)$. \\
Conversely, the above calculations show that if $(2)$ holds and the limit $\lim_{q\to 1}\widetilde{\mathfrak{A}}$ is a well defined element of $\mathrm{M}_2(\C(x,\boldsymbol{\lambda},\y,\bZZ,t))$, then this limit is of the required form. Moreover, it is straightforward to check (with some more effort), that if $(2)$ holds, then 
the limit is well defined. Here one needs to use the Taylor series expansion of the $\Theta_i$'s and $\overline{\Theta}_i$'s up to order $O(q-1)^3$ and use the relation on the thereby appearing second order terms imposed by the equality $\Theta_0\overline{\Theta}_0=\Theta_\infty\overline{\Theta}_\infty\Theta_t\overline{\Theta}_t\Theta_1\overline{\Theta}_1$. Hence  $(2)\Rightarrow (1)$. \end{proof}

Note that we have arranged the general definition of $(\boldsymbol{\lambda}, \y, \bZZ)$ associated to a matrix $\mathfrak{A}(x,t)$ as in Definition \ref{qfuchsian}, such that for the matrix $\widetilde{\mathfrak{A}}=\frac{\mathfrak{A}-\Itwo }{(q-1)x}$, equations \eqref{defqyZ} and \eqref{eqdefbZZ} may be written as 
$$\begin{array}{llll}\widetilde{\mathfrak{A}}^{(1,2)}(x,t)=\frac{\boldsymbol{\lambda}(t)(x-\y(t))}{(x-1)(x-t)}& , && \widetilde{\mathfrak{A}}^{(1,1)}(\y(t),t)=\bZZ(t)\, .\end{array}$$
This definition is analogous to the general definition in equations \eqref{defyZ} and \eqref{defZ}  of $(\lambda, y, \ZZ)$ associated to a matrix $A(x,t)$ as in Definition~\ref{defi1}. Indeed, recall that these equations were given by 
$$\begin{array}{llll}A^{(1,2)}(x,t)=\frac{ \lambda(t)(x-y(t))}{(x-1)(x-t)}& , &&A^{(1,1)}(y(t),t)=\ZZ(t)\, .\end{array}$$
Therefore, we expect $\partial_{q,x}Y=\mathfrak{A}(x,t,q)Y$ as in Definition \ref{qfuchsian}, but with an additional parameter $q$, to be an appropriate discretization of family of $\mathfrak{sl}_2$-Fuchsian systems $\partial_xY=A(x,t)Y$ as in Definition~\ref{defi1}, where moreover $(\lambda, y,\ZZ)$ are well defined, if the spectral data $(\boldsymbol{\Theta},\overline{\boldsymbol{\Theta}})(q)$ satisfy  \eqref{ThetaconfcondGen} and if, in some convenient sense, we have 
\begin{equation}\label{eq:GrasVersMince}\lim_{q\to 1} (\boldsymbol{\lambda}, \y, \bZZ)(t,q)=(\lambda, y,\ZZ)(t)\, .\end{equation}

Let us now explain what we shall mean by this convenient sense.  Let $\qDom\subset \C\setminus \{0,1\}$ be a connected, not necessarily open, subset, with $1$ in its closure.  Consider a connected open subset $\mathfrak{D}\subset \C^*$ and let $\boldsymbol{f}(t,q)$ be a  function such that for each fixed $q\in \qDom$ sufficiently close to $1$, we have a well-defined  meromorphic function $t\mapsto \boldsymbol{f}(t,q)$ in $\mathcal{M}(\mathfrak{D})$. We say that $f\in  \mathcal{M}(\mathfrak{D})$  is the \emph{limit} of $\boldsymbol{f}$ as $q\to 1$ if for generic values (\emph{i.e.} outside a proper closed analytic subset) of $t\in  \mathfrak{D}$, we have  
$$\lim_{\begin{smallmatrix}q\to 1\\q\in \qDom\end{smallmatrix}}\boldsymbol{f}(t,q)=f(t)\, .
$$
Analogously, if we have a reduced rational function $\boldsymbol{\phi}\left(x,(\boldsymbol{f}_\ell)_{0\leq \ell \leq m+n+1}\right)=\frac{\sum_{k=0}^n \boldsymbol{f}_kx^k }{\sum_{k=0}^{m} \boldsymbol{f}_{k+n+1}x^{k}}$, where each coefficient $\boldsymbol{f}_\ell(t,q)$ is as above and $\lim_{q\to 1}\boldsymbol{f}_\ell(t,q)={f}_\ell(t)\in \mathcal{M}(\mathfrak{D})$, then we say that $ {\phi}(x,t)=\frac{\sum_{k=0}^n  {f}_kx^k }{\sum_{k=0}^{m} {f}_{k+n+1}x^{k}}$ is the limit of $(x,t,q)\mapsto \boldsymbol{\phi}\left(x, (\boldsymbol{f}_\ell(t,q))_{0\leq \ell \leq m+n+1}\right)$ as $q\to 1$. 

  \begin{prop}\label{prop3}
  Let $\mathfrak{D}$ and $\qDom$ be as above. Let $\boldsymbol{\theta}=(\theta_0,\theta_1,\theta_t,\theta_\infty)\in  \mathbb{C}^4$ with $\theta_\infty \neq 0$. Let $\boldsymbol{\Theta}(q)$ and $\overline{\boldsymbol{\Theta}}(q)$ be two quadrupels of elements of $\C(q)$ such that $\Theta_\infty \neq \overline{\Theta}_\infty$ and such that equations \eqref{thetarel} and \eqref{ThetaconfcondGen} hold. Let $( \boldsymbol{\lambda}(t,q), \y(t,q),\bZZ(t,q))$ be a triple of meromorphic functions in neighborhood of  $\mathfrak{D}\times \qDom\subset \C^2$ such that $\boldsymbol{\lambda}\y(\y-1)(\y-t)(1+(q-1)\y\bZZ)$ does not vanish identically on $\mathfrak{D}\times \qDom$ and let $(\lambda(t),y(t),\ZZ(t))$ be a triple of meromorphic functions on $\mathfrak{D}$ such that $\lambda y(y-1)(y-t)$ does not vanish identically. Then for $\widetilde{\mathfrak{A}}=\frac{\mathfrak{A}-\Itwo }{(q-1)x}$ and $A$ as in Lemmas \ref{AfrakfromTriple} and \ref{AparlamyZ} respectively, we have 
  \begin{equation}\label{eqlim0}\lim_{q\to 1}\widetilde{\mathfrak{A}}(x,t,q)=A(x,t)\end{equation} (in the above sense) if and only if  \eqref{eq:GrasVersMince} holds. 
 \end{prop}
 \begin{proof}   Let us first prove the ``if'' part of the statement. 
 One the one hand, if \eqref{eq:GrasVersMince} holds, then 
 \begin{equation}\label{eqlim1}\lim_{q\to 1} \left[ (x,t,q)\mapsto A(x,\boldsymbol{\lambda}(t,q), \y(t,q),\bZZ(t,q))\right]=\left[(x,t)\mapsto A(x,\lambda(t),y(t),\ZZ(t))\right]\, .\end{equation}
 On the other hand, let us consider the rational function 
 $$L :=\frac{\widetilde{\mathfrak{A}}(x,\boldsymbol{\lambda}, \y,\bZZ,t)-A(x,\boldsymbol{\lambda}, \y, \bZZ,t)}{q-1}\in \C(x,\boldsymbol{\lambda}, \y,\bZZ,t,q)\, .$$
By definition of $\mathfrak{A}$, $A$ and by Proposition \ref{propThetaconf}, the affine part of the polar divisor of $L$ is contained in 
$$\{x\in \{0,1,t\}\}\cup \{\boldsymbol{\lambda}=0\}\cup  \{\y\in \{0,1,t\}\}\cup\{1+(q-1)\y\bZZ=0\} \cup\{t\in \{0,1\}\}\cup \{\Theta_\infty(q) = \overline{\Theta}_\infty(q)\}\,$$
and does not contain $\{q=1\}$. 
Hence if \eqref{eq:GrasVersMince} holds, then  \begin{equation}\label{eqlim2}
\lim_{q\to 1} \left[(x,t,q)\mapsto (q-1)L(x,\boldsymbol{\lambda}(t,q), \y(t,q),\bZZ(t,q),t,q)\right]=0\, .
\end{equation}
The addition of the limits \eqref{eqlim1} and \eqref{eqlim2} yields \eqref{eqlim0}. 

Let us now prove the ``only if'' part of the statement. If \eqref{eqlim0} holds, then the limit of the $(1,2)$ coefficient  of $\widetilde{\mathfrak{A}}(x,t,q)$ yields the $(1,2)$ coefficient of $A$ as $q\to 1$. From the explicit formulae, we deduce $$\lim_{q\to 1} \left[(x,t,q)\mapsto \frac{\boldsymbol{\lambda}(t,q)(x-\y(t,q))}{(x-1)(x-t)}\right] =\left[(x,t)\mapsto \frac{ {\lambda(t)}(x-y(t))}{(x-1)(x-t)}\right],$$
and thus $ \lim_{q\to 1} \left(\y(t,q), \boldsymbol{\lambda}(t,q)\right) =\left(y(t), \lambda(t)\right)$. 
By assumption, we have  $$\lim_{q\to 1}\left( \widetilde{\mathfrak{A}}(x,t,q)-A(x,t)\right)=0\, .$$
Since $\lim_{q\to 1}\y(t,q)=y(t)$, we deduce 
\begin{equation}\label{eqlim3} 0=\lim_{q\to 1}\left(\left( \widetilde{\mathfrak{A}}^{(1,1)}(x,t,q)-A^{(1,1)}(x,t)\right)|_{x=\y}\right)=\lim_{q\to 1}\left( \bZZ(t,q)-A^{(1,1)}(x,t)|_{x=\y}\right)\, .\end{equation}
 On the other hand, again from  $\lim_{q\to 1}\y(t,q)=y(t)$, we get 
\begin{equation}\label{eqlim4}0=\lim_{q\to 1}\left(A^{(1,1)}(x,t)|_{x=\y}-A^{(1,1)}(x,t)|_{x=y}\right)=\lim_{q\to 1}\left(A^{(1,1)}(x,t)|_{x=\y}-Z(t)\right)\, .\end{equation}
The addition of the limits \eqref{eqlim3} and \eqref{eqlim4}  yields $\lim_{q\to 1}\bZZ(t,q)=Z(t)$. 
  \end{proof}
   
 According to   the above proposition, under some generic hypotheses, for a convenient choice of spectral value functions $(\boldsymbol{\Theta}, \overline{\boldsymbol{\Theta}})(q)$,  equation \eqref{ThetaconfcondGen} provides a convenient setting for the discretization of families of $\mathfrak{sl}_2$-Fuchsian systems  as in Definition \ref{defi1}.  Here the convenient conditions the spectral value functions must satisfy are the following  (up to permutation of the roles of $\Theta_i$ and $\overline{\Theta}_i$ for $i\in \{0,1,t\}$):
  $$\left\{ \begin{array}{ll}
  \Theta_i(q)= 1+(q-1)\frac{\theta_i}{2}+O(q-1)^2&\forall i\in \{0,1,t,\infty\}
\vspace{.2cm}\\
\overline{\Theta}_i(q)= 1-(q-1)\frac{\theta_i}{2}+O(q-1)^2&\forall i\in \{0,1,t,\infty\}
\vspace{.2cm}\\
\Theta_0\overline{\Theta}_0=\Theta_\infty\overline{\Theta}_\infty\Theta_t\overline{\Theta}_t\Theta_1\overline{\Theta}_1\, \\
\Theta_\infty (q)\neq \overline{\Theta}_\infty (q).
  \end{array}\right.$$
 A simple way to achieve these conditions is to choose the following setting:
    $$\left\{ \begin{array}{ll}
  \Theta_i(q)= 1+(q-1)\frac{\theta_i}{2}+O(q-1)^2&\forall i\in \{0,1,t,\infty\}
\vspace{.2cm}\\
\overline{\Theta}_i= \frac{1}{\Theta_i} &\forall i\in \{0,1,t,\infty\}\\
\Theta_{\infty}(q)\neq 1 .
  \end{array}\right.$$
This convention $\Theta_i\overline{\Theta}_i=1$ can be seen as a  $q$-analogue of the tracefreeness of the differential Fuchsian systems we consider. 
Note that if $\Theta_i(q)$ is analytic in a neighborhood of $1$ and $ \Theta_i(q)= 1+(q-1)\frac{\theta_i}{2}+O(q-1)^2$, then the condition $\overline{\Theta}_i=\frac{1}{\Theta_i}$ implies that $$\Theta_i(q)+\overline{\Theta}_i(q)=2+(q-1)^2\frac{\theta_i^2}{4}+O(q-1)^3\, ,$$
independently of the particular value of the second order term in the Taylor series expansion of~$\Theta_i$. 
 
   \section{Schlesinger equations}
  \subsection{Differential case}
  Let $A_0,A_1,A_t\in \mathfrak{sl}_2(\C)$ and $t\in \C\setminus \{0,1\}$. 
  Consider the Fuchsian system $$\partial_x Y(x) =A(x) Y(x)\quad \textrm{with}\quad A(x)=\frac{A_0}{x}+\frac{A_1}{x-1}+\frac{A_t}{x-t}$$
  over $\mathbb{P}^1$.
An important invariant of such a system is its monodromy, defined as follows.

   In a neighborhood $V$ of a point $x_0\in  \C\setminus \{0,1, t\}$, this system admits a fundamental solution $\mathcal{Y}$, \emph{i.e.} a holomorphic function $\mathcal{Y}:V\to \mathrm{SL}_2(\C)$ satisfying 
  $\mathcal{Y}'=A\mathcal{Y}$, yielding a group homomorphism 
  $$\rho : \left\{\begin{array}{ccc}\pi_1( \mathbb{P}^1 \setminus \{0,1,t,\infty\} , x_0) &\to&  \mathrm{SL}_2 (\C)\\ \gamma & \mapsto & (\mathcal{Y}^{\gamma})^{-1} \cdot \mathcal{Y}\, ,
\end{array}\right.$$
  where $\mathcal{Y}^{\gamma}$ denotes the analytic continuation of $\mathcal{Y}$ along $\gamma$. 
  If $V$ is connected, any other fundamental solution on $V$ is of the form $\mathcal{Y}\cdot M$ for some matrix $M\in \mathrm{SL}_2 (\C)$. 
  Hence the conjugacy class 
  $$[\rho]:=\{ M^{-1}\rho M~|~M\in \mathrm{SL}_2 (\C)\} \subset \mathrm{Hom}(\pi_1( \mathbb{P}^1 \setminus \{0,1,t,\infty\} , x_0) ,  \mathrm{SL}_2 (\C))$$
  does not depend on the choice of the fundamental solution $\mathcal{Y}$ near $x_0$ and is referred as the \emph{monodromy} of the Fuchsian system. Note that the monodromy does not depend on the choice of the base point $x_0$ in the following sense. If $x_1\in  \C\setminus \{0,1, t\}$, we may choose a path $\gamma_1$ from $x_0$ to $x_1$ in  $\C\setminus \{0,1, t\}$, yielding an isomorphism 
  $\tau_{\gamma_1} : \pi_1( \mathbb{P}^1 \setminus \{0,1,t,\infty\} , x_1) \stackrel{\sim}{\to} \pi_1( \mathbb{P}^1 \setminus \{0,1,t,\infty\} , x_0)$. 
  The representation $\rho_1:=\rho\circ \tau_{\gamma_1} $ then is the monodromy representation with respect to the fundamental solution $\mathcal{Y}^{\gamma_1}$, and the conjugacy class 
  $[\rho_1]$ does not depend on the choice of the path $\gamma_1$.
  \begin{rem} In general, it is not possible to compute explicity the monodromy of the Fuchsian system associated to three matrices $A_0,A_1,A_t$ as above. 
  However, if $\gamma_0, \gamma_1, \gamma_t,\gamma_\infty$ denote the standard generators of $\pi_1( \mathbb{P}^1 \setminus \{0,1,t,\infty\} , x_0)$ (each $\gamma_i$ turning clockwise around $i$), then the matrix 
 $\rho(\gamma_i)$ is conjugated to the matrix $\mathrm{exp}(2\sqrt{-1}\pi A_i)$. 
    \end{rem}
    
    Let now 
    \begin{equation}\label{sl2Systembis}\partial_x Y(x,t) =A(x,t) Y(x,t)
    \end{equation} 
    be a family, parametrized by $t\in U$, of $\mathfrak{sl}_2$-Fuchsian systems over $\mathbb{P}^1$ with spectral data $\boldsymbol{\theta}$ as in Definition \ref{defi1}. Here   $U\subset \C\setminus \{0,1\}$ is a connected open subset and $\theta_\infty \neq 0$. Let $t_0\in U$ and let $\Delta \subset U$ be a small disc centered at $t_0$ such that $0,1\notin \Delta$. Let $x_0\in \mathbb{P}^1 \setminus (\{0,1,\infty\}\cup \Delta)$.  Then for any $t\in \Delta$ we have a canonical isomorphism
    $$\pi_1( \mathbb{P}^1 \setminus \{0,1,t,\infty\} , x_0)\simeq \pi_1\left(  \mathbb{P}^1 \setminus (\{0,1,\infty\}\cup \Delta) , x_0\right)\, .$$
    By the Cauchy-Kowalewskaja theorem on linear differential equations with parameters \cite[thm. 9.4.5, p. 348]{hormander1985analysis}, see also \cite[p. 14]{MyBook}, if $\Delta$ is sufficiently small, there exists a neighborhood $V$ of $x_0\in \mathbb{P}^1 \setminus (\{0,1,\infty\}\cup \Delta)$ such that there is a local holomorphic fundamental solution $\mathcal{Y}: V\times \Delta \to \mathrm{SL}_2 (\mathbb{C})$ satisfying $\partial_x \mathcal{Y} =A \mathcal{Y}$. 
    For any $t_1\in \Delta$, this fundamental solution provides a group homomorphism 
     $$\rho_{t_1} : \left\{\begin{array}{ccc}\pi_1(  \mathbb{P}^1 \setminus (\{0,1,\infty\}\cup \Delta)  , x_0) &\to&  \mathrm{SL}_2 (\C)\\ \gamma & \mapsto & (\mathcal{Y}^{\gamma})^{-1}(x_0,t_1) \cdot \mathcal{Y}(x_0,t_1)\, .
\end{array}\right.$$
    This yields a holomorphic family $(\rho_t)_{t\in \Delta}$ of representations, and one may consider the induced family $([\rho_t])_{t\in \Delta}$ of conjugacy classes of representations of $\pi_1(  \mathbb{P}^1 \setminus (\{0,1,\infty\}\cup \Delta)  , x_0)$. 
    
    \begin{defn}\label{defisom} 
    We say that \eqref{sl2Systembis} is \emph{isomonodromic} if one of the two following equivalent properties hold. 
        \begin{enumerate}
    \item Any point $t_0\in U$ admits a neighborhood $\Delta$ such that the associated family $([\rho_t])_{t\in \Delta}$ of monodromies is constant, \emph{i.e.} 
    $$[\rho_t]=[\rho_{t'}]\quad \quad \forall t,t'\in \Delta\, .$$
    \item The system \eqref{sl2Systembis} can locally be \emph{completed into a Lax pair}, \emph{i.e.} any point $t_0\in U$ admits a neighborhood $\Delta$ where there exists 
    $B\in \mathfrak{sl}_2(\mathcal{O}(\Delta)(x))$ such that the following holds:
    \begin{itemize} 
    \item[$\bullet$] the polar locus of $B$ is contained in $$D:=\{x=0\}\cup\{x=1\}\cup\{x=t\}\cup\{x=\infty\}\subset \mathbb{P}^1\times \Delta\, ,$$ 
    \item[$\bullet$] and the system of differential equations 
   $$\left\{\begin{array}{lll}\partial_x  Y(x,t)&=&A(x,t) Y(x,t)\vspace{.2cm} \\\partial_t Y(x,t)& =&B(x,t) Y(x,t)\end{array}\right.\quad \quad$$
over $\mathbb{P}^1\times \Delta$ satisfies  the \emph{Lax equation} $$\partial_t A-\partial_x B=[B,A]\, .$$
\end{itemize}
    \end{enumerate}
    \end{defn}
    
 That these properties are indeed equivalent can deduced from   \cite[Thm. 2]{bolibruch1997isomonodromic}. Moreover, considering the special case of non-resonant spectral data in \cite[Thm. 3]{bolibruch1997isomonodromic}, we have the following lemma. 

\begin{lem}\label{lem2} Assume that the spectral data $\boldsymbol{\theta}$ satisfy the non-resonancy condition \eqref{eq:resonance}. 
Under that condition, if the system \eqref{sl2Systembis} can be {completed into a Lax pair} over $\mathbb{P}^1\times \Delta$ \emph{via} a certain matrix function $B$, then this matrix is of the   form
$$B(x,t)=-\frac{A_t(t)}{x-t}+C(t)\, ,$$
where $C\in \mathfrak{sl}_2(\mathcal{O}(\Delta))$ is a tracefree diagonal matrix function. 
\end{lem}
 
 A key ingredient in the proof of the above classical results are the well-known Schlesinger equations. They yield particular types of ismonodromic deformations. 

\begin{defn}\label{defSchles} 
We say that \eqref{sl2Systembis} is \emph{Schlesinger isomonodromic} if one of the two following equivalent properties hold. 
        \begin{enumerate}
    \item The system \eqref{sl2Systembis} can be completed into a Lax pair \emph{via} the matrix 
    $$B(x,t)=-\frac{A_t(t)}{x-t}\, .$$
    \item The residues $A_i(t)$ of  \eqref{sl2Systembis}  satisfy the \emph{Schlesinger equations}
    \begin{equation}\label{Schles}\left\{\begin{array}{rcc}
A_0'(t)&=& \frac{[A_0(t),A_t(t)]}{0-t} \vspace{.2cm}\\
A_1'(t)&=&\frac{[A_1(t),A_t(t)]}{1-t} \vspace{.2cm}\\
A_t'(t)&=& -\frac{\left[A_{0}(t) ,A_{t}(t)\right]}{0-t}-\frac{\left[A_{1}(t),A_{t}(t)\right]}{1-t} \, .
\end{array}\right.
\end{equation}
 \end{enumerate}
    \end{defn}
In order to see that these properties are equivalent, it suffices to compare the residues at ${x=0,1,t}$ of the Lax equation in the case $B =-\frac{A_t(t)}{x-t}$. 
 Of course Schlesinger isomonodromic families are isomonodromic. As is immediate to check, the converse holds locally up to conjugation:

\begin{lem}\label{lemconj}
 Assume that \eqref{sl2Systembis} is isomonodromic and non-resonant. Let us consider  ${C(t)=\mathrm{diag}(c(t),-c(t))\in \mathfrak{sl}_2(\mathcal{O}(\Delta))}$ be the tracefree diagonal matrix appearing in Lemma \ref{lem2}. Let $t_0\in U$ and let $\Delta$ be a sufficiently small neighborhood of $t_0$ such that there exists a non-vanishing holomorphic function $\mu\in \mathcal{O}(\Delta)$ such that $\mu'(t)=c(t)\mu(t)$.
 Then, the gauge transformation $Y=M\widetilde{Y}$ with $M(t)=\mathrm{diag}\left(\mu(t),\frac{1}{\mu (t)}\right)$
 yields a family $$\partial_x \widetilde{Y}=\widetilde{A} \widetilde{Y}$$ of $\mathfrak{sl}_2$-Fuchsian systems  which is Schlesinger isomonodromic. 
  \end{lem}
 
By definition, the family  \eqref{sl2Systembis} is Schlesinger isomonodromic if and only if the entries of $A$ satisfy a certain list of differential equations. Since these entries may be written in terms of the triple $(\lambda,y,Z)$, see Lemma \ref{AparlamyZ}, we may translate the differential equations in terms of the entries of $A$ into differential equations in terms of the triple $(\lambda,y,Z)$. This will lead to the sixth Painlev\'e equation, according to the following classical results, due to R. Fuchs \cite{fuchs1907lineare}.

\begin{prop}\label{isoPain}
Let $\partial_x Y=A (x,t)Y$ be a family of Fuchsian $\mathfrak{sl}_2$-systems with spectral data $\boldsymbol{\theta}$ as in Definition~\ref{defi1}, giving rise to a triple 
  $(\lambda,y,\ZZ)$ as in \eqref{lamyZcond}. \\ Then $\partial_x Y=A Y$ is Schlesinger isomonodromic if and only if $(y,\ZZ)$ is a solution of \begin{equation}\label{eqHVI}\left\{ \begin{array}{rcl}
y'(t)&=& \frac{y(y-1)(y-t)}{t(t-1)}\left(2\ZZ+\frac{1}{y-t}\right)\vspace{.2cm}
\\
\ZZ'(t)&=&\frac{-3y^2+2(t+1)y-t}{t(t-1)}\ZZ^2-\frac{2y-1}{t(t-1)}\ZZ+\frac{1}{4}\left( \frac{(\theta_\infty-1)^2-1}{t(t-1)}-\frac{\theta_0^2}{(t-1)y^2}-\frac{\theta_t^2 }{(y-t)^2}+\frac{\theta_1^2}{t(y-1)^2}\right)\, ,\end{array}\right.\end{equation}
 and $\lambda$ is a solution of
\begin{equation}\label{eqlam}\frac{\lambda'(t)}{\lambda(t)} = \frac{(\theta_\infty-1)(y(t)-t)}{t(t-1)} \, .\end{equation}
\end{prop}

\begin{proof} Since this is less often detailed in the literature, let us explain how to verify this result by direct computation. 
Since $A_\infty'(t)=0$, we have $A_t'(t)=-A_0'(t)-A_1'(t)$.
Then, the Schlesinger equations  \eqref{Schles} are equivalent to the vanishing of the following two matrices
$$\begin{array}{lcccc}\mathrm{Sch}_0&:=&A_0'(t) +\frac{1}{t}[A_0(t),A_t(t)]  \vspace{.2cm}\\
\mathrm{Sch}_1&:=&A_1'(t)-\frac{1}{1-t}[A_1(t),A_t(t)]\, . \end{array} $$
Substituting the explicit values of the entries the $A_i$'s given by Lemma \ref{AparlamyZ}, one easily checks that $\mathrm{Sch}_0^{(1,2)}\equiv \mathrm{Sch}_1^{(1,2)} \equiv 0$ is equivalent to $$\left\{\begin{array}{lll}
\frac{\lambda'(t)}{\lambda(t)} &=& \frac{(\theta_\infty-1)(y(t)-t)}{t(t-1)} \vspace{.2cm}\\
y'(t)&=& \frac{y(y-1)(y-t)}{t(t-1)}\left(2\ZZ+\frac{1}{y-t}\right).
\end{array}\right.$$
With some (computer assisted) effort, one checks that $\mathrm{Sch}_0^{(1,1)}\equiv \mathrm{Sch}_1^{(1,1)} \equiv 0$ is equivalent to  \eqref{eqHVI}. Finally, one checks that if \eqref{eqlam} and \eqref{eqHVI} hold, then the coefficients $\mathrm{Sch}_i^{(2,1)}$ automatically vanish. Since moreover $\mathrm{Sch}_i^{(2,2)}=-\mathrm{Sch}_i^{(1,1)}$ by tracefreeness, the result follows.
\end{proof}

\begin{cor}\label{corisoiff}
Assume that $\boldsymbol{\theta}$ satisfies the non-resonancy condition \eqref{eq:resonance}. 
Let $\partial_x Y=A Y$ be a family of Fuchsian $\mathfrak{sl}_2$-systems   with spectral data $\boldsymbol{\theta}$ as in Definition~\ref{defi1}, parametrized by $t\in U\subset \C\setminus \{0,1\}$, giving rise to a triple 
  $(\lambda,y,\ZZ)$ as in \eqref{lamyZcond}.
  Define $$U^*:=\{t\in U~|~\lambda(t)\neq 0\, , y(t)\neq \infty\}.$$
  Then the family $\partial_x Y=A Y$ parametrized by $U^*$
 is isomonodromic  if and only if $(y,\ZZ)$ is a solution of \eqref{eqHVI}. \end{cor}
 
\begin{proof} If $\partial_x Y =A Y$  is isomonodromic, then by Lemma \ref{lemconj}, it is locally  conjugated to a Schlesinger isomonodromic family \emph{via} a diagonal gauge transformation. The latter does not  affect $(y,\ZZ)$.  Hence by Proposition~\ref{isoPain}, $(y,\ZZ)$ is a solution of \eqref{eqHVI}.\par 
Conversely, if $(y,\ZZ)$ is a solution of \eqref{eqHVI}, then locally in $U^*$ one may choose a local non-vanishing solution $\widetilde{\lambda}$ of \eqref{eqlam}. One obtains a family of Fuchsian $\mathfrak{sl}_2$-systems $\partial_x \widetilde{Y}=\widetilde{A} \widetilde{Y}$ given by $(\widetilde{\lambda}, y, \ZZ)$ as in Lemma \ref{AparlamyZ}. This family of systems is Schlesinger isomonodromic by Proposition~\ref{isoPain} and conjugated to $\partial_x Y =A Y$ by a diagonal gauge transformation of the form $Y=M\widetilde{Y}$ with $M=\mathrm{diag}( \mu,1/\mu)$ satisfying $\widetilde{\lambda}\mu^2=\lambda$. Hence the initial system is isomonodromic. 
  \end{proof}
 
 \subsection{A discrete analogue}\label{sec1}
 
 We will now define a convenient $q$-analogue of isomonodromy, which will lead to a $q$-analogue of Schlesinger equations for families of Fuchsian linear $q$-difference systems. Analogously to the differential case,  we will define $q$-isomonodromy by the existence of the $q$-analogue of the Lax pair. This definition differs from the approach to $q$-isomonodromy used for example in \cite{jimbo1996q}.  The relation between the two will be explained in Section  \ref{app:qFundsols}.
  Let $q\in \C \setminus \{0,1\} \, .$
From now on, we make the additional asumption that $\mathfrak{D}$ is a connected open subset of  $\C^{*}$, which is stable under multiplication by $q$ and $\frac{1}{q}$.  Let \begin{equation}\label{eqq1bis}
\sigma_{q,x} Y(x,t)=\mathfrak{A}(x,t)Y(x,t),\quad  \textrm{ with }\quad   \mathfrak{A}(x,t)=\mathfrak{A}_{0}(t)+x\frac{\mathfrak{A}_{1}(t)}{x-1}+x\frac{\mathfrak{A}_{t}(t)}{t(x-t)},\, 
\end{equation}
be a family of $q$-Fuchsian systems with spectral data $(\boldsymbol{\Theta}\, , \overline{\boldsymbol{\Theta}})$ as in Definition \ref{qfuchsian}.\par

\begin{defn}\label{def:qisom} We say that the system \eqref{eqq1bis} is $q$-\emph{isomonodromic} if it can be completed into \emph{a $q$-Lax pair}, \emph{i.e.} there exists $\mathfrak{B}\in \mathrm{GL}_{2}(\mathcal{M}(\mathfrak{D})(x))$ such that the system 
$$\left\{\begin{array}{rcr}
\sigma_{q,x}Y&=&\mathfrak{A}(x,t)Y\vspace{.2cm}\\
\sigma_{q,t}Y&=&\mathfrak{B}(x,t)Y\\
\end{array}\right.$$
 satisfies the \emph{q-Lax equation} 
 \begin{equation}\label{eq:qlax}
\mathfrak{A}(x,qt)\mathfrak{B}(x,t)=\mathfrak{B}(qx,t)\mathfrak{A}(x,t).
\end{equation}
\end{defn}

We will now establish the first step towards the $q$-analogue of Lemma \ref{lem2}: under the assumption that $q$ is not a root of unity, if \eqref{eqq1bis} is non-resonant and can be completed into a $q$-Lax pair \emph{via} a matrix function $\mathfrak{B}$, then  $\mathfrak{B}$ has a very particular shape. 

\begin{prop}\label{qLaxPair} Assume that $q\not \in \mathrm{e}^{2i\pi \Q}$ and assume that $(\boldsymbol{\Theta}\, , \overline{\boldsymbol{\Theta}})$ 
satisfies the non-resonancy condition \eqref{eq:qresonance}. Let $\mathfrak{A}\in \mathrm{GL}_{2}(\mathcal{O}(\mathfrak{D})(x))$ be as in \eqref{eqq1bis} and let $\mathfrak{B}\in \mathrm{GL}_{2}(\mathcal{M}(\mathfrak{D})(x))$ such that \eqref{eq:qlax} holds. Then,  there are matrices $\mathfrak{C}(t), B_0(t)\in \mathrm{GL}_{2}(\mathcal{M}(\mathfrak{D}))$ with $\mathfrak{C}(t)$ diagonal,  such that
\begin{equation}\label{exprB}
\mathfrak{B}(x,t)=\mathfrak{C}(t)\frac{(x-qt)(x\Itwo +B_{0}(t))}{(x-qt\Theta_t)(x-qt\overline{\Theta}_t)}\, .
\end{equation} 
Moreover, $\mathrm{Spec}(B_0(t))=\left\{-qt\Theta_t,-qt\overline{\Theta}_t\right\}$.
 \end{prop}

\begin{proof}
  Let us write 
 $$\widehat{\mathfrak{A}}(x,t):=\dfrac{(x-t)(x-1){\mathfrak{A}}(x,t)}{(x-t\Theta_t)(x-t\overline{\Theta}_t)}  \, , \quad  \mathfrak{B}(x,t)=\dfrac{(x-qt)\widehat{\mathfrak{B}}(x,t)}{(x-qt\Theta_t)(x-qt\overline{\Theta}_t)}\, .$$ 
 With this notation, equation  \eqref{eq:qlax}  reads
   \begin{equation}\label{eq:qlaxtilde0} \begin{array}{rcl}q\widehat{\mathfrak{A}}(x,qt) \widehat{\mathfrak{B}}(x,t) &=& \widehat{\mathfrak{B}}(qx,t) \widehat{\mathfrak{A}}(x,t)   \, .\end{array}\end{equation}
From the expression of the determinant of ${\mathfrak{A}}(x,t)$ we have   
   \begin{equation}\label{eq:detAtilde}\det\left(\widehat{\mathfrak{A}}(x,t)\right)=\Theta_\infty \overline{\Theta}_\infty \frac{(x-\Theta_1)(x-\overline{\Theta}_1)}{(x-t\Theta_t)(x-t\overline{\Theta}_t)} \, .
   \end{equation}
The determinant of $\widehat{\mathfrak{B}}(x,t)$ is an element of $\mathcal{M}(\mathfrak{D})(x)$. Let us write $\det\left(\widehat{\mathfrak{B}}(x,t)\right)=c(t)\frac{\prod_{i=1}^{n_1}x-a_{i}(t)}{\prod_{j=1}^{n_2}x-b_{j}(t)}$, with $0\not \equiv c(t)\in \mathcal{M}(\mathfrak{D})$ and $a_{i},b_{j}$ being elements of the algebraic closure $\overline{\mathcal{M}(\mathfrak{D})}$ of ${\mathcal{M}(\mathfrak{D})}$ such that $a_i(t)\not \equiv b_j(t)$ for all $i,j$. 
Applying the determinant in the both sides of \eqref{eq:qlaxtilde0}, we find that 
$$
q^{2}c(t)\frac{(x-\Theta_1)(x-\overline{\Theta}_1)}{(x-qt\Theta_t)(x-qt\overline{\Theta}_t)}
\frac{\prod_{i=1}^{n_1}x-a_{i}(t)}{\prod_{j=1}^{n_2}x-b_{j}(t)}
=c(t)\frac{\prod_{i=1}^{n_1}qx-a_{i}(t)}{\prod_{j=1}^{n_2}qx-b_{j}(t)}\frac{(x-\Theta_1)(x-\overline{\Theta}_1)}{(x-t\Theta_t)(x-t\overline{\Theta}_t)} \,,
$$
which simplifies as follows:
$$
\frac{1}{(x-qt\Theta_t)(x-qt\overline{\Theta}_t)}
\frac{\prod_{i=1}^{n_1}x-a_{i}(t)}{\prod_{j=1}^{n_2}x-b_{j}(t)}
=\frac{1}{(qx-qt\Theta_t)(qx-qt\overline{\Theta}_t)}\frac{\prod_{i=1}^{n_1}qx-a_{i}(t)}{\prod_{j=1}^{n_2}qx-b_{j}(t)} \,.
$$
At $x=\infty$, the left hand side behaves like  $x^{n_{1}-n_{2}-2}$, while the right hand side behaves like $q^{n_{1}-n_{2}-2}x^{n_{1}-n_{2}-2}$. Since $q$ is not a root of unity, we must have $n_{1}=n_{2}+2$. We may thus rewrite the equality as 
$$f(x)=f(qx)\, , \quad \textrm{where} \quad f:x\mapsto \frac{1}{(x-qt\Theta_t)(x-qt\overline{\Theta}_t)}
\frac{\prod_{i=1}^{n_2+2}x-a_{i}(t)}{\prod_{j=1}^{n_2}x-b_{j}(t)} \in \mathcal{M}(\mathfrak{D})(x)\, .$$
It follows easily that $f\in \mathcal{M}(\mathfrak{D})(x)$ is constant in $x$, forcing $n_2=0$ and, up to renumbering the $a_{i}$'s, that  $a_1=qt\Theta_t$ and $a_2=qt\overline{\Theta}_t$. In particular, 
 \begin{equation}\label{eq:detBtilde} \det(\widehat{\mathfrak{B}}(x,t))=c(t) {(x-qt\Theta_t)(x-qt\overline{\Theta}_t)}
 \, .\end{equation}
We may rewrite \eqref{eq:qlaxtilde0} as follows.
  \begin{equation}\label{eq:qlaxtilde} \begin{array}{rcl}q \widehat{\mathfrak{B}}(x,t) &=& \widehat{\mathfrak{A}}(x,qt)^{-1}\widehat{\mathfrak{B}}(qx,t) \widehat{\mathfrak{A}}(x,t)
  \, .\end{array}
  \end{equation}
Let us now show that $\widehat{\mathfrak{B}}(x,t)$ is of the form $x\cdot \widehat{B}_\infty(t)+\widehat{B}_0(t)$, \emph{i.e.} let us show that the function $x\mapsto \widehat{\mathfrak{B}}(x,t)$ has only one possible pole, at $x=\infty$, and that this is at most a simple pole.   

\begin{itemize}
\item[$\bullet$] We have $\widehat{\mathfrak{B}}(x,t)= x^kR_\infty(t)+O(x^{k-1})$ as $x\to \infty $ for a certain $k\in \Z$ and a certain non-zero $R_\infty\in \mathrm{M}_2(\mathcal{M}(\mathfrak{D}))$. Setting $\widehat{\mathfrak{A}}_\infty(t):=\widehat{\mathfrak{A}}(\infty,t)$, we have $\widehat{\mathfrak{A}}_\infty(t)={\mathfrak{A}}_\infty (t)$ which is invertible. Moreover, $\mathfrak{A}_\infty(t)= \mathrm{diag}(\overline{\Theta}_\infty,\Theta_\infty)$ is independent of $t$.
 From  equation  \eqref{eq:qlaxtilde}  we get
  $$R_\infty(t) =q^{k-1}\mathrm{diag}(\overline{\Theta}_\infty,\Theta_\infty)^{-1} R_\infty(t)\mathrm{diag}(\overline{\Theta}_\infty,\Theta_\infty). $$ 
Taking the determinant in both sides yield 
$\det(R_\infty(t)) =\det(R_\infty(t))\left(q^{k-1}\right)^{2}$.
By \eqref{eq:detBtilde} we find $\det(R_\infty(t))= c(t)\neq 0$, so that $R_\infty(t)$ must actually be invertible and $1=q^{2k-2}$.
Since $q$ is not a root of unity, we have $k=1$. Hence $\widehat{\mathfrak{B}}(x,t)$ has a simple pole at $x=\infty$. Furthermore, since $\Theta_\infty \neq \overline{\Theta}_\infty$, we deduce that $R_\infty(t)$  is diagonal.
  \item[$\bullet$]   Let $\alpha(t) \in \mathcal{M}(\mathfrak{D})$ such that $\alpha \not \in  \{0,\Theta_1 , \overline{\Theta}_1 ,  t \Theta_t ,  t \overline{\Theta}_t\}\cdot q^{\Z_{> 0}}$. Assume for a contradiction that $x= \alpha$ is a pole of $\widehat{\mathfrak{B}}(x,t)$. Then $x=\frac{1}{q} \alpha$ is a pole of $\widehat{\mathfrak{B}}(qx,t)$. From the particular form of $\widehat{\mathfrak{A}}$ and \eqref{eq:detAtilde} we know that both 
 $\widehat{\mathfrak{A}}(x,qt)^{-1}$ and $\widehat{\mathfrak{A}}(x,t)$ are finite and invertible at $x=\frac{1}{q} \alpha$, \emph{i.e.} $\widehat{\mathfrak{A}}(\alpha/q,t),\widehat{\mathfrak{A}}(\alpha/q,qt)^{-1}\in\mathrm{GL}_{2}(\mathcal{M}(\mathfrak{D}))$. 
 Indeed,  the only possible poles of $\widehat{\mathfrak{A}}(x,t) $ are $x=t\Theta_t$ and $x=t\overline{\Theta}_t$ and the only possible poles of $\widehat{\mathfrak{A}}(x,qt)^{-1} $ are $x=\Theta_1$ and $x=\overline{\Theta}_1$. Hence by \eqref{eq:qlaxtilde}, $\widehat{\mathfrak{B}}(x,t)$ has a pole at $x=\frac{1}{q} \alpha$ as well. By induction, $\widehat{\mathfrak{B}}(x,t)$ has a pole at $x=\frac{1}{q^n} \alpha$ for each $n\in \N$. 
  Yet by assumption, $\widehat{\mathfrak{B}}(x,t)$ is a rational function of $x$ and can therefore only have finitely many poles. 
   \item[$\bullet$]   Let $\alpha(t) \in  \{\Theta_1 , \overline{\Theta}_1 ,  t \Theta_t ,  t \overline{\Theta}_t\}\cdot q^{\Z_{> 0}}$.   Assume for a contradiction that $x= \alpha$ is a pole of $\widehat{\mathfrak{B}}(x,t)$. Then $x=q\alpha$ is a pole of $\widehat{\mathfrak{B}}(x/q,t)$.   We may  rewrite equations \eqref{eq:qlaxtilde} and \eqref{eq:detAtilde}  as 
  $$\det\left(\widehat{\mathfrak{A}}(x/q,t)\right)=\frac{(x-q\Theta_1)(x-q\overline{\Theta}_1)}{(x-qt\Theta_t)(x-qt\overline{\Theta}_t)} \, ,$$
 $$\begin{array}{rcl} \widehat{\mathfrak{B}}(x,t) &=&q\widehat{\mathfrak{A}}(x/q,qt) \widehat{\mathfrak{B}}(x/q,t)\widehat{\mathfrak{A}}(x/q,t)^{-1}\, .\end{array}$$
 The same reasoning as before shows that $\widehat{\mathfrak{B}}(x,t)$ then has a pole at  $x=q\alpha$ as well. Again this leads by induction to an infinite number of poles, and therefore, a contradiction. 
   \item[$\bullet$]  Finally, we treat the case $\alpha=0$. We have $\widehat{\mathfrak{B}}(x,t)= x^kR_0(t)+O(x^{k+1})$ as $x\to 0 $ for a certain $k\in \Z$ and a certain non-zero $R_0 \in \mathrm{M}_2(\mathcal{M}(\mathfrak{D}))$. Setting $\widehat{\mathfrak{A}}_0(t):=\widehat{\mathfrak{A}}(0,t)$, we have $\widehat{\mathfrak{A}}_0(t)=\frac{1}{t\Theta_t\overline{\Theta}_t}{\mathfrak{A}}_0(t)$ which is invertible.
 From  equation  \eqref{eq:qlaxtilde}  we get
 $$R_0(t) =q^{k-1} \widehat{\mathfrak{A}}_0(qt)^{-1}R_0(t)\widehat{\mathfrak{A}}_0(t)=q^{k} \mathfrak{A}_0^{-1}(qt)R_0(t)\mathfrak{A}_0(t). $$
Equation \eqref{eq:detBtilde}  yields $\det(R_0(t))=q^2t^2c(t)\Theta_t \overline{\Theta}_t\neq 0$, so that $R_0$ is actually invertible.
Since $\mathrm{Spec}(\mathfrak{A}_0(t))=\{\Theta_0,\overline{\Theta}_0\}$, we have $\det(\mathfrak{A}_0(t))=\det(\mathfrak{A}_0(qt))$ and we may take the determinant in both sides of $R_{0}(t)=q^{k} \mathfrak{A}_0^{-1}(qt)R_0(t)\mathfrak{A}_0(t)$ and find $q^{2k}=1$. Since $q$ is not a root of unity, we deduce $k=0$.
 \end{itemize}

  We have now proven that $\widehat{\mathfrak{B}}(x,t)$ has a simple pole at $x=\infty$ and is finite and non-zero everywhere else. More precisely, we have proven that $$\widehat{\mathfrak{B}}(x,t)=\mathfrak{C}(t)(x\Itwo +B_0(t))$$ holds for the diagonal matrix $\mathfrak{C}(t):=R_\infty(t)\in \mathrm{GL}_2(\mathcal{M}(\mathfrak{D}))$ and the matrix $B_0:=\mathfrak{C}(t)^{-1}R_0(t)\in \mathrm{GL}_2(\mathcal{M}(\mathfrak{D}))$. 
   Moreover, from 
   \eqref{eq:detBtilde} we get ${\det(x\Itwo +B_0(t))=(x-qt\Theta_t)(x-qt\overline{\Theta}_t)}$, yielding the sought expression for the eigenvalues of $B_0$.\end{proof}

 We will see in the proof of the following proposition that  the matrix $B_0$, as well as the matrix $\mathfrak{C}$ up to a scalar multiple, are uniquely defined by \eqref{eq:qlax}. However, once these matrices are obtained, \eqref{eq:qlax} still imposes strong conditions on the matrix $\mathfrak{A}$, which will yield the $q$-Painlev\'e VI equation.

   \begin{prop}\label{qisomchar} Assume that $q\not \in \mathrm{e}^{2i\pi \Q}$. Let $(\boldsymbol{\Theta}, \overline{\boldsymbol{\Theta}})\in (\C^*)^4\times (\C^*)^4$ such that $\Theta_\infty \neq \overline{\Theta}_\infty$ and such that the relation \eqref{thetarel} as well as the non-resonancy condition \eqref{eq:qresonance} hold. Let $(\boldsymbol{\lambda}, \y, \bZZ)\in (\mathcal{M}(\mathfrak{D}))^3$. Denote 
   $$X:=\frac{(\y-1)(\y-t)(1+(q-1)\y\bZZ)}{\Theta_\infty \overline{\Theta}_\infty (\y-\Theta_1)(\y-\overline{\Theta}_1)}\in \mathcal{M}(\mathfrak{D})$$
   and assume that 
   $\y,(\y-1),(\y-t),\boldsymbol{\lambda}, X,\, \Theta_\infty X-1\, , q\overline{\Theta}_\infty X-1\, , \Theta_\infty X-\frac{ (\y-t\Theta_t)(\y-t\overline{\Theta}_t)}{ (\y-\Theta_1)(\y-\overline{\Theta}_1)}$ are all well defined meromorphic functions on $\mathfrak{D}$ and are each not identically zero. 
   Let $\mathfrak{A}(x,t)=\mathfrak{A}_0(t)+\frac{x}{x-1}\mathfrak{A}_1(t)+\frac{x}{t(x-t)}\mathfrak{A}_t(t)$ be defined by $(\boldsymbol{\lambda}, \y, \bZZ)$ as in Lemma \ref{AfrakfromTriple} and assume that the coefficients of $\mathfrak{A}_0, \mathfrak{A}_1, \mathfrak{A}_t$ are holomorphic on a domain $\mathfrak{D}^*\subset \mathfrak{D}$ stable under multiplication by $q$ and $1/q$.  Let us denote     
   $\mathfrak{A}_\infty=\mathfrak{A}_0+\mathfrak{A}_1+\frac{1}{t}\mathfrak{A}_t
   =\mathrm{diag}(\overline{\Theta}_\infty, \Theta_\infty)$. 
The following are equivalent. 
   \begin{enumerate}
   \item The family of $q$-Fuchsian systems $\sigma_{q,x} Y=\mathfrak{A}Y$, parametrized by $\mathfrak{D}^*$, is $q$-isomonodromic.
      \item Denoting         $$\begin{array}{l}
   B_0:=-qt\left(\mathfrak{A}_\infty+\frac{t-1}{(t\Theta_t-1)(t\overline{\Theta}_t-1)}\mathfrak{A}_1\right)\left(\frac{1}{\Theta_t \overline{\Theta}_t}\mathfrak{A}_0+\frac{t(t-1)}{(t\Theta_t-1)(t\overline{\Theta}_t-1)} \mathfrak{A}_1\right)^{-1}\vspace{.2cm}\\
   \ \mathfrak{C}:=\mathrm{diag}(c,1)\, , \quad \textrm{with} \quad c=\frac{\sigma_{q,t}\boldsymbol{\lambda}}{\boldsymbol{\lambda}} \frac{\Theta_\infty X-1}{q\overline{\Theta}_\infty X-1}\, , 
   \end{array}$$
   the following equations hold:
   $$\left\{ \begin{array}{rcl}
   \sigma_{q,t}\mathfrak{A}_0&=&\mathfrak{C}B_0\mathfrak{A}_0 B_0^{-1}\mathfrak{C}^{-1}\vspace{.2cm}\\
    \sigma_{q,t}\mathfrak{A}_1&=&\displaystyle \frac{t-1}{qt-1}\frac{(qt\Theta_t-1)(qt\overline{\Theta}_t-1)}{q(t\Theta_t-1)(t\overline{\Theta}_t-1)}\cdot \mathfrak{C}(q\Itwo +B_0) \, \mathfrak{A}_1\left(\Itwo +B_0 \right)^{-1}\mathfrak{C}^{-1}\vspace{.2cm}\\
     \sigma_{q,t}\mathfrak{A}_t&=&\displaystyle-\mathfrak{C}B_0 \mathfrak{A}_0 \left( \frac{1}{\Theta_t\overline{\Theta}_t}\Itwo +qtB_0^{-1}\right)\mathfrak{C}^{-1}\vspace{.2cm}\\&&\displaystyle-\frac{ t(t-1)\mathfrak{C}(q\Itwo +B_0)\mathfrak{A}_1}{(t\Theta_t-1)(t\overline{\Theta}_t-1)}\left(\Itwo  + \frac{(qt\Theta_t-1)(qt\overline{\Theta}_t-1)}{qt-1 }(\Itwo +B_0)^{-1} \right) \mathfrak{C}^{-1} \, .
   \end{array} \right.$$
   \item The pair $(\y, \bZZ)$ is a solution of the following system of $q$-difference equations: 
\begin{equation} \label{qPVIyZ} \left\{ \begin{array}{rcl}
   \sigma_{q,t}\y &=& \displaystyle \frac{\Theta_1\overline{\Theta}_1}{\y} \cdot \frac{\left(X-t\frac{\Theta_t\overline{\Theta}_t}{\Theta_0}\right)\left(X-t\frac{\Theta_t\overline{\Theta}_t}{ \overline{\Theta}_0}\right)}{\left( X-\frac{1}{\Theta_\infty}\right)\left( X-\frac{1}{q\overline{\Theta}_\infty}\right)} \vspace{.2cm}\\
    \sigma_{q,t}\bZZ &=&  \displaystyle\frac{1}{ (q-1) \sigma_{q,t}\y}\left(\frac{(\sigma_{q,t}\y-qt\Theta_t)(\sigma_{q,t}\y-qt\overline{\Theta}_t)}{q( \sigma_{q,t}\y-1)( \sigma_{q,t}\y-qt)X}-1\right)
    \, .
   \end{array} \right.\end{equation}

      \end{enumerate}
      \end{prop}
   \begin{proof}
     Let us first show that $(1)\Leftrightarrow (2)$. 
Since $q$ is not a root of unity, by Proposition \ref{qLaxPair} and the non-resonancy assumption, $(1)$ is equivalent to the existence  of $B_0, \mathfrak{C}\in \mathrm{GL}_2(\mathcal{M}(\mathfrak{D}))$ with $\mathfrak{C}$ diagonal, such that for  $\mathfrak{B}$ given by \eqref{exprB},  
   we have $\left(\sigma_{q,t}\mathfrak{A}(x,t)\right)\mathfrak{B}(x,t)=\left(\sigma_{q,x}\mathfrak{B}(x,t)\right)\mathfrak{A}(x,t)$. We may rewrite $$\begin{array}{l} \frac{1}{x}\left(\sigma_{q,t}\mathfrak{A}(x,t)\right)\mathfrak{B}(x,t)-\frac{1}{x}\left(\sigma_{q,x}\mathfrak{B}(x,t)\right)\mathfrak{A}(x,t)=
   \vspace{.2cm}\\
   =-\frac{1}{qt\Theta_t\overline{\Theta}_t} \frac{R_0(t)}{x}-\frac{qt-1}{(qt\Theta_t-1)(qt\overline{\Theta}_t-1)} \frac{R_1(t)}{x-1}  +\frac{(x-qt(\Theta_t+\overline{\Theta}_t))R_2(t)+qtR_3(t)}{(x-qt{\Theta}_t)(x-qt\overline{\Theta}_t)} -\mathfrak{C}\frac{ (x-t(\Theta_t+\overline{\Theta}_t))R_4-tR_5}{(x-t\Theta_t)(x-t\overline{\Theta}_t)}\, , \end{array}$$
   where 
   $$\begin{array}{rcl}R_0&:=& \left(\sigma_{q,t}\mathfrak{A}_0\right)\mathfrak{C}B_0-\mathfrak{C}B_0\mathfrak{A}_0\, , \vspace{.2cm}\\
   R_1&:=& \left(\sigma_{q,t}\mathfrak{A}_1\right)\mathfrak{C}(\Itwo +B_0)-\frac{q(t-1)}{qt-1}\frac{(qt\Theta_t-1)(qt\overline{\Theta}_t-1)}{q^2(t\Theta_t-1)(t\overline{\Theta}_t-1)}\mathfrak{C}(q\Itwo +B_0)\mathfrak{A}_1
       \vspace{.2cm}\\
   R_2&:=&\frac{1}{qt\Theta_t\overline{\Theta}_t}\left(\sigma_{q,t}\mathfrak{A}_0\right)\mathfrak{C}B_0 +\frac{qt-1}{(qt\Theta_t-1)(qt\overline{\Theta}_t-1)}\left(\sigma_{q,t}\mathfrak{A}_1\right)\mathfrak{C}(\Itwo +B_0) +\left(\sigma_{q,t}\mathfrak{A}_\infty \right)\mathfrak{C}  
        \vspace{.2cm}\\
   R_3&:=&\left(\sigma_{q,t}\mathfrak{A}_0\right)\mathfrak{C}+\frac{qt-1}{qt}\left(\sigma_{q,t}\mathfrak{A}_1\right)\mathfrak{C}-\frac{(qt-1)\left(\sigma_{q,t}\mathfrak{A}_1\right)\mathfrak{C}(\Itwo +B_0)}{qt(qt\Theta_t-1)(qt\overline{\Theta}_t-1)}-\frac{1}{qt}\left(\sigma_{q,t}\mathfrak{A}_\infty \right)\mathfrak{C} (qt\left(\Theta_t+\overline{\Theta}_t)\Itwo +B_0\right)     \vspace{.2cm}\\
   R_4&:=& \frac{1}{qt\Theta_t\overline{\Theta}_t}B_0\mathfrak{A}_0+\frac{t-1}{q(t\Theta_t-1)(t\overline{\Theta}_t-1)}(q\Itwo +B_0)\mathfrak{A}_1+\mathfrak{A}_\infty
    \vspace{.2cm}\\
   R_5&:=& \mathfrak{A}_0+\frac{t-1}{t}\mathfrak{A}_1-\frac{(t-1)}{qt(t\Theta_t-1)(t\overline{\Theta}_t-1)}\left(q\Itwo +B_0\right)\mathfrak{A}_1-\frac{1}{qt}\left(qt(\Theta_t+\overline{\Theta}_t)\Itwo  +B_0\right) \mathfrak{A}_\infty\, .\end{array}
   $$
   Hence $(1)$ is equivalent to the vanishing  $R_i=0$ for all $i\in \{0,1,2,3,4,5\}$  for some $B_0, \mathfrak{C}\in \mathrm{GL}_2(\mathcal{M}(\mathfrak{D}))$ with $\mathfrak{C}$ diagonal. Note that by Proposition \ref{qLaxPair}, the matrix $B_0$ must have eigenvalues  $-qt\Theta_t, -qt\overline{\Theta}_t$, so that $\Itwo +B_0$ is invertible. 
       If $R_0=R_1=R_4=R_5=0$, then $R_2$ and $R_3$ are both equivalent to $\sigma_{q,t}\mathfrak{A}_\infty = \mathfrak{C}\mathfrak{A}_\infty \mathfrak{C}^{-1}$. We may therefore omit $R_3$ in the following. 
   For $i\in \{0,1\}$, the vanishing of $R_i$ is equivalent to the equation for $ \sigma_{q,t}\mathfrak{A}_i$ as in the statement, with general $B_0,\mathfrak{C}$.  Substituting these equations into $R_2=0$ yields 
   the equation for $\sigma_{q,t}\mathfrak{A}_t$ as in the statement, again with general $B_0,\mathfrak{C}$.  However, 
   $R_4=0$ is equivalent to $B_0$ being as in the statement. Note that this matrix $B_0$ is well-defined and invertible under the assumptions. As one can check by direct computation, this $B_0$ solves $R_5=0$. 
   Hence  $(1)$ is equivalent to the existence  of a diagonal matrix $\mathfrak{C}\in \mathrm{GL}_2(\mathcal{M}(\mathfrak{D}))$ such that the equations in $(2)$ hold for $B_0$ as in the statement. 
   If $\mathfrak{C}$ is such a convenient matrix, then for any $f\in \mathcal{M}(\mathfrak{D})$ non-vanishing, $f\mathfrak{C}$ is also convenient. Hence we may require that $\mathfrak{C}$ is of the form $\mathrm{diag}(c,1)$. 
   Since $(q-1)\boldsymbol{\lambda} = (1-t)\mathfrak{A}_1^{(1,2)}-t\mathfrak{A}_0^{(1,2)}$, we must have
   $$(q-1)\sigma_{q,t} \boldsymbol{\lambda} = (1-qt)\sigma_{q,t} \mathfrak{A}_1^{(1,2)}-qt\sigma_{q,t} \mathfrak{A}_0^{(1,2)}\, .$$
    With the equations in $(2)$ for the $ \sigma_{q,t}\mathfrak{A}_i$'s, this is equivalent to $c$ being as in the statement. 
    We conclude that $(1)\Leftrightarrow (2)$. 
    
    Let us now show that $(2)\Rightarrow (3)$. From $\y=-\frac{t}{(q-1)\boldsymbol{\lambda}} \mathfrak{A}_0^{(1,2)}$, we obtain 
    $$\sigma_{q,t}\y =-\frac{qt}{(q-1)\sigma_{q,t} \boldsymbol{\lambda}} \sigma_{q,t} \mathfrak{A}_0^{(1,2)}= -qt \frac{\sigma_{q,t} \mathfrak{A}_0^{(1,2)}}{(1-qt)\sigma_{q,t} \mathfrak{A}_1^{(1,2)}-qt\sigma_{q,t} \mathfrak{A}_0^{(1,2)}}\, .$$
Substituting the values of $\mathfrak{A}_0^{(1,2)}$ and $\mathfrak{A}_1^{(1,2)}$ from $(2)$ then yields
\begin{equation}\label{sigyval}  \sigma_{q,t}\y = \frac{q}{\y} \cdot \frac{\Theta_1\overline{\Theta}_1\Theta_\infty \overline{\Theta}_\infty X^2-t(\Theta_0+\overline{\Theta}_0)X+t^2\Theta_t\overline{\Theta}_t}{(\Theta_\infty X-1)(q\overline{\Theta}_\infty X-1)}\, .\end{equation}
For the purpose of factorization, we use the equality $\Theta_0\overline{\Theta}_0=\Theta_t\overline{\Theta}_t\Theta_1\overline{\Theta}_1\Theta_\infty\overline{\Theta}_\infty$. This yields the expression for $\sigma_{q,t}\y$ in the statement. 
Similarly, from $$1+(q-1)\y \bZZ=\mathfrak{A}_0^{(1,1)}+\frac{\y}{\y-1} \mathfrak{A}_1^{(1,1)}+\frac{\y}{\y-t}\left( \mathfrak{A}_\infty^{(1,1)}-\mathfrak{A}_0^{(1,1)}-\mathfrak{A}_1^{(1,1)}\right),$$
 we obtain 
$$\sigma_{q,t}\left((\y-1)(\y-t)( 1+(q-1)\y \bZZ)\right)=(\sigma_{q,t}\y-1)\left(\overline{\Theta}_\infty \sigma_{q,t}\y-qt\sigma_{q,t}\mathfrak{A}_0^{(1,1)}\right)  + (1-qt)\sigma_{q,t}\y\cdot \sigma_{q,t} \mathfrak{A}_1^{(1,1)}\, .   $$
Substituting the value of $\sigma_{q,t}\y$ from \eqref{sigyval}, the values of $\mathfrak{A}_0^{(1,1)}$ and $\mathfrak{A}_1^{(1,1)}$ from $(2)$, as well as $\bZZ=\frac{1}{(q-1)\y}\left(\frac{\Theta_\infty \overline{\Theta}_\infty (\y-\Theta_1)(\y-\overline{\Theta}_1)}{(\y-1)(\y-t)}X-1\right)$, the right hand side simplifies to 
$$\frac{(\sigma_{q,t}\y-qt\Theta_t)(\sigma_{q,t}\y-qt\overline{\Theta}_t)}{qX}\, .$$
Therefore, $\sigma_{q,t}\bZZ$ is as in the statement. 

 Let us now show that $(3)\Rightarrow (2)$. Note that for each  $i\in \{0,1,t\}$, the matrices $\mathfrak{A}_i$ may be expressed as functions of $\boldsymbol{\lambda}, X$ and $\y$. If $(3)$ holds, then the matrices 
$\mathfrak{C}^{-1}\sigma_{q,t}\mathfrak{A}_i \mathfrak{C}$, with $\mathfrak{C}$ as in $(2)$, can also be expressed as functions of $\boldsymbol{\lambda}, X$ and $\y$. It its straightforward to check (with computer assistance) that the equations in $(2)$ then are satisfied. 
   \end{proof}
  
  In analogy with the differential case, we give a name to the particular case when a family can be completed into a Lax pair \emph{via} a matrix $\mathfrak{B}$ as in \eqref{exprB} with $\mathfrak{C}(t)=\Itwo $:
  
 \begin{defn}\label{def:qSchlesisom} We say that the family $\sigma_{q,x}Y=\mathfrak{A}(x,t)Y$ of Fuchsian systems is $q$-\emph{Schlesinger isomonodromic} if it can be completed into a $q$-Lax pair \emph{via} a matrix $\mathfrak{B} \in \mathrm{GL}_{2}(\mathcal{M}(\mathfrak{D})(x))$ of the form  
 \begin{equation}\label{matrixBinSchles} \mathfrak{B}(x,t)=\dfrac{(x-qt)(x\Itwo +B_0(t))}{(x-qt\Theta_t)(x-qt\overline{\Theta}_t)}\, .\end{equation} \end{defn}

 Let us now say a few words about whether, analogously to the differential setting, a family of $q$-Fuchsian systems which is $q$-isomonodromic can be made $q$-Schlesinger isomonodromic \emph{via} a gauge transformation. Let $\mathfrak{A}\in \mathrm{GL}_{2}(\mathcal{O}(\mathfrak{D})(x))$ be as in \eqref{eqq1bis} and assume that the family of $q$-Fuchsian systems $\partial_{q,x}Y=\mathfrak{A}(x,t)Y$ can be completed into a Lax pair \emph{via} a matrix $\mathfrak{B}\in \mathrm{GL}_{2}(\mathcal{M}(\mathfrak{D})(x))$ of the form \eqref{exprB}, with $\mathfrak{C}\in \mathrm{GL}_{2}(\mathcal{M}(\mathfrak{D}))$ diagonal. Assume there exists   ${\mathfrak{M}\in \mathrm{GL}_2(\mathcal{M}(\mathfrak{D}))}$ which is diagonal and solves the $q$-difference equation $$\sigma_{q,t}\mathfrak{M}(t)=\mathfrak{C}(t)\mathfrak{M}(t)\, .$$
 Since $\mathfrak{M}$ does not depend on $x$, performing the gauge transformation $Y=\mathfrak{M}\widehat{Y}$ yields the family  $\sigma_{q,x}\widehat{Y}=\widehat{\mathfrak{A}}(x,t)\widehat{Y}$ given by 
 $$\widehat{\mathfrak{A}}(x,t)=(\sigma_{q,x}\mathfrak{M}(t))^{-1}\mathfrak{A}(x,t)\mathfrak{M}(t)=\mathfrak{M}(t)^{-1}\mathfrak{A}(x,t)\mathfrak{M}(t)\, .$$
 Since $\mathfrak{M}$ is diagonal, up to shrinking $\mathfrak{D}$ to the domain of holomorphy of the coefficients of $\widehat{\mathfrak{A}}\in \mathrm{GL}(\mathcal{M}(\mathfrak{D})(x)),$ this new family is still a family of $q$-Fuchsian systems in the sense of Definition \ref{qfuchsian}. Moreover, this new family can be completed into a Lax pair \emph{via} the matrix 
 $$\widehat{\mathfrak{B}}:=\mathfrak{M}^{-1}\mathfrak{C}^{-1}\mathfrak{B}\mathfrak{M} \in \mathrm{GL}_{2}(\mathcal{M}(\mathfrak{D})(x))\, .$$
 Indeed, from the $q$-Lax equation for the initial family, we get
 $$\begin{array}{rcl}\widehat{\mathfrak{A}}(x,qt)\widehat{\mathfrak{B}}(x,t)&=&\mathfrak{M}(qt)^{-1}\mathfrak{A}(x,qt)\mathfrak{M}(qt)\mathfrak{M}^{-1}(t)\mathfrak{C}^{-1}(t)\mathfrak{B}(x,t)\mathfrak{M} (t)\vspace{.2cm}\\&=&\mathfrak{M}(t)^{-1}\mathfrak{C}(t)^{-1}\mathfrak{A}(x,qt) \mathfrak{B}(x,t)\mathfrak{M} (t)\vspace{.2cm}\\&=&
 \mathfrak{M}(t)^{-1}\mathfrak{C}(t)^{-1}\mathfrak{B}(qx,t) \mathfrak{A}(x,t)\mathfrak{M} (t)\vspace{.2cm}\\&=&\widehat{\mathfrak{B}}(qx,t)\widehat{\mathfrak{A}}(x,t)\, .\end{array}$$
 Note that $\widehat{\mathfrak{B}}$ is given by 
 $$\widehat{\mathfrak{B}}(x,t)= \frac{(x-qt)(x\Itwo +\widehat{B}_0(t))}{(x-qt\Theta_t)(x-qt\overline{\Theta}_t)}\quad \quad \textrm{with} \quad \quad \widehat{B}_0=\mathfrak{M}^{-1}B_{0}\mathfrak{M}\, .$$
 In other words, the conjugated family $\sigma_{q,x}\widehat{Y}=\widehat{\mathfrak{A}}(x,t)\widehat{Y}$ is $q$-Schlesinger isomonodromic. 
To find this conjugated family, we had to solve a diagonal system of $q$-difference equations, which boils down to solving two scalar  linear $q$-difference equations. Contrarily to the differential case, the resolution of $q$-difference equations even of such simple form is not trivial, and does not seem to be known in full generality. However, if some strong assumptions on the domain of definition $\mathfrak{D}$ are satisfied, one can use for example the following lemma. 
  
  \begin{lem} Assume that $|q|>1$ and $ \mathfrak{D}$ contains an  annulus of the form 
$$\{ t \in \C~|~ a<|t|<b\}$$
for some real numbers $0<a<b<\infty$.
   Let $c\in \mathcal{M}(\mathfrak{D})\setminus \{0\}$. 
Then there exists a meromorphic solution $\mathfrak{m}(t)\in \mathcal{M}\left({\mathfrak{D}}\right)\setminus \{0\}$
of $$\sigma_{q,t}\mathfrak{m}=c\, \mathfrak{m}\, .$$
 \end{lem}
 
 Note that the assumption on $\mathfrak{D}$ of the above lemma is satisfied if for instance
 ${\mathfrak{D}=\C^* \setminus \displaystyle\bigcup_{\ell=1}^{k} a_{\ell}q^{\Z}}$, for some $a_{\ell}\in \C^*$.
 
\begin{proof}
  Let us define  
$$O_{1}:= \{ t\in \mathbb{P}^1~|~|t| <b\} \cup \big(\mathfrak{D}\cap \{ t\in \mathbb{P}^1~|~|t|\geq b\}\big),$$
$$O_{2}:= \{ t\in \mathbb{P}^1~|~|t|> a \} \cup \big(\mathfrak{D}\cap \{ t\in \mathbb{P}^1~|~|t|\leq a\}\big).$$
These are connected open sets satisfying $O_{1} \cap O_{2}=\mathfrak{D}$ and $O_{1} \cup O_{2}=\mathbb{P}^1 $.
By \cite[Lemma~4.4]{bachmayr2016differential},  there exist $c_{1}\in \mathcal{M}(O_{1})$, and $c_{2}\in \mathcal{M}(O_{2})$ such that $c=c_{1}c_{2}$. By construction, $c_{1}$ is a germ of meromorphic function at $0$. By Remark \ref{rem1}, there exists $0\neq \mathfrak{m}_{1}$ that is meromorphic on a punctured neighborhood of $0$ in $\C^*$ such that $\sigma_{q,t}\mathfrak{m}_{1}=c_{1} \mathfrak{m}_{1}$. Using the functional equation and using the fact that $\mathfrak{D}$ is stable by multiplication by $q$, we find that $\mathfrak{m}_{1}$ may be continued into a meromorphic function on $qO_{1}\setminus \{0\}$ where $qO_{1}=\{qt, t\in O_{1}\}$.  Similarly, we construct a non-zero meromorphic solution of $\sigma_{q,t}\mathfrak{m}_{2}=c_{2} \mathfrak{m}_{2}$ that is meromorphic on $q^{-1}O_{2}\setminus \{\infty\}$.
Since $\mathfrak{D}\subset \C^*$ and $qO_{1} \cap q^{-1}O_{2}\subset O_{1} \cap O_2=\mathfrak{D}$ we find that a convenient solution is $\mathfrak{m}=\mathfrak{m}_{1}\mathfrak{m}_{2}$.\end{proof}

 \subsection{Confluence}\label{sec:confSchlesinger}

 Let $\boldsymbol{\Theta}(q)=(\Theta_0(q), \Theta_1(q),\Theta_t(q),\Theta_\infty(q))$ be a quadruple of rational functions in a complex variable $q$ such that as $q\to 1$, we have 
 \begin{equation}\label{ThetaConfValSchles}\Theta_i(q) = 1+(q-1)\frac{\theta_i}{2}+O(q-1)^2\quad \forall i\in \{0,1,t,\infty\}\, \end{equation}
 with $\theta_i\in \C$. We define $\overline{\boldsymbol{\Theta}}$  by $\overline{\Theta}_i =\frac{1}{\Theta_i }$. Recall from Section \ref{sec:confFuchsSys} that these requirements on $(\boldsymbol{\Theta},\overline{\boldsymbol{\Theta}})$ are a convenient setting for the discretization of $\mathfrak{sl}_2$-Fuchsian systems with spectral data~$\boldsymbol{\theta}$ (if $\theta_\infty\neq 0$).  We shall now see under these requirements, the $q$-Schlesinger equations discretize the (differential) Schlesinger equations, and that the difference equation \eqref{qPVIyZ} generically characterizing  $q$-isomonodromy discretizes  the differential equation \eqref{eqHVI} generically characterizing  isomonodromy.  
 
 \subsubsection{The $q$-Schlesinger equations discretize the differential ones} \label{sub1confSchlesinger}$ $\\
    The \emph{$q$-Schlesinger equations} are obtained from the equations in point $(2)$ of Proposition \ref{qisomchar} by setting $\mathfrak{C}=\Itwo $. With respect to the $\partial_{q,t}$-operator and the matrices $$\widetilde{\mathfrak{A}}_0=\frac{\mathfrak{A}_0-\Itwo }{q-1},\quad  \quad \widetilde{\mathfrak{A}}_1=\frac{\mathfrak{A}_1}{q-1}\,,  \quad \quad \widetilde{\mathfrak{A}}_t=\frac{\mathfrak{A}_t}{t(q-1)}\, , $$ they read as follows:
   \begin{equation} \label{eq:qSchlestilde} \left\{ \begin{array}{rcl}
   \partial_{q,t}\widetilde{\mathfrak{A}}_0&=&\displaystyle \frac{B_0\widetilde{\mathfrak{A}}_0 B_0^{-1} -\widetilde{\mathfrak{A}}_0}{(q-1)t}\vspace{.2cm}\\
    \partial_{q,t}\widetilde{\mathfrak{A}}_1&=&\displaystyle \frac{1}{(q-1)t}\left(  \frac{(t-1)}{qt-1}\frac{(qt\Theta_t-1)(qt\overline{\Theta}_t-1)}{q(t\Theta_t-1)(t\overline{\Theta}_t-1)}(q\Itwo +B_0) \,\widetilde{\mathfrak{A}}_1\left(\Itwo +B_0 \right)^{-1}-\widetilde{\mathfrak{A}}_1\right)\vspace{.2cm}\\
        \partial_{q,t}\widetilde{\mathfrak{A}}_t&=& \displaystyle -\frac{ 1}{(q-1)^2t}\left(\Itwo +\frac{1}{qt\Theta_t\overline{\Theta}_t}B_0\right)- \frac{1}{(q-1)t} \left( B_0\widetilde{\mathfrak{A}}_0  \left(\frac{1}{qt\Theta_t\overline{\Theta}_t}\Itwo +B_0^{-1}\right)+\widetilde{\mathfrak{A}}_t\right)\vspace{.2cm}\\
     &&\displaystyle - \frac{(t-1)(q\Itwo +B_0) \, \widetilde{\mathfrak{A}}_1 }{q(q-1)t(t\Theta_t-1)(t\overline{\Theta}_t-1)}\cdot \left(\Itwo +\frac{(qt\Theta_t-1)(qt\overline{\Theta}_t-1)}{qt-1}\left(\Itwo +B_0 \right)^{-1} \right) \, ,
     \end{array} \right.\end{equation}
     where $B_0=B_0\left(\widetilde{\mathfrak{A}}_0, \widetilde{\mathfrak{A}}_1,t,q\right)$ is the function with values in 
     $$G:=\{M\in \mathrm{GL}_2(\C)~|~\det(M+\Itwo )\neq 0\}\,,$$
     defined, on the complement of some proper Zariski closed   subset of $\mathrm{GL}_2(\C)\times \mathrm{GL}_2(\C) \times \C \times \C$, as $B_0\left(\widetilde{\mathfrak{A}}_0, \widetilde{\mathfrak{A}}_1,t,q\right)$ being given by
     $$\begin{array}{l}-qt\left(\Itwo +(q-1)\left(\frac{t-1}{(t\Theta_t-1)(t\overline{\Theta}_t-1)}\widetilde{\mathfrak{A}}_1-\widetilde{\mathfrak{A}}_\infty\right)\right)\left(\Itwo +(q-1)\left(\widetilde{\mathfrak{A}}_0+\frac{t(t-1)}{(t\Theta_t-1)(t\overline{\Theta}_t-1)} \widetilde{\mathfrak{A}}_1\right)\right)^{-1}\, , \end{array}$$
where $\widetilde{\mathfrak{A}}_\infty(q)=\mathrm{diag}\left(\frac{1-\overline{\Theta}_\infty}{q-1},\frac{1-{\Theta}_\infty}{q-1} \right)=\mathrm{diag}\left(\frac{\theta_\infty}{2},-\frac{{\theta}_\infty}{2} \right)+O(q-1)\, .$
Since $\widetilde{\mathfrak{A}}_\infty(q)$ is given by this context, may write everything as a function of $\widetilde{\mathfrak{A}}_0$ and $\widetilde{\mathfrak{A}}_1$ by  identifying 
  $\widetilde{\mathfrak{A}}_t:=-\left(\widetilde{\mathfrak{A}}_0+  \widetilde{\mathfrak{A}}_1+ \widetilde{\mathfrak{A}}_\infty \right)$. With this notation, as $q\to 1$, up to terms of order $O(q-1)^2$, we have
 $$\begin{array}{rcl}
   B_0\left(\widetilde{\mathfrak{A}}_0, \widetilde{\mathfrak{A}}_1,t,q\right)&\sim &-qt\left(\Itwo +(q-1)\left(\frac{1}{t-1}\widetilde{\mathfrak{A}}_1-\widetilde{\mathfrak{A}}_\infty\right)\right)\left(\Itwo +(q-1)\left(\widetilde{\mathfrak{A}}_0+\frac{t}{t-1} \widetilde{\mathfrak{A}}_1\right)\right)^{-1} \vspace{.2cm}\\
   &\sim &-qt\left(\Itwo +(q-1)\left(\frac{1}{t-1}\widetilde{\mathfrak{A}}_1-\widetilde{\mathfrak{A}}_\infty\right)\right)\left(\Itwo -(q-1)\left(\widetilde{\mathfrak{A}}_0+\frac{t}{t-1} \widetilde{\mathfrak{A}}_1\right)\right)\vspace{.2cm}\\
  &\sim &-qt\left(\Itwo +(q-1) \widetilde{\mathfrak{A}}_t\right)  \, .
   \end{array}$$

     Let $f_{0},f_{1},f_{t}$ be the functions with values in $\mathrm{M}_2(\C)$, defined on some obvious domain of definition inside $ \mathrm{GL}_2(\C)\times  \mathrm{GL}_2(\C)\times \C\times \left(\C\setminus \{1\}\right)$, that  
when evaluated in $\left(\widetilde{\mathfrak{A}}_0,\widetilde{\mathfrak{A}}_1,t,q\right)$, yield the right hand sides of the equations in  \eqref{eq:qSchlestilde}. Then we have 
  $$\begin{array}{rcl}
 f_0(\widetilde{\mathfrak{A}}_0,\widetilde{\mathfrak{A}}_1,t,q)&=&\displaystyle \frac{\left(\Itwo +(q-1)\widetilde{\mathfrak{A}}_t\right)\widetilde{\mathfrak{A}}_0\left(\Itwo -(q-1)\widetilde{\mathfrak{A}}_t\right) -\widetilde{\mathfrak{A}}_0}{(q-1)t} +O(q-1)\vspace{.2cm}\\
 &=&\displaystyle \frac{[\widetilde{\mathfrak{A}}_0,\widetilde{\mathfrak{A}}_t]}{0-t} +O(q-1)\, , \vspace{.2cm}\\ 
    f_1(\widetilde{\mathfrak{A}}_0,\widetilde{\mathfrak{A}}_1,t,q) &=&\displaystyle \frac{ (\Itwo -(q-1)\frac{t}{1-t}\widetilde{\mathfrak{A}}_t)\widetilde{\mathfrak{A}}_1 \left( \Itwo -(q-1)\frac{t}{1-qt}\widetilde{\mathfrak{A}}_t \right)^{-1}-\widetilde{\mathfrak{A}}_1}{(q-1)t}+O(q-1)\vspace{.2cm}\\  
    &=&\displaystyle \frac{ (\Itwo -(q-1)\frac{t}{1-t}\widetilde{\mathfrak{A}}_t)\widetilde{\mathfrak{A}}_1 \left( \Itwo +(q-1)\frac{t}{1-t}\widetilde{\mathfrak{A}}_t \right)-\widetilde{\mathfrak{A}}_1}{(q-1)t}+O(q-1)\vspace{.2cm}\\ 
    &=&\displaystyle \frac{[\widetilde{\mathfrak{A}}_1,\widetilde{\mathfrak{A}}_t]}{1-t}  +O(q-1)\, .    \end{array}$$
    Using similar calculations, and the Taylor series expansion of $B_0$ until its second order term, one finds
  $$\begin{array}{rcl}
    f_t(\widetilde{\mathfrak{A}}_0,\widetilde{\mathfrak{A}}_1,t,q) &=& \displaystyle -\frac{[\widetilde{\mathfrak{A}}_0,\widetilde{\mathfrak{A}}_t]}{0-t}-\frac{[\widetilde{\mathfrak{A}}_1,\widetilde{\mathfrak{A}}_t]}{1-t}  +O(q-1)\, .
    \end{array}$$
To summary, we have proved 
   \begin{equation*}\left\{\begin{array}{rcc}
\partial_{q,t}\widetilde{\mathfrak{A}}_0 (t)&=& \frac{[\widetilde{\mathfrak{A}}_0(t),\widetilde{\mathfrak{A}}_t(t)]}{0-t}+O(q-1) \vspace{.2cm}\\
\partial_{q,t}\widetilde{\mathfrak{A}}_1 (t)&=&\frac{[\widetilde{\mathfrak{A}}_1(t),\widetilde{\mathfrak{A}}_t(t)]}{1-t}+O(q-1) \vspace{.2cm}\\
\partial_{q,t}\widetilde{\mathfrak{A}}_t (t)&=& -\frac{\left[\widetilde{\mathfrak{A}}_{0}(t) ,\widetilde{\mathfrak{A}}_{t}(t)\right]}{0-t}-\frac{\left[\widetilde{\mathfrak{A}}_{1}(t),\widetilde{\mathfrak{A}}_{t}(t)\right]}{1-t}+O(q-1).
\end{array}\right.\end{equation*}

This proves that the $q$-Schlesinger equations \eqref{eq:qSchlestilde} discretize the (differential) Schlesinger equations \eqref{Schles}. 
\\
In order to complement this result, let us consider the function $\widetilde{\mathfrak{B}}$ with values in $\mathrm{M}_2(\C)$, given, on its obvious set of definition inside $\mathrm{GL}_2(\C)\times \mathrm{GL}_2(\C)\times \C\times \C\times \left(\C\setminus \{1\}\right)$, by 
   $$\widetilde{\mathfrak{B}}\left(\widetilde{\mathfrak{A}}_0, \widetilde{\mathfrak{A}}_1,x,t,q\right):= \frac{1}{(q-1)t}\left(\frac{(x-qt)\left(x\Itwo +B_0\left(\widetilde{\mathfrak{A}}_0, \widetilde{\mathfrak{A}}_1,t,q\right)\right)}{(x-qt\Theta_t)(x-qt\overline{\Theta}_t)}-\Itwo \right)\, .$$
   This function corresponds to the right hand side of the $\delta_{q,t}$-version of $\sigma_{q,t}Y=\mathfrak{B}Y$ with $\mathfrak{B}$ as in \eqref{matrixBinSchles}. 
It  behaves, when $q\to 1$, as   $$\begin{array}{rcl}
\widetilde{\mathfrak{B}}\left(\widetilde{\mathfrak{A}}_0, \widetilde{\mathfrak{A}}_1,x,t,q\right)    &\sim &\frac{1}{(q-1)t}\left(\frac{(x-qt)\left((x-qt)\Itwo -(q-1)t \widetilde{\mathfrak{A}}_t\right)}{(x-qt)^2}-\Itwo +O(q-1)^2\right)\vspace{.2cm}\\
  &\sim &\frac{1}{(q-1)t}\left(-(q-1)t\frac{ \widetilde{\mathfrak{A}}_t}{(x-qt)}+O(q-1)^2\right)\vspace{.2cm}\\
  &\sim & - \frac{ \widetilde{\mathfrak{A}}_t}{x-t}+O(q-1) \, .
   \end{array}$$
By the above estimates, $\widetilde{\mathfrak{B}}$ can be continued analytically to $\{q=1\}$ and is there given by  $- \frac{ \widetilde{\mathfrak{A}}_t}{x-t}$.

\subsubsection{The $q$-Lax pairs discretize the differential ones}\label{sub2confSchlesinger}$ $\\
Let $$\qDom\subset \C^*\setminus \mathrm{e}^{2i\pi \Q}$$ be a connected subset with $1$ in its closure. Let $\mathfrak{D}\subset \C^*$ be an open connected subset. 
We shall assume that the pair $(\mathfrak{D},\qDom)$ satisfies the property that  $\mathfrak{D}$ is stable by multiplication by $q^{\pm 1}$, for every $q\in \qDom$.  Note that unless $\mathfrak{D}= \C^*$, the subset $\qDom$ cannot be too large.
Two examples of a convenient pair   $( \mathfrak{D},\qDom )$ with $\mathfrak{D}\subset \C^*$ are the following:
 \begin{itemize}
\item $(\mathfrak{D},\qDom)=\left(\C^{*}\setminus q_{0}^{\R},q_{0}^{\R_{>0}} \right)$ where  $q_{0}\in \C$, with $|q_{0}|\neq 1$.\vspace{.2cm}
\item $\mathfrak{D}$ is an open sector with infinite radius centered at $0$ and  $\qDom=]1,+\infty[$.
\end{itemize} 

In addition to our previous requirements on $(\boldsymbol{\Theta}, \overline{\boldsymbol{\Theta}})$, let us now moreover assume that the (differential) spectral values $\boldsymbol{\theta}$ satisfy the non-resonancy condition $\theta_i\not \in \Z^*$ for $i\in \{0,1,t,\infty\}$ and $\theta_\infty \neq 0$. Note that for values of $q\in \qDom$ sufficiently close to $1$,  the ($q$-difference) non-resonancy condition \eqref{eq:qresonance} then is automatically is satisfied.
 
Let $A_0,A_1,A_t\in \mathrm{GL}_2(\mathcal{O}(\mathfrak{D}))$ such that 
$\partial_xY(x,t)=A(x,t)Y$ with $A=\frac{A_0}{x}+\frac{A_1}{x-1}+\frac{A_t}{x-t}$ is a family of $\mathfrak{sl}_2$-Fuchsian systems with spectral data $\boldsymbol{\theta}$ as in Definition \ref{defi1}. 
Let  $\widetilde{\mathfrak{A}}_0,\widetilde{\mathfrak{A}}_1,\widetilde{\mathfrak{A}}_t$ be holomorphic functions in a neighborhood of $\mathfrak{D}\times \qDom\subset \C^2$ with values in $\mathrm{GL}_2(\C)$, 
such that for each $q\in \qDom$, the $q$-difference equation $\partial_{q,x}Y=\widetilde{\mathfrak{A}}(x,t,q)Y$ with $\widetilde{\mathfrak{A}}=\frac{\widetilde{\mathfrak{A}}_0}{x}+\frac{\widetilde{\mathfrak{A}}_1}{x-1}+\frac{\widetilde{\mathfrak{A}}_t}{x-t}$ yields, \emph{via} $\mathfrak{A}=\Itwo +x(q-1)\widetilde{\mathfrak{A}}$, a family of $q$-Fuchsian systems with spectral data $(\boldsymbol{\Theta}(q), \overline{\boldsymbol{\Theta}}(q))$ as in Definition \ref{qfuchsian}. 
By Proposition \ref{prop3} it is convenient to assume that  
$$\forall i\in \{0,1,t\}\, , \quad \quad \lim_{q\to 1} \widetilde{\mathfrak{A}}_i(t,q)=A_i(t)\, ,$$
and that these limits are uniform on compact subsets of $\mathfrak{D}$, such that 
$$\forall i\in \{0,1,t\}\, , \quad \quad  \lim_{q\to 1} \partial_{q,t}\widetilde{\mathfrak{A}}_i(t,q)=A_i'(t). $$
 Finally, let us assume that the family  $\partial_{q,x}Y=\widetilde{\mathfrak{A}}(x,t,q)Y$ is $q$-isomonodromic for each $q\in \qDom$. 
By non-resonancy, \textsection \ref{sub1confSchlesinger} and the proof of Proposition \ref{qisomchar}, this means that  this family can be completed into a $q$-Lax pair 
\begin{equation*}\left\{ \begin{array}{rcl}\sigma_{q,x}Y&=&\left(\Itwo +(q-1)x\widetilde{\mathfrak{A}}(x,t,q)\right)Y\\
\sigma_{q,t}Y&=&\mathfrak{C}(t,q)\left(\Itwo +(q-1)t\widetilde{\mathfrak{B}}(x,t,q)\right)Y\, . \end{array}\right.\end{equation*}
We then find
\begin{equation}\label{qLaxConf}\left\{ \begin{array}{rcl}\partial_{q,x}Y&=&\widetilde{\mathfrak{A}}(x,t,q)Y\\\partial_{q,t}Y&=&\left(\mathfrak{C}(t,q)\widetilde{\mathfrak{B}}(x,t,q)+\frac{\mathfrak{C}(t,q)-\Itwo }{(q-1)t}\right)Y\, , \end{array}\right.\end{equation}
where 
$$\lim_{q\to 1}\widetilde{\mathfrak{B}}(x,t,q)=-\frac{A_t(t)}{x-t}\quad \quad \textrm{and}\quad \quad \mathfrak{C}(t,q)=f(t,q)\mathrm{diag}(c_1(t,q),c_2(t,q))\, .$$
Here $f\in \mathcal{M}(\mathfrak{D}\times \qDom)\setminus \{0\}$ can be chosen arbitrarily and 
$$c_1=\frac{\sigma_{q,t}\left(t\widetilde{\mathfrak{A}}_0^{(1,2)}+(t-1)\widetilde{\mathfrak{A}}_1^{(1,2)}\right)}{t\widetilde{\mathfrak{A}}_0^{(1,2)}+(t-1)\widetilde{\mathfrak{A}}_1^{(1,2)}}\, , \quad c_2 =\frac{q\overline{\Theta}_\infty X-1 }{\Theta_\infty X-1} \, ,$$
where $X$ is some rational expression in terms of $q,t,\Theta_1, \Theta_\infty, \widetilde{\mathfrak{A}}_0^{(1,1)}, \widetilde{\mathfrak{A}}_0^{(1,2)}, \widetilde{\mathfrak{A}}_1^{(1,1)},\widetilde{\mathfrak{A}}_1^{(1,2)}$ that can easily be made explicit.  We assume these $c_i$ and $X$ to be well-defined and finite. 
Using the Taylor series expansion of $\Theta_1(q)$ and $\Theta_\infty(q)$, we readily compute that up to terms of order $O(q-1)^2$, we have
$$\frac{1}{f}\mathfrak{C}\sim \Itwo +\frac{(q-1)t}{t\widetilde{\mathfrak{A}}_0^{(1,2)}+(t-1)\widetilde{\mathfrak{A}}_1^{(1,2)}} \begin{pmatrix} 
 t\partial_{q,t}\widetilde{\mathfrak{A}}_0^{(1,2)}+(t-1)\partial_{q,t}\widetilde{\mathfrak{A}}_1^{(1,2)}&0\\
 0& (1-\theta_\infty)\left(\widetilde{\mathfrak{A}}_0^{(1,2)}+\widetilde{\mathfrak{A}}_1^{(1,2)}\right)\end{pmatrix}\, .$$
 Choosing $f$ of the form $f(t,q)=1+(q-1)g(t)$ for some meromorphic function $g$, we can make sure that $\lim_{q\to 1}\mathfrak{C}(t,q)=\Itwo $ and that 
 the matrix
 $$C(t):=\lim_{q\to 1}\frac{\mathfrak{C}(t,q)-\Itwo }{(q-1)t} $$
 exists and is tracefree. Consider a subset of  $\mathfrak{D}$ where $C(t)$ is holomorphic.
With $\mathfrak{C}=\Itwo+(q-1)tC$, the $q$-Lax pair induced by \eqref{qLaxConf} is
\begin{multline*}
\left(\Itwo +(q-1)x\sigma_{q,t}\widetilde{\mathfrak{A}}\right)(\Itwo+(q-1)tC)\left(\Itwo +(q-1)t\widetilde{\mathfrak{B}}\right)\\
=(\Itwo+(q-1)tC)\left(\Itwo +(q-1)t\sigma_{q,x}\widetilde{\mathfrak{B}}\right)\left(\Itwo +(q-1)x\widetilde{\mathfrak{A}}\right).
\end{multline*}
The term in $O(1)$ in both sides of the equality is $\Itwo=\Itwo$. The term in $O(q-1)$ is  
$$(q-1)x\sigma_{q,t}\widetilde{\mathfrak{A}}+(q-1)t\widetilde{\mathfrak{B}}=(q-1)t\sigma_{q,x}\widetilde{\mathfrak{B}}+(q-1)x\widetilde{\mathfrak{A}}. $$
Finally, the term in $O(q-1)^{2}$ is 
\begin{multline*}(q-1)^{2}xt\sigma_{q,t}(\widetilde{\mathfrak{A}})C+(q-1)^{2}xt\sigma_{q,t}(\widetilde{\mathfrak{A}})\widetilde{\mathfrak{B}}+(q-1)^{2}t^{2}C\widetilde{\mathfrak{B}}\\
=(q-1)^{2}xt\sigma_{q,x}(\widetilde{\mathfrak{B}})\widetilde{\mathfrak{A}}+(q-1)^{2}t^{2}C\sigma_{q,x}(\widetilde{\mathfrak{B}})+(q-1)^{2}xtC\widetilde{\mathfrak{A}}.\end{multline*}
Then, dividing by $(q-1)^{2}xt$ and we obtain with  $\partial_{q,x}C=0$ that
$$\partial_{q,t}(\widetilde{\mathfrak{A}})-\partial_{q,x}(\widetilde{\mathfrak{B}}+C)=\sigma_{q,x}(\widetilde{\mathfrak{B}})\widetilde{\mathfrak{A}}+x^{-1}tC\sigma_{q,x}(\widetilde{\mathfrak{B}})+C\widetilde{\mathfrak{A}}-\sigma_{q,t}(\widetilde{\mathfrak{A}})C-\sigma_{q,t}(\widetilde{\mathfrak{A}})\widetilde{\mathfrak{B}}-x^{-1}tC\widetilde{\mathfrak{B}}+O(q-1).$$
With $\sigma_{q,t}(\widetilde{\mathfrak{A}})=\widetilde{\mathfrak{A}}+O(q-1)$,  $\sigma_{q,x}(\widetilde{\mathfrak{B}})=\widetilde{\mathfrak{B}}+O(q-1)$,  we obtain 
$$\partial_{q,t}(\widetilde{\mathfrak{A}})-\partial_{q,x}(\widetilde{\mathfrak{B}}+C)=\left[\widetilde{\mathfrak{B}}+C, \widetilde{\mathfrak{A}}
\right] +O(q-1).$$
Since  $\widetilde{\mathfrak{A}}=A+O(q-1)$, $\widetilde{\mathfrak{B}}=-\frac{A_t(t)}{x-t}+O(q-1)$,
this shows that we obtain the confluence of the $q$-Lax pair \eqref{qLaxConf} to the differential Lax pair of 
$$\left\{ \begin{array}{rcl}\partial_{x}Y&=&A(x,t)Y\\\partial_{t}Y&=&\left(-\frac{A_t(t)}{x-t}+C(t) \right)Y\, , \end{array}\right.$$
 which is 
$$\partial_{t}(A)-\partial_{x}\left(-\frac{A_t(t)}{x-t}+C\right)=\left[-\frac{A_t(t)}{x-t}+C, A
\right].$$
 \subsubsection{The difference equations for triples  $(\boldsymbol{\lambda}, \y,\bZZ)$ discretize the differential ones}\label{sub3confSchlesinger}$ $\\
Of particular interest for the results in this paper is the discretization of the characterization of Schlesinger isomonodromy in terms of triples $(\lambda(t), y(t),\ZZ(t))$ given in Proposition \ref{isoPain}, namely the system of differential equations given by \eqref{eqHVI} and \eqref{eqlam}. Proposition \ref{qisomchar} suggests that a convenient $q$-analogue of this differential equation is given by 
 
   \begin{equation}\label{eq:qSchlesdiff}\left\{ \begin{array}{rcl}
 \frac{ \partial_{q,t}\boldsymbol{\lambda}}{\boldsymbol{\lambda}} &=&\displaystyle \frac{(\Theta_\infty-q\overline{\Theta}_\infty) X }{(q-1)t\left(1-\Theta_\infty X\right)}
   \vspace{.2cm}\\
   \partial_{q,t}\y &=&\displaystyle \frac{E-\y}{(q-1)t} \vspace{.2cm}\\
    \partial_{q,t}\bZZ &=&\displaystyle \frac{1}{(q-1)t} \cdot\left( \frac{(E-qt\Theta_t)(E-qt\overline{\Theta}_t)}{q(q-1)E(E-1)( E-qt)X}-\frac{1}{(q-1)E}-\bZZ\right)
    \, ,
   \end{array} \right.\end{equation}
where $X(\y,\bZZ,t,q)$ and $E(\y,\bZZ,q,t)$ are defined respectively as  $$X:=\frac{(\y-1)(\y-t)(1+(q-1)\y\bZZ)}{  (\y-\Theta_1)(\y-\overline{\Theta}_1)}\quad ,  \quad E:=\frac{q}{\y} \cdot \frac{(X-t\Theta_0)(X-t\overline{\Theta}_0)}{(\Theta_\infty X-1)(q\overline{\Theta}_\infty X-1)}\, .$$
Here, as usual when considering confluence, we used our convention $\Theta_i\overline{\Theta}_1=1$.
Let us now show that \eqref{eq:qSchlesdiff} discretizes the system of differential equations   given by equations \eqref{eqlam} and \eqref{eqHVI}. 
 
 Let $f_{\boldsymbol{\lambda}},f_{\y},f_{\bZZ}$ be the rational functions in the variables $ \y,\bZZ,t,q$ forming the right hand sides of the equations in \eqref{eq:qSchlesdiff}. 
 Using the estimates \eqref{ThetaConfValSchles}, we may compute the Taylor series expansion of 
 $X(q):=X(\y,\bZZ,t,q)$ as $q\to 1$. Up to terms of order $O(q-1)^3$, we have 
 $$\begin{array}{rcl}
X(q)&= &\frac{(\y-1)(\y-t)(1+(q-1)\y\bZZ)}{(\y-1)^2+\y(2-\Theta_1-\overline{\Theta}_1)}  \sim  \frac{\frac{\y-t}{\y-1}(1+(q-1)\y\bZZ)}{1-(q-1)^2\frac{\theta_1^2}{4}\frac{\y}{(\y-1)^2}}   \vspace{.2cm}\\ 
 &\sim &\frac{\y-t}{\y-1}(1+(q-1)\y\bZZ) \left(1+(q-1)^2\frac{\theta_1^2}{4}\frac{\y}{(\y-1)^2}\right)
 \vspace{.2cm}\\ 
 &\sim &\frac{\y-t}{\y-1}\left(1+(q-1)\y\bZZ +(q-1)^2\frac{\theta_1^2}{4}\frac{\y}{(\y-1)^2}\right)\, .
    \end{array}$$
So in particular, we have $X(q)=\frac{\y-t}{\y-1}+O(q-1)$. Since moreover $\Theta_\infty=1+O(q-1)$ and $\Theta_\infty-q\overline{\Theta}_\infty=(q-1)(\theta_\infty-1)+O(q-1)^2$, we conclude that 
$$f_{\boldsymbol{\lambda}}(q):=f_{\boldsymbol{\lambda}}(\y,\bZZ,t,q)=\frac{(\theta_\infty-1)(\y-t)}{t(t-1)}+O(q-1)\, .$$
In other words, the first difference equation in \eqref{eq:qSchlesdiff} discretizes \eqref{eqlam}. Similarly, 
up to order terms of order $O(q-1)^3$, we obtain 
 $$\begin{array}{rcl}
E(q)&= &\frac{\left(X(q)-t\Theta_0 \right)\left(X(q)-t\frac{1}{\Theta_0}\right)}{\y \left(X(q)-\frac{1}{\Theta_\infty} \right)\left(X(q)-\frac{\Theta_\infty}{q}\right)} \vspace{.2cm}\\ 
&\sim &   \frac{\frac{1}{\y}\left(X(q)-t\right)^2-(q-1)^2\frac{\theta_0^2}{4} tX(q)}{\left(X(q)-1\right)^2 -(q-1)\frac{t-1}{\y-1}+(q-1)^2\left(\frac{t-1}{\y-1} -\frac{\y-t}{\y-1} \left(\frac{\theta_\infty^2-2\theta_\infty}{4}-\y\bZZ \right) \right)  }\, .
      \end{array}$$
      Substituting the Taylor expansion of $X(q)$ up to order $O(q-1)^3$ yields
       $$\begin{array}{rcl}
E(q)&\sim &  \frac{
     \y \left(\left(1-(q-1)\frac{\y-t}{t-1}\bZZ\right)^2-(q-1)^2\left(\frac{t\theta_0^2}{4}\frac{(\y-t)(\y-1)}{(t-1)^2\y^2}+\frac{\theta_1^2(\y-t)}{2(t-1)(\y-1)^2}\right)\right)
      }{
\left(1-(q-1)\y\frac{\y-t}{t-1}\bZZ\right)^2     -(q-1)\frac{\y-1}{t-1}+(q-1)^2\left(\frac{\y-1}{t-1} -\frac{(\y-t)(\y-1)}{(t-1)^2} \left(\frac{\theta_\infty^2-2\theta_\infty}{4}-\y\bZZ \right) -\frac{\theta_1^2\y(\y-t)}{2(\y-1)^2(t-1)}  \right)
    }
\vspace{.2cm}\\ 
&\sim &   \y +(q-1)\frac{\y(\y-1)(2(\y-t)\bZZ+1)}{t-1}
 \vspace{.2cm}\\
&& \quad + (q-1)^2\cdot\frac{\y(\y-1)(\y-t)}{t-1}\left(\frac{(3\y-1)(\y-t)\bZZ^2+(3\y-2)\bZZ}{t-1}-\frac{t\theta_0^2}{4(t-1)\y^2} +\frac{\theta_1^2}{2(\y-1)^2}+\frac{(\theta_\infty-1)^2+3}{4(t-1)}\right)

\, .
      \end{array}$$
      Already from the Taylor expansion of $E(q)$ up to order $O(q-1)^2$, we deduce that 
      $$f_{\y}(q)= \frac{\y(\y-1)(2(\y-t)\bZZ+1)}{t(t-1)}+O(q-1)\, .$$
      From the Taylor expansion of $E(q)$ and $X(q)$ up to order $O(q-1)^3$, we deduce by a series of tedious but straightforward calculations that 
        $$\begin{array}{rcl}f_{\bZZ}(q)&=&
       -  \frac{\y(\y-1)(\y-t)\bZZ^2}{t(t-1)} \left(\frac{1}{\y} +\frac{1}{\y-1} +\frac{1}{\y-t} \right)  
     -\frac{\left(2\y-1\right)\bZZ}{t(t-1)} \vspace{.2cm}\\
  &&\quad -\frac{\theta_0^2}{4(t-1)\y^2} +\frac{\theta_1^2}{4t(\y-1)^2}-\frac{\theta_t^2}{4( \y-t)^2} +\frac{(\theta_\infty-1)^2-1}{4t(t-1)}+O(q-1)\, .\end{array}$$
It follows that the second and third difference equation in \eqref{eq:qSchlesdiff} together discretize the system of differential equations  \eqref{eqHVI}.

\section{The sixth Painlev\'e equation}
\subsection{Differential case}\label{sec:diffP6}
Let  $\boldsymbol{\theta}=(\theta_0,\theta_1,\theta_t,\theta_\infty)\in ( \mathbb{C})^4$. We define the rational function $H^{\boldsymbol{\theta}}_{\mathrm{VI}}\in \C(y,Z,t)$ in three variables given by 
$$\begin{array}{rcl}H^{\boldsymbol{\theta}}_{\mathrm{VI}}&:=&\frac{y(y-1)(y-t)}{t(t-1)}\left(\ZZ^2+\frac{\ZZ}{y-t}\right)-\frac{1}{4}\left(\frac{(\theta_\infty-1)^2-1}{t(t-1)}y+\frac{\theta_0^2}{(t-1)y}+\frac{\theta_t^2 }{y-t}-\frac{\theta_1^2}{t(y-1)}\right)\, .\end{array}$$ Consider the non-autonomous Hamiltonian system defined by
\begin{equation}\label{eq:HamSys}\left\{ \begin{array}{rcll}
y'(t)&=& &\frac{\partial H^{\boldsymbol{\theta}}_{\mathrm{VI}}}{\partial \ZZ}(y,\ZZ, t)\vspace{.2cm}
\\
\ZZ'(t)&=&-&\frac{\partial H^{\boldsymbol{\theta}}_{\mathrm{VI}}}{\partial y}(y,\ZZ, t)\, .\end{array}\right.\end{equation} Explicitly, it is given by 
\begin{equation}\label{eqHVIbis}\left\{ \begin{array}{rcl}
y'(t)&=& \frac{y(y-1)(y-t)}{t(t-1)}\left(2\ZZ+\frac{1}{y-t}\right)\vspace{.2cm}
\\
\ZZ'(t)&=&\frac{-3y^2+2(t+1)y-t}{t(t-1)}\ZZ^2-\frac{2y-1}{t(t-1)}\ZZ+\frac{1}{4}\left( \frac{(\theta_\infty-1)^2-1}{t(t-1)}-\frac{\theta_0^2}{(t-1)y^2}-\frac{\theta_t^2 }{(y-t)^2}+\frac{\theta_1^2}{t(y-1)^2}\right)\, .\end{array}\right.\end{equation}
Recall from Corollary \ref{corisoiff} that if for all $i\in \{0,1,t,\infty\}$, we have $\theta_i\not \in \Z^*$ and $\theta_\infty \neq 0$, then this system of differential equations characterizes isomonodromy for families of $\mathfrak{sl}_2$-Fuchsian systems. 
Substituting
$Z= \frac{t(t-1)y'(t)}{2y(y-1)(y-t)}-\frac{1}{2(y-t)}$ (from the first equation in \eqref{eqHVIbis}) into the second, we obtain the \emph{sixth Painlev\'e equation} associated to the spectral data $\boldsymbol{\theta}$:
$$P_{\mathrm{VI}}: \quad \left\{\begin{array}{rcl}y''&=&\frac{1}{2}\left(\frac{1}{y}+\frac{1}{y-1}+\frac{1}{y-t}\right){y'}^2-\left(\frac{1}{t}+\frac{1}{t-1}+\frac{1}{y-t}\right){y'} \\&+&\frac{y(y-1)(y-t)}{2t^2(t-1)^2}\left((\theta_\infty-1)^2+\frac{\theta_1^2(t-1)}{(y-1)^2} -\frac{\theta_0^2t}{y^2}-\frac{(\theta_t^2-1)(t-1)t}{(y-t)^2}\right)\, .\end{array}\right.$$
Conversely, given a meromorphic solution $y$ of $P_{\mathrm{VI}}$ (we will see in the sequel that it exists), and assuming it is not identically equal to $0,1,t$ (which is a trivially satisfied if $\theta_0\theta_1\theta_t\neq 0$), then the substitution formula yields a meromorphic function $\ZZ$ such that the pair $(y,\ZZ)$ is a meromorphic solution of  \eqref{eqHVIbis}.

Let us briefly recall the well-known results concerning the existence  of analytic solutions of $P_{\mathrm{VI}}$. 
By the Cauchy-Lipschitz theorem, for every $t_0 \in \C\setminus \{ 0,1\}$ and every choice of $(y_0,y_1)\in \left( \C\setminus \{ 0,1, t_0\}\right) \times \C$, there exists a unique holomorphic function $y(t)$ defined in a neighborhood of $t_0$ such that $y(t_0)=y_0$ and $y'(t_0)=y_1$, and such that $y$ is a solution of the sixth Painlev\'e equation. Equivalently, for every $t_0 \in \C\setminus \{ 0,1\}$ and every choice of $(y_0,Z_0)\in \left( \C\setminus \{ 0,1, t_0\}\right) \times \C$, there exists a unique holomorphic solution $(y(t),Z(t))$ of the Hamiltonian system \eqref{eqHVIbis}, defined in a neighborhood of $t_0$, such that $(y(t_0),Z(t_0))=(y_0,Z_0)$. 
By the so-called \emph{Painlev\'e property}, any such germ of holomorphic solution can be meromorphically continued along any path in $\C\setminus \{ 0,1\}$. In particular, on any simply connected subset $U$ of $\mathbb{P}^1\setminus \{0,1,\infty\}$, there exists a unique meromorphic solution satisfying some initial condition as above at $t_0\in U$ (see for instance  \cite{hinkkanen2004meromorphic,joshi1994direct}, see also Section \ref{sec:oka} for some details).

\subsection{A discrete analogue}\label{sec:32}
Let us fix $q\in \C\setminus \{0,1\}$ and 
let us consider the spectral data $(\Theta_0, \Theta_1,\Theta_t,\Theta_\infty, \overline{\Theta}_0, \overline{\Theta}_1, \overline{\Theta}_t, \overline{\Theta}_\infty)\in (\C^*)^8$ such that 
$$\Theta_0\overline{\Theta}_0=\Theta_1\overline{\Theta}_1\Theta_t\overline{\Theta}_t\Theta_\infty\overline{\Theta}_\infty\, .$$

In \cite{jimbo1996q} the $q$-Painlev\'e ${\mathrm{VI}}$ equation associated to such a spectral data was introduced. It is given by the following system of $q$-difference equations: 
 \begin{equation}\label{qPVIJS} qP_{\mathrm{JS,VI}}(\boldsymbol{\Theta}, \overline{\boldsymbol{\Theta}}) : \left\{\begin{array}{rcl}
\displaystyle  \frac{\y \cdot \sigma_{q,t} \y }{\Theta_1\overline{\Theta}_1}  &=& \displaystyle\frac{\left( \sigma_{q,t}  \zJS - t\frac{\Theta_t\overline{\Theta}_t}{\Theta_0} \right)\left(\sigma_{q,t}  \zJS - t\frac{\Theta_t\overline{\Theta}_t}{\overline{\Theta}_0}\right)}{\left(\sigma_{q,t}  \zJS - \frac{1}{q\overline{\Theta}_\infty}  \right)\left(\sigma_{q,t} \zJS - \frac{1}{\Theta_\infty} \right)},
 \vspace{.2cm}\\
\displaystyle \frac{\zJS \cdot \sigma_{q,t}\zJS }{\frac{1}{q\overline{\Theta}_\infty \Theta_\infty}} &=&\displaystyle \frac{\left(\y -t\Theta_t \right)\left(\y -t\overline{\Theta}_t \right)}{\left(\y -\Theta_1\right)\left(\y -\overline{\Theta}_1 \right)} \, .
\end{array}\right. \end{equation}
The auxiliary parameters in \cite{jimbo1996q} bear other names, but we have written the equation in a way that the dictionary between the auxiliary parameters in \cite{jimbo1996q} and the above $\Theta_i$, $\overline{\Theta}_i$  is obvious, see \eqref{introeq2}. 
This system of difference equations has been derived  in  \cite{jimbo1996q}, for $|q|\neq 1$, from the pseudo-constancy condition of the Birkhoff connection matrix for $q$-Fuchsian systems with non-resonant spectral data $(\boldsymbol{\Theta}, \overline{\boldsymbol{\Theta}})$. Note that the change of variable $$\zJS=\frac{(\y-t{\Theta}_t)(\y-t\overline{\Theta}_t)}{q(\y-1)(\y-t)(1+(q-1)\y\bZZ)}$$
applied to \eqref{qPVIJS} yields the $q$-difference system \eqref{qPVIyZ}. In the case $q\not \in \mathrm{e}^{2i\pi \Q}$ and non-resonant $(\boldsymbol{\Theta}, \overline{\boldsymbol{\Theta}})$, solutions of the latter system have been shown in Proposition \ref{qisomchar} to 
correspond (under some generic assumptions) to $q$-isomonodromic (in the sense of Definition \ref{def:qisom}) families of $q$-Fuchsian systems. 
Conversely, when starting with \eqref{qPVIyZ}, the change of variable 
$$\bZZ=\frac{\frac{(\y-t{\Theta}_t)(\y-t\overline{\Theta}_t)}{q(\y-1)(\y-t)\zJS}-1}{(q-1)\y}$$ yields equation \eqref{qPVIJS}, which has a significantly shorter and more symmetric expression. 
Note that with this change of variable, one has $\sigma_{q,t}\zJS=X$, for $X$ as in \eqref{qPVIyZ}. 

From now on, because we are ultimately interested in the behaviour under confluence, we will use our convention $$ \forall i\in \{0,1,t,\infty\},\, \quad \Theta_i\overline{\Theta}_i=1, $$
by which equations \eqref{qPVIJS} and \eqref{qPVIyZ} can obviously be simplified. In particular, from now on, the following system of $q$-difference equations will be referred to as the \emph{$q$-Painlev\'e ${\mathrm{VI}}$ equation} associated to spectral data 
$\boldsymbol{\Theta}\in (\C^*)^4$:
 \begin{equation}\label{qPVI} qP_{\mathrm{VI}}(\boldsymbol{\Theta}) : \left\{\begin{array}{rcl}
\displaystyle \y \cdot \sigma_{q,t} \y   &=& \displaystyle\frac{\left(\sigma_{q,t}  \zJS - t\Theta_0\right)\left( \sigma_{q,t}  \zJS - t\frac{1}{\Theta_0} \right)}{\left(\sigma_{q,t}  \zJS - \frac{{\Theta}_\infty}{q}  \right)\left(\sigma_{q,t} \zJS - \frac{1}{\Theta_\infty} \right)},
 \vspace{.2cm}\\
\displaystyle \zJS \cdot \sigma_{q,t}\zJS  &=&\displaystyle\frac{1}{q} \frac{\left(\y -t\Theta_t \right)\left(\y -t\frac{1}{\Theta_t} \right)}{\left(\y -\Theta_1\right)\left(\y -\frac{1}{{\Theta}_1} \right)} \, .
\end{array}\right. \end{equation}

Unfortunately, contrarily to the differential situation, the existence of a meromorphic solution having a prescribed value at a point $t_0\in   \mathbb{C}^*$ is in general not known. 
 Let us now focus on discrete solutions, \emph{i.e.} the sequence of values on $q^\Z t_0$ for some $t_0\in \mathbb{C}^*\setminus q^\Z$ that a meromorphic solution defined on a domain containing the spiral $q^\Z t_0$ should interpolate. 
  More precisely, a \emph{discrete solution} of \eqref{qPVI} is a sequence $$( \y_{\ell} ,  \zJS_{\ell},t_\ell)_{\ell \in \Z}$$ of points in $ \mathbb{P}^{1}\times \mathbb{P}^{1}\times \C^* $, such that 
\begin{itemize}
\item[$\bullet$] the sequence $(t_{\ell})_{\ell \in \Z}$ is given by $t_{\ell}=q^{\ell } t_0$ for some $t_0\in\C^* $
\item[$\bullet$] the sequence $(\y_{\ell} ,  \zJS_{\ell})_{\ell \in \Z}$ satisfies the following equations for each $\ell \in \Z$:
\begin{equation}\label{eq:recurrence}\left\{\begin{array}{rcl}
 \y_{\ell} \cdot \y_{\ell+1}  &=& \frac{\left(\zJS_{\ell+1}-t_{\ell}\Theta_0 \right)\left( \zJS_{\ell+1}-t_{\ell} \frac{1}{\Theta_0}\right)}{\left(\zJS_{\ell+1}-\frac{1}{\Theta_\infty} \right)\left( \zJS_{\ell+1}-\frac{\Theta_\infty}{q}\right)}\vspace{.2cm}\\
 \zJS_{\ell}\cdot \zJS_{\ell+1} &=&\frac{1}{q}\frac{\left(\y_{\ell} -t_{\ell}\Theta_t\right)\left(\y_{\ell} -t_{\ell}\frac{1}{\Theta_t}\right)}{\left(\y_{\ell} -\Theta_1\right)\left(\y_{\ell} -\frac{1}{\Theta_1}\right)} \, ,
\end{array}\right.\end{equation}
\item[$\bullet$] moreover, for each $\ell \in \Z$, the rational simplification of the relation implied by \eqref{eq:recurrence} between $(\y_{0} ,  \zJS_{0},t_0)$ and $(\y_{\ell} ,  \zJS_{\ell}, q^\ell t_{0})$ is also satisfied. 
\end{itemize}
Note that a more intrinsic notion of discrete solution will appear in Section \ref{sec:qoka}, where we briefly review the construction in \cite{sakai2001rational} of a $q$-analogue of the Okamoto space. 

Let us explain why for discrete solutions as in the above definition we do not only take into account the relations between successive pairs. Let $t_0\in\C^* $ and consider for example a pair $(\y_0,\zJS_0)$ where $\y_0=\overline{\Theta}_1$ and where $z_0\in \C^*$. 
Then by the recurrence relation \eqref{eq:recurrence}, we have $(\y_1,\zJS_1)=(\Theta_1,\infty)$. Then the second equation of the recurrence relation for $\ell=1$ simply writes $\infty = \infty$. So a priori, $\zJS_2$ could take any value. But  $\zJS_2$ is uniquely determined if we go back to  $(\y_0,\zJS_0)$ and take into account rational simplifications. Indeed, let us introduce an additional variable $\nu_{\ell}:=\zJS_{\ell}\cdot (\y_{\ell}-\Theta_1)$. When we write $\nu_1$ as a rational function of general $(\y_0,\zJS_0,t_0)$ \emph{via} the recurrence relation, then one immediately checks that nominator and denominator can both be factorized by $(\y_0-\overline{\Theta}_1)$. After this rational simplification, $\nu_1$ is well-defined for $\y_0=\overline{\Theta}_1$ and yields
$$\nu_1(\overline{\Theta}_1, \zJS_0, t_0)=\frac{(t_0\Theta_1-{\Theta}_t)(t_0{\Theta}_1-\overline{\Theta}_t)}{(\Theta_1-\overline{\Theta}_1)q\zJS_0}+\Theta_1\left(\left(\frac{\Theta_\infty}{q}+\overline{\Theta}_\infty\right)-t_0(\Theta_0+\overline{\Theta}_0)\right)\, .$$
On the other hand, the recurrence relation at level $\ell=1$ may be written as 
 $$ \left\{\begin{array}{rcl}
 \y_{2}  &=&\frac{1}{\y_1} \frac{\left(\zJS_{2}-qt_0\Theta_0 \right)\left( \zJS_{2}-qt_{0}\overline{\Theta}_0\right)}{\left(\zJS_{2}-\overline{\Theta}_\infty \right)\left( \zJS_{2}-\frac{\Theta_\infty}{q}\right)}\vspace{.2cm}\\
\zJS_{2} &=&\frac{1}{q}\frac{\left(\y_{1} -qt_0\Theta_t\right)\left(\y_{1} -qt_0\overline{\Theta}_t\right)}{\nu_1\left(\y_{1} -\overline{\Theta}_1\right)} \, .
\end{array}\right.$$
Substituting $\y_1=\Theta_1$ and the value of $\nu_1$ above 
yields 
$$\zJS_2= \frac{\left(\Theta_1 -qt_0\Theta_t\right)\left(\Theta_1 -qt_0\overline{\Theta}_t\right)}{\frac{(t_0{\Theta}_1- {\Theta}_t)(t_0\Theta_1-\overline{\Theta}_t)}{ \zJS_0}+q\Theta_1\left(\Theta_1 -\overline{\Theta}_1\right)\left(\left(\frac{\Theta_\infty}{q}+\overline{\Theta}_\infty\right)-t_0(\Theta_0+\overline{\Theta}_0)\right)}\, .$$
Hence $\zJS_2\in  \mathbb{P}^1$ is determined uniquely in terms of $\zJS_0$. Moreover, since $\y_1=\Theta_1$,  $\y_2$ is determined uniquely in terms of $\zJS_2$ by the recurrence relation. 
So in summary, if $(\y_0,\zJS_0)=(\overline{\Theta}_1,\zJS_0)$ with $\zJS_0\in \C^*$, then $(\y_1,\zJS_1)$ and $(\y_2,\zJS_2)$ are uniquely determined in terms of $\zJS_0$. Conversely, from $(\y_2,\zJS_2)$ as above we can recover $(\overline{\Theta}_1,\zJS_0)$ under the condition that $(t_0{\Theta}_1- {\Theta}_t)(t_0\Theta_1-\overline{\Theta}_t)$ is non-zero.

More generally,  \cite[Proposition 1]{sakai2001rational} ensures, in the case of sufficiently generic spectral data, the existence and uniqueness of discrete solutions with sufficiently generic prescribed initial data.
These genericity conditions will be made precise  in Section \ref{sec:qoka}. What we will obtain is the following, see Remark \ref{RemSols}. 
\begin{prop}\label{theo2} Let $q\in \C\setminus \{0,1\}$. Let $\boldsymbol{\Theta}=(\Theta_0,\Theta_1,\Theta_t,\Theta_\infty)\in (\C^*)^4$ such that $1\not \in \{\Theta_0^2, \Theta_1^2, \Theta_t^2\}$ and such that $\Theta_\infty^2\neq q$.  
Denote \begin{equation}\label{eq:defSq} S_q:=  \left\{ \, \Theta_1^{\varepsilon_1}\Theta_t^{\varepsilon_t}\, ,  \,   \Theta_0^{\varepsilon_0}\Theta_\infty^{\varepsilon_\infty} ~\middle|~\varepsilon_0,\varepsilon_1,\varepsilon_t,\varepsilon_\infty \in \{-1,1\}  \right\} \cdot q^\Z\, .\end{equation}
Let $$(\y_0,\zJS_0,t_0)\in \C^*\times \C^*\times \left(\C^*\setminus  S_q\right) \, .$$
Then, there exists a unique discrete solution $(  \y_{\ell} ,  \zJS_{\ell},t_\ell)_{\ell \in \Z}$  of  \eqref{qPVI} with initial value $(\y_0,\zJS_0,t_0)$.  \end{prop}

\subsection{Confluence}\label{sec:confPainl}
As usual for matters of confluence, in this section we will only consider spectral data $(\boldsymbol{\Theta}, \overline{\boldsymbol{\Theta}})$ related by $\Theta_i\overline{\Theta}_i=1$. 
We will first establish that the sixth Painlev\'e equation, up to the change of variable and spectral data that we previously found to be convenient for confluence,  admits a $q$-analogue of Hamiltonian formulation. From this, we will deduce the confluence of discrete and meromorphic solutions. 

\subsubsection{A $q$-analogue of Hamiltonian system}
  Let us apply the change of variable
\begin{equation}\label{eqzJSbZZ}
\zJS=\frac{(\y-t{\Theta}_t)(\y-t\overline{\Theta}_t)}{q(\y-1)(\y-t)(1+(q-1)\y\bZZ)}\, , \quad \bZZ=\frac{\frac{(\y-t{\Theta}_t)(\y-t\overline{\Theta}_t)}{q(\y-1)(\y-t)\zJS}-1}{(q-1)\y},
\end{equation}
to \eqref{qPVI}. The resulting equation, which is the simplification of \eqref{qPVIyZ} by the convention $\overline{\Theta}_i=1/\Theta_i$, is better adapted for questions of confluence. Indeed, roughly summarizing the result in Section~\ref{sec:confFuchsSys}, if a family of $q$-Fuchsian systems given by $(\boldsymbol{\lambda}, \y, \bZZ)$ discretizes a family of $\mathfrak{sl}_2$-Fuchsian systems given by $(\lambda, y, \ZZ)$, then we have $(\y, \bZZ)\to (y, \ZZ)$ as $q\to 1$. On the other hand, the implied estimate 
$(\y, \zJS)\to \left(y, \frac{y-t}{y-1}\right)$ as $q\to 1$ looses too much information. So we consider, for spectral data $\boldsymbol{\Theta}=(\Theta_0, \Theta_1, \Theta_t,\Theta_\infty)\in (\C^*)^4$, the system of $q$-difference equations 
\begin{equation} \label{qPVIyZbis}q\widetilde{P}_{\mathrm{VI}}(\boldsymbol{\Theta}) : \left\{ \begin{array}{rcl}
   \sigma_{q,t}\y &=&\displaystyle  \frac{1}{\y} \cdot \frac{\left(\frac{(\y-1)(\y-t)(1+(q-1)\y\bZZ)}{(\y-\Theta_1)(\y-\overline{\Theta}_1)}-t{\Theta_0}\right)\left(\frac{(\y-1)(\y-t)(1+(q-1)\y\bZZ)}{(\y-\Theta_1)(\y-\overline{\Theta}_1)}-t \overline{\Theta}_0\right)}{\left( \frac{(\y-1)(\y-t)(1+(q-1)\y\bZZ)}{(\y-\Theta_1)(\y-\overline{\Theta}_1)}-\frac{\Theta_\infty}{q}\right)\left( \frac{(\y-1)(\y-t)(1+(q-1)\y\bZZ)}{(\y-\Theta_1)(\y-\overline{\Theta}_1)}- \overline{\Theta}_\infty\right)} \vspace{.2cm}\\
    \sigma_{q,t}\bZZ &=& \displaystyle \frac{\frac{(\sigma_{q,t}\y-qt\Theta_t)(\sigma_{q,t}\y-qt\overline{\Theta}_t)(\y-\Theta_1)(\y-\overline{\Theta}_1)}{q( \sigma_{q,t}\y-1)( \sigma_{q,t}\y-qt) (\y-1)(\y-t)(1+(q-1)\y\bZZ) }-1}{(q-1) \sigma_{q,t}\y}
    \, ,
   \end{array} \right.\end{equation}
   where we denote $\overline{\Theta}_i:=1/\Theta_i$ for conciseness.

   The following result states that  this \emph{modified} $q$-Painlev\'e $\mathrm{VI}$ equation $q\widetilde{P}_{\mathrm{VI}}(\boldsymbol{\Theta})$ is an appropriate $q$-analogue of a Hamiltonian system.   First, let us introduce the Hamiltonian. To each datum $\boldsymbol{\theta}=(\theta_0,\theta_1,\theta_t,\theta_\infty)\in \C^4,$ we associate the rational function $H^{\boldsymbol{\theta}}_{\mathrm{VI}}\in \C(\y,\bZZ,t)$ in three variables given by
 \begin{equation}\label{HamilDefConf} \begin{array}{rcl}H^{\boldsymbol{\theta}}_{\mathrm{VI}}(\y,\bZZ,t)&:=&\frac{\y (\y -1)(\y -t)}{t(t-1)}\left(\bZZ^2+\frac{\bZZ}{\y -t}\right)-\frac{1}{4}\left(\frac{(\theta_\infty-1)^2-1}{t(t-1)}\y +\frac{\theta_0^2}{(t-1)\y }+\frac{\theta_t^2 }{\y -t}-\frac{\theta_1^2}{t(\y -1)}\right)\, .\end{array}\end{equation} 
   Note this is nothing else than the Hamiltonian for the differential case. In the following,  we  denote abusively $1+(q-1)\frac{\boldsymbol{\theta}}{2}:=\left(1+(q-1)\frac{\theta_0}{2},1+(q-1)\frac{\theta_1}{2} , 1+(q-1)\frac{\theta_t}{2},1+(q-1)\frac{\theta_\infty}{2}\right)$.

   \begin{thm}\label{thmHamilt} Let $\boldsymbol{\theta}\in \C^4$. Let $\mathcal{R}^{\boldsymbol{\theta}}_1,\mathcal{R}_2^{\boldsymbol{\theta}}\in \C(\y,\bZZ,t,q)$ be the (well-defined) rational functions in four variables
   such that the modified $q$-Painlev\'e $\mathrm{VI}$ equation \eqref{qPVIyZbis} with spectral data given by $$\boldsymbol{\Theta}:= 1+(q-1)\frac{\boldsymbol{\theta}}{2}$$ reads
$$q\widetilde{P}_{\mathrm{VI}}\left(\boldsymbol{\Theta} \right) : \left\{ \begin{array}{rcl}
  \partial_{q,t}\y &=& \phantom{+}\partial_{q,\bZZ}H^{\boldsymbol{\theta}}_{\mathrm{VI}}(\y,\bZZ,t)+  (q-1) \mathcal{R}^{\boldsymbol{\theta}}_1(\y,\bZZ,t,q) \vspace{.2cm}\\
 \partial_{q,t}\bZZ &=&  \, -\partial_{q,\y}H^{\boldsymbol{\theta}}_{\mathrm{VI}}(\y,\bZZ,t)+  (q-1) \mathcal{R}^{\boldsymbol{\theta}}_2(\y,\bZZ,t,q)
    \,    \end{array} \right.$$
    \emph{via} \eqref{HamilDefConf} and the operator  identity $1+t(q-1)\partial_{q,t}=\sigma_{q,t}$. Let $\mathcal{R}\in \{\mathcal{R}^{\boldsymbol{\theta}}_1, \mathcal{R}^{\boldsymbol{\theta}}_2\}$.
 The divisor $\{q=1\}$ in $\C^4_{\y,\bZZ,t,q}$ is not an irreducible component of the polar divisor of $\mathcal{R}$. 
 Moreover, the polar locus of the therefore well-defined rational function $\mathcal{R}|_{q=1}$ on $\C^3_{\y,\bZZ,t}$
 is contained in the set $$\mathcal{P}:=\{\y=0\}\cup\{\y=1\}\cup\{\y=t\}\cup \{t=0\}\cup\{t=1\}\, .$$
   \end{thm}
   
  \begin{proof} 
 Let us first say some words about the well-definedness of $\mathcal{R}^{\boldsymbol{\theta}}_1,\mathcal{R}_2^{\boldsymbol{\theta}}$. For quadrupels $\boldsymbol{\Theta}$, there are well defined rational functions $f,g\in \C(\y,\bZZ,t,q,\Theta_0,\Theta_1,\Theta_t,\Theta_\infty)$ such that $q\widetilde{P}_{\mathrm{VI}}(\boldsymbol{\Theta})$ can be written as
$$\left\{ \begin{array}{rcl}
   \sigma_{q,t}\y &=& f(\y,\bZZ,t,q,\boldsymbol{\Theta})  \vspace{.2cm}\\
    \sigma_{q,t}\bZZ &=&   g(\y,\bZZ,t,q,\boldsymbol{\Theta})
    \, .\end{array}\right.$$
    Indeed, it suffices to substitute the first equation in \eqref{qPVIyZbis} into the second, so that the right hand sides only depends on the variables $\y,\bZZ,t,q,\boldsymbol{\Theta}$. 
    Note that $\partial_{q,\bZZ}H^{\boldsymbol{\theta}}_{\mathrm{VI}}(\y,\bZZ,t)$ and $\partial_{q,\y}H^{\boldsymbol{\theta}}_{\mathrm{VI}}(\y,\bZZ,t)$ can easily be calculated and are elements of  $\C(\y,\bZZ,t,q)$. 
  Then   $$\left\{ \begin{array}{rcl}
   \mathcal{R}^{\boldsymbol{\theta}}_1(\y,\bZZ,t,q)&:=& \frac{1}{q-1}\left(\frac{f\left(\y,\bZZ,t,q,1+\frac{q-1}{2}\boldsymbol{\theta}\right)-\y}{(q-1)t} -\partial_{q,\bZZ}H^{\boldsymbol{\theta}}_{\mathrm{VI}}(\y,\bZZ,t)\right) \vspace{.2cm}\\
\mathcal{R}_2^{\boldsymbol{\theta}}(\y,\bZZ,t,q) &:=&  \frac{1}{q-1}\left(\frac{ g\left(\y,\bZZ,t,q,1+\frac{q-1}{2}\boldsymbol{\theta}\right)-\bZZ}{(q-1)t}+\partial_{q,\y}H^{\boldsymbol{\theta}}_{\mathrm{VI}}(\y,\bZZ,t)\right)
    \, \end{array}\right.$$ are indeed elements of $\C(\y,\bZZ,t,q)$ and are those required by the statement.  Let us define 
    $$\left\{ \begin{array}{rcl}
   \widetilde{\mathcal{R}}^{\boldsymbol{\theta}}_1(\y,\bZZ,t,q)&:=& \frac{1}{q-1}\left(\frac{f\left(\y,\bZZ,t,q,1+\frac{q-1}{2}\boldsymbol{\theta}\right)-\y}{(q-1)t} -\partial_{\bZZ}H^{\boldsymbol{\theta}}_{\mathrm{VI}}(\y,\bZZ,t)\right) \vspace{.2cm}\\
\widetilde{\mathcal{R}}_2^{\boldsymbol{\theta}}(\y,\bZZ,t,q) &:=&  \frac{1}{q-1}\left(\frac{ g\left(\y,\bZZ,t,q,1+\frac{q-1}{2}\boldsymbol{\theta}\right)-\bZZ}{(q-1)t}+\partial_{\y}H^{\boldsymbol{\theta}}_{\mathrm{VI}}(\y,\bZZ,t)\right).
    \, \end{array}\right.$$ 
  Since $\frac{\partial}{\partial \bZZ}H^{\boldsymbol{\theta}}_{\mathrm{VI}}$ and $\frac{\partial}{\partial \y}H^{\boldsymbol{\theta}}_{\mathrm{VI}}$ are rational functions of $(\y,\bZZ,t)$, these $\widetilde{\mathcal{R}}^{\boldsymbol{\theta}}_1,\widetilde{\mathcal{R}}^{\boldsymbol{\theta}}_2 $ are again rational functions in the variables $\y,\bZZ,t,q$. 
Denoting $\nabla_{\star}:= \frac{ \partial_{q,\star} - \partial_{\star} }{q-1}$, we have
${\mathcal{R}}^{\boldsymbol{\theta}}_1-\widetilde{\mathcal{R}}^{\boldsymbol{\theta}}_1= -\nabla_{\bZZ}(H^{\boldsymbol{\theta}}_{\mathrm{VI}})$,
 and ${\mathcal{R}}^{\boldsymbol{\theta}}_2-\widetilde{\mathcal{R}}^{\boldsymbol{\theta}}_2=\nabla_{\y}(H^{\boldsymbol{\theta}}_{\mathrm{VI}})$. In order to compute these differences, note that $H^{\boldsymbol{\theta}}$ is rational with only simple poles independent of $\bZZ$ and for $\star \in \{\y,\bZZ\}$, the operator $\nabla_{\star}$ is $\C(t)$-linear. So it suffices to compute $\nabla_{x} (x^{n})$ for $n\in \N$ and $\nabla_{x} (\frac{1}{x-a})$ for $a$ independent of $x$.
We have $\nabla_{x} (1)=0$, and for $n\in \N^*$, we find
$$  
\nabla_{x} (x^{n})=x^{n-1}\left(\frac{q^{n}-1}{(q-1)^{2}}-\frac{n}{q-1}\right)=x^{n-1}\left(\frac{\sum_{k=0}^{n-1} q^k -n}{q-1}\right)
=x^{n-1}\sum_{k=0}^{n-1}[k]_{q} \, , 
 $$
where $[k]_{q}=\frac{q^{k}-1}{q-1}=\sum_{i=0}^{k-1}q^i$ with $[0]_{q}=0$.
In particular, we find $\nabla_{x} (x)=0$, $\nabla_{x} (x^{2})=x$, $\nabla_{x} (x^{3})=(2+q)x^{2}$.
For $a$ independent of $x$, we find 
$$\nabla_{x} \left(\frac{1}{x-a}\right)=\frac{-1}{(qx-a)(x-a)(q-1)}+\frac{1}{(x-a)^{2}(q-1)} 
=\frac{x}{(qx-a)(x-a)^{2}}.$$
We deduce $$\large \left\{\begin{array}{rcl}
-\nabla_{\bZZ}(H^{\boldsymbol{\theta}}_{\mathrm{VI}}) &=&- \frac{\y (\y -1)(\y -t)}{t(t-1) }\bZZ\, , \vspace{.2cm}\\
\nabla_{\y}(H^{\boldsymbol{\theta}}_{\mathrm{VI}})&=& \frac{((q+2)\y-1-t)\bZZ+1}{t(t-1)}\y\bZZ-\frac{1}{4}\left(\frac{\theta_0^2}{q(t-1)\y^2}+\frac{\theta_t^2\y}{(\y-t)^2(q\y-t)}-\frac{\theta_1^2\y}{t(\y-1)^2(q\y-1)} \right)\, .
 \end{array}\right.$$
Obviously, these differences do not have $\{q=1\}$ as an irreducible component of their respective polar divisors, and  their restrictions to $q=1$ do not have poles outside $\mathcal{P}$. This means that 
the statement holds for $\mathcal{R}\in \{\mathcal{R}^{\boldsymbol{\theta}}_1, \mathcal{R}^{\boldsymbol{\theta}}_2\}$ if and only if it holds for $\mathcal{R}\in \{\widetilde{\mathcal{R}}^{\boldsymbol{\theta}}_1, \widetilde{\mathcal{R}}^{\boldsymbol{\theta}}_2\}$. But for the latter, we have already done most of the work. Indeed,
  the calculations in \textsection \ref{sub3confSchlesinger} at the end of Section \ref{sec:confSchlesinger} show that in restriction to any line $\{(\y,\bZZ,t)=(\y_0,\bZZ_0,t_0)\}\subset \C^4_{\y,\bZZ,t,q}$ with $(\y_0,\bZZ_0,t_0)\in \C^3\setminus \mathcal{P}$,
the two  rational functions $$\begin{array}{lll} \left.\left( (q-1) \widetilde{\mathcal{R}}^{\boldsymbol{\theta}}_1\right)\right|_{\{(\y,\bZZ,t)=(\y_0,\bZZ_0,t_0)\}}(q)&, &  \left.\left( (q-1) \widetilde{\mathcal{R}}^{\boldsymbol{\theta}}_2\right)\right|_{\{(\y,\bZZ,t)=(\y_0,\bZZ_0,t_0)\}}(q)
  \end{array}
 $$
 vanish both at $q=1$. It follows that $\{q=1\}$ is an irreducible component of the zero divisor of both $\left( (q-1) \widetilde{\mathcal{R}}^{\boldsymbol{\theta}}_2\right)$ and $\left( (q-1) \widetilde{\mathcal{R}}^{\boldsymbol{\theta}}_2\right)$. 
 In particular, $\{q=1\}$ is not an irreducible component of the polar divisor of  $\widetilde{\mathcal{R}}^{\boldsymbol{\theta}}_1 $ or $ \widetilde{\mathcal{R}}^{\boldsymbol{\theta}}_2$. 
 Moreover,  even though we did not push the Taylor series expansions in Section \ref{sec:confSchlesinger} far enough as to have an explicit expression for $\widetilde{\mathcal{R}}^{\boldsymbol{\theta}}_1|_{q=1} $ and $ \widetilde{\mathcal{R}}^{\boldsymbol{\theta}}_2|_{q=1}$, it is still clear from the calculations that these functions cannot have poles outside $\mathcal{P}$. The result follows.
  \end{proof}

\subsubsection{Confluence of discrete solutions}

We will now see that discrete solutions of the modified $q$-Painlev\'e $\mathrm{VI}$ equation yield holomorphic solutions of the differential Painlev\'e  $\mathrm{VI}$ equation by some limit process. The idea is that the successive $\partial_{q,t}$-derivations should lead, when $q\to 1$, to the coefficients of the Taylor series expansion of the limit functions. Let us consider the operator  $$\delta_{q,t}:=t\partial_{q,t}=\frac{\sigma_{q,t}-1}{q-1}\, .$$
As one can easily check, for each $n\in \N$, we have
\begin{equation}\label{eq:signtodeln} \delta_{q,t}^n=\frac{1}{(q-1)^n}\cdot \sum_{k=0}^n {n \choose k} (-1)^{n-k} \sigma_{q,t}^k\, .
\end{equation}
Here we use the convention $\delta_{q,t}^0=\sigma_{q,t}^0=1$.
We will prove the following. 
\begin{thm}\label{confDiscSol} Let $\boldsymbol{\theta}\in \C^4$. Let $t_0\in \C\setminus \{0,1\}$ and $(\y_0,\bZZ_0)\in \left(\C\setminus \{0,1,t_0\}\right)\times \C$. 
Then, there exists a family $(\y_n, \bZZ_n)_{n\in \N}$ of pairs of rational functions  $(\y_n, \bZZ_n) \in \C(q)\times \C(q)$ such that for generic values of $q$, 
the sequence $(\y_n(q), \bZZ_n(q),q^nt_0)_{n\in \N}$ is the positive part of the discrete 
solution of  $q\widetilde{P}_{\mathrm{VI}}\left(1+(q-1)\frac{\boldsymbol{\theta}}{2}\right)$ with initial value $( \y_0, \bZZ_0, t_0)$. 
Consider the sequence  $(a_n, b_n)_{n\in \N}$ of pairs of rational functions  $(a_n, b_n) \in \C(q)\times \C(q)$ defined by
$$\left\{\begin{array}{rcl} 
a_{n}(q)&=&\frac{1}{(q-1)^n}\cdot \sum_{k=0}^n {n \choose k} (-1)^{n-k} \y_k(q)
\vspace{.2cm}\\
b_n(q)&=& \frac{1}{(q-1)^n}\cdot \sum_{k=0}^n {n \choose k} (-1)^{n-k} \bZZ_k(q)\, .
\end{array}\right.$$
Then for each $n\in \N$, neither $a_n(q)$ nor $b_n(q)$ has a pole at $q=1$. Moreover,  the power series 
$$\sum_{n\geq 0}\frac{a_n(1)}{n!}(q-1)^n\, , \quad \sum_{n\geq 0}\frac{b_n(1)}{n!}(q-1)^n,$$
both converge and yield the pair of functions $q\mapsto (y(q\cdot t_0), Z(q\cdot t_0))$, where $(y(t), Z(t))$ is  the unique solution  of the Painlev\'e Hamiltonian system \eqref{eq:HamSys} with initial condition $(\y_0,\bZZ_0)$ at $t_0$. 
\end{thm}
This result will be proven by the end of this section. As it turns out, rather than trying to calculate the limit for $q\to 1$ for the $(a_n(q), b_n(q))$ directly, it is easier to first construct a particular sequence of rational functions that are finite at $q=1$ and then show that this sequence is actually the one from the statement. First we will need some general remarks. 

The $\delta_{q,t}$-operator on a field of functions with complex variable $t$  is additive and satisfies the following algebraic properties:
 \begin{equation}\label{eq1}\left\{\begin{array}{ccc}
\delta_{q,t}(fg)&=& \displaystyle(\delta_{q,t}f)g+f\delta_{q,t}g+(q-1) \delta_{q,t}( f) \delta_{q,t}(g),\vspace{.2cm}\\
\delta_{q,t}\left(\frac{1}{f}\right)&=& \displaystyle \frac{-\delta_{q,t}f}{f(f+(q-1)\delta_{q,t} f)} .
\end{array}\right.
\end{equation}
In particular, for any rational function $F\in\C(\y,\bZZ,t,q)$ we may define, by treating $\y$ and $\bZZ$ like functions of $t$, a rational function $\Delta_F$ with two additional variables such that
$$\delta_{q,t}F(\y,\bZZ,t,q)=\Delta_F\left(\y,\bZZ,t, q, \delta_{q,t}\y,\delta_{q,t}\bZZ\right) .$$
Let $\delta_{t}=t\partial_{t}$ that is the formal limit of $\delta_{q,t}$ when $q$ goes to $1$. The $q$-analogue of the chain rule that we will need is the following. 

\begin{lem}\label{lem1}
Let $F\in\C(\y,\bZZ,t,q)$ and $\Delta_F\in \C(\y,\bZZ,t, q, \delta_{q,t}\y,\delta_{q,t}\bZZ)$ be as above. 
Define $R_{F}\in \C(\y,\bZZ,t, q, \delta_{q,t}\y,\delta_{q,t}\bZZ)$ by 
$$R_{F}:=\frac{\Delta_F-\left(\frac{\partial F}{\partial \y}\cdot \delta_{t} \y+\frac{\partial F}{\partial \bZZ}\cdot \delta_{t} \bZZ+\frac{\partial F}{\partial t}\cdot t\right) }{q-1}\, .$$
If $\{q=1\}$ is not an irreducible component of the polar locus of $F$, then it is not an irreducible component of the polar locus of $\Delta_F$ and $R_{F}$. Moreover, when $F$ is seen as an element of $\C(\y,\bZZ,t, q, \partial_{q,t}\y,\partial_{q,t}\bZZ)$,  then the affine parts of the polar locus of $\Delta_F|_{q=1}$ and $R_{F}|_{q=1}$ are contained in the polar locus of $F|_{q=1}$.
\end{lem}
\begin{proof}
Let $P,Q \in \C[\y,\bZZ,t, q]$, $0\neq Q$, such that $F=P/Q$. We use \eqref{eq1} to compute successively  $\Delta_{Q^{-1}}$, $R_{Q^{-1}}$ and $\Delta_{P/Q}$, $R_{P/Q}$.
We have 
$$\Delta_{Q^{-1}}=\frac{-\Delta_Q}{Q(Q+(q-1)\Delta_Q)}=\frac{-\Delta_Q}{Q^{2}}+\frac{(q-1)\Delta_Q^{2}}{Q^{2}(Q+(q-1)\Delta_Q)} .$$
Then $(q-1)R_{Q^{-1}}$ is given by 
$$ \Delta_{Q^{-1}} +\frac{1}{Q^{2}}\left( \frac{\partial Q}{\partial \y} \cdot \delta_{t} \y +\frac{\partial Q}{\partial \bZZ}\cdot \delta_{t} \bZZ +\frac{\partial Q}{\partial t}\cdot t\right)=-\frac{(q-1)R_{Q}}{Q^{2}}+\frac{(q-1)\Delta_Q^{2}}{Q^{2}(Q+(q-1)\Delta_Q)}.$$
We have 
$$\begin{array}{rcl}
\Delta_F&=&\displaystyle \Delta_{P/Q}=\frac{\Delta_P}{Q}+P\Delta_{Q^{-1}}+(q-1)\Delta_P\Delta_{Q^{-1}}\vspace{.2cm}\\
&=&\displaystyle \frac{\Delta_P}{Q}+P\left(\frac{-\Delta_Q}{Q^{2}}+\frac{(q-1)\Delta_Q^{2}}{Q^{2}(Q+(q-1)\Delta_Q)}\right)-(q-1)\Delta_P \frac{\Delta_Q}{Q(Q+(q-1)\Delta_Q)} .
\end{array}$$
Finally, 
$$\begin{array}{rcl} R_{F}&=&\displaystyle R_{P/Q}=R_{P}Q^{-1}+PR_{Q^{-1}}+\Delta_P\Delta_{Q^{-1}}\vspace{.2cm}\\
&=&\displaystyle\frac{R_{P}}{Q}+P\left(\frac{-R_{Q}}{Q^{2}}+\frac{\Delta_Q^{2}}{Q^{2}(Q+(q-1)\Delta_Q)}\right)- \frac{\Delta_P\Delta_Q}{Q(Q+(q-1)\Delta_Q)} .
\end{array}$$
This proves that $\{q=1\}$ is not an irreducible component of the polar locus of $\Delta_F$ and $R_{F}$.
Note that  $\Delta_P|_{q=1},\Delta_Q|_{q=1},R_{P}|_{q=1},R_{Q}|_{q=1}\in \C[\y,\bZZ,t, \delta_{q,t}\y,\delta_{q,t}\bZZ]$. Hence the affine parts of the polar loci of $\Delta_F|_{q=1}$ and $R_{F}|_{q=1}$ are contained in the zero locus of $Q|_{q=1}$. This concludes the proof.
  \end{proof}
  Let $H_1,H_2\in \C(\y,\bZZ,t)$ and $R_1,R_2\in \C(\y,\bZZ,t,q)$ be rational functions in three and four complex variables respectively such that 
\begin{itemize}
\item[$\bullet$] the affine part of the polar locus of each of the functions $H_i$ with $i\in \{1,2\}$  is contained in the subset $\mathcal{P}\subset \C^3_{\y,\bZZ,t}$ given by $\mathcal{P}:=\{\y=0\}\cup\{\y=1\}\cup\{\y=t\}\cup \{t=0\}\cup\{t=1\}\, .$
\item[$\bullet$] for each $i\in \{1,2\}$,  the polar locus of $R_i$ does not contain $\{q=1\}$, and the affine part of the polar locus of $R_i|_{\{q=1\}}$ is contained in  $\mathcal{P}$. 
\end{itemize}
Consider the system of $q$-difference equations 
\begin{equation}\label{generalHamil} \left\{\begin{array}{rcl}\delta_{q,t}\y &=& H_1(\y,\bZZ,t)+  (q-1) R_1(\y,\bZZ,t,q) \vspace{.2cm}\\
 \delta_{q,t}\bZZ &=& H_2(\y,\bZZ,t)+  (q-1) R_2(\y,\bZZ,t,q)\, .\end{array}\right.
\end{equation}
Applying the  operator $\delta_{q,t}$ on both sides and then substituting the values of $\delta_{q,t}\y, \delta_{q,t}\bZZ$ imposed by this system yields a second order relation. 
There exist rational functions $R_1^{(1)},R_2^{(1)}\in \C(\y,\bZZ,t,q)$
such that this second order system is of the form
$$ \left\{\begin{array}{rcl}\delta_{q,t}^2\y &=& H_1^{(1)}(\y,\bZZ,t)+  (q-1) R_1^{(1)}(\y,\bZZ,t,q) \vspace{.2cm}\\
 \delta_{q,t}^2\bZZ &=& H_2^{(1)}(\y,\bZZ,t)+  (q-1) R_2^{(1)}(\y,\bZZ,t,q)\,  ,\end{array}\right.
$$
where $H_1^{(1)},H_2^{(1)}\in \C(\y,\bZZ,t)$ and $R_1^{(1)},R_2^{(1)}\in \C(\y,\bZZ,t,q)$  are given for $i\in \{1,2\}$ by 
$$\begin{array}{l}
H_i^{(1)}=\displaystyle\frac{\partial H_i}{\partial \y}\cdot \delta_{t} \y +\frac{\partial H_i}{\partial \bZZ}\cdot \delta_{t} \bZZ +\frac{\partial H_i}{\partial t}\cdot t \vspace{.2cm}
\\ R_i^{(1)}=\displaystyle \Delta_{R_i}(\y,\bZZ,t, q, H_1,H_2)+\frac{\Delta_{H_i}(\y,\bZZ,t, q, H_1,H_2)-\left(\frac{\partial H_i}{\partial \y}\cdot \delta_{t} \y +\frac{\partial H_i}{\partial \bZZ}\cdot \delta_{t} \bZZ  +\frac{\partial H_i}{\partial t}\cdot t\right) }{q-1}\,.\end{array}$$
By Lemma \ref{lem1}, for $i\in \{1,2\}$,  the polar locus of $R_i^{(1)}$ does not contain $\{q=1\}$, and the affine part of the polar locus of $R_i^{(1)}|_{\{q=1\}}$ is contained in  $\mathcal{P}$. 

We may apply this discussion to the modified $q$-Painlev\'e $\mathrm{VI}$ equation \eqref{qPVIyZbis} with spectral data given by $\boldsymbol{\Theta}:= 1+(q-1)\frac{\boldsymbol{\theta}}{2}$. 
Let us fix $\boldsymbol{\theta}\in \C^4$ and set 
$$\begin{array}{lll}H_1:=H_1^{(0)}:=\phantom{+}t\frac{\partial}{\partial \bZZ}H^{\boldsymbol{\theta}}_{\mathrm{VI}}&,&R_1:=t\widetilde{\mathcal{R}}^{\boldsymbol{\theta}}_1\, , \vspace{.2cm} \\
H_2:=H_2^{(0)}:=-t\frac{\partial}{\partial \y}H^{\boldsymbol{\theta}}_{\mathrm{VI}}&,&R_2:=t\widetilde{\mathcal{R}}^{\boldsymbol{\theta}}_2\, ,\end{array}$$
where $H^{\boldsymbol{\theta}}_{\mathrm{VI}}(\y,\bZZ,t)$ is given by \eqref{HamilDefConf} and where for $i\in \{1,2\}$, $\widetilde{\mathcal{R}}^{\boldsymbol{\theta}}_i(\y,\bZZ,t,q)$ is as in the proof of Theorem \ref{thmHamilt}. 
With this convention, by Theorem \ref{thmHamilt}, the modified $q$-Painlev\'e $\mathrm{VI}$ equation  with spectral data  $1+(q-1)\frac{\boldsymbol{\theta}}{2}$ is, when expressed with respect to the $\delta_{q,t}$-operator, given by \eqref{generalHamil}. Moreover, by induction on $n\in \N$, the associated system of order $n+1$ is of the form
$$ \left\{\begin{array}{rcl}\delta_{q,t}^{n+1}\y &=& H_1^{(n)}(\y,\bZZ,t)+  (q-1) R_1^{(n)}(\y,\bZZ,t,q) \vspace{.2cm}\\
 \delta_{q,t}^{n+1}\bZZ &=& H_2^{(n)}(\y,\bZZ,t)+  (q-1) R_2^{(n)}(\y,\bZZ,t,q)\,  ,\end{array}\right.
$$
for some well defined rational functions $H_1^{(n)},H_2^{(n)}\in \C(\y,\bZZ,t)$ and $R_1^{(n)},R_2^{(n)}\in \C(\y,\bZZ,t,q)$ such that
\begin{itemize}
\item[$\bullet$] $H_i^{(n)}=\frac{\partial H_i^{(n-1)}}{\partial \y}\cdot \delta_{t} \y+\frac{\partial H_i^{(n-1)}}{\partial \bZZ}\cdot \delta_{t} \bZZ  +\frac{\partial H_i^{(n-1)}}{\partial t}\cdot t$ and
\item[$\bullet$]  the polar locus of $R_i^{(n)}$ does not contain $\{q=1\}$, and the affine part of the polar locus of $R_i^{(n)}|_{\{q=1\}}$  is contained in  $\mathcal{P}$. 
\end{itemize}
Now let us fix
$$t_0\in \C\setminus \{0,1\}\, , \quad (\y_0,\bZZ_0)\in  \left(\C\setminus \{0,1, t_0\}\right)\times \C\, .$$
Note that if the line $\{(\y,\bZZ,t)=(\y_0,\bZZ_0,t_0)\} \subset \C^4_{\y,\bZZ,t,q}$ were contained in the polar divisor of $R_i^{(n)}$ for some $n$, then $R_i^{(n)}|_{q=1}$ would have a pole at $(\y_0,\bZZ_0,t_0)$. But this cannot happen because $(\y_0,\bZZ_0,t_0)\not \in \mathcal{P}$. 
Therefore, we may define a sequence of pairs of rational functions $( \widetilde{a}_n,\widetilde{b}_n)_{n\in \N}\in (\C(q)\times \C(q))^{\N}$ as follows. We take the initial functions to be the constants  $\widetilde{a}_0(q):=\y_0,\widetilde{b}_0(q):=\bZZ_0$, and 
for $n\in \N$, we set 
$$ \left\{\begin{array}{rcl}\widetilde{a}_{n+1}(q) &=& H_1^{(n)}(\y_0,\bZZ_0,t_0)+  (q-1) R_1^{(n)}(\y_0,\bZZ_0,t_0,q) \vspace{.2cm}\\
\widetilde{b}_{n+1}(q) &=& H_2^{(n)}(\y_0,\bZZ_0,t_0)+  (q-1) R_2^{(n)}(\y_0,\bZZ_0,t_0,q)\,  .\end{array}\right.
$$
Note that by construction, for each $n\in \N$,  the pair $(\widetilde{a}_n,\widetilde{b}_n)$ is well defined and finite when evaluated at $q=1$. 
If $(y(t), Z(t))$ is the unique solution of the Painlev\'e Hamiltonian system \eqref{eq:HamSys} with initial condition $(\y_0,\bZZ_0)$ at $t_0$, then 
its successive derivations with respect to the differential operator $\delta_t$ satisfy precisely $\left((\delta_t^{n+1}y)(t_0), (\delta_t^{n+1}Z)(t_0)\right)=(H_1^{(n)}(\y_0,\bZZ_0,t_0), H_2^{(n)}(\y_0,\bZZ_0, t_0))$. 
On the other hand, the successive $\delta_{t}$-derivatives of $(y(t),Z(t))$ evaluated at $t_0$ coincide with the evaluation at $q=1$ of the the successive $\partial_q$-derivatives of  the holomorphic functions  
$$q\mapsto y(q\cdot t_0) \, , \quad q\mapsto Z(q\cdot t_0)\, .$$
Therefore, the power series
$\sum_{n\geq 0}\frac{\widetilde{a}_n(1)}{n!}(q-1)^n\, , \quad \sum_{n\geq 0}\frac{\widetilde{b}_n(1)}{n!}(q-1)^n$
both converge and yield the pair of functions $q\mapsto (y(qt_0),Z(qt_0))$, where $(y(t), Z(t))$ is this unique solution. It remains to prove that the sequence $(\widetilde{a}_n,\widetilde{b}_n)_{n\in \N}$ coincides with  the $(a_n,b_n)_{n\in \N}$ defined in the statement of Theorem \ref{confDiscSol}.  
On the other hand, since $\sigma_{q,t}=1+(q-1)\delta_{q,t}$,  for each $n\in \N$
we have 
\begin{equation}\label{eq:sigfromdelta}\sigma_{q,t}^n=(1+(q-1)\delta_{q,t})^{n}=\sum_{k=0}^n{n \choose k}(q-1)^k\delta_{q,t}^k\, .\end{equation}
We are therefore inclined to define the sequence $(\y_n,\bZZ_n)_{n\in \N}\in (\C(q)\times \C(q))^{\N}$ given by 
$$\left\{\begin{array}{rcl} 
\y_n(q)&=&\sum_{k=0}^n{n \choose k}(q-1)^k \widetilde{a}_k(q)
\vspace{.2cm}\\
\bZZ_n(q)&=&\sum_{k=0}^n{n \choose k}(q-1)^k\widetilde{b}_k(q)\, .
\end{array}\right.$$

By construction, the elements of the sequence $(\y_n,\bZZ_n,q^nt_0)_{n\in \N}\in (\C(q)\times\C(q)\times \C(q))^{\N}$ are related to $(\y_0,\bZZ_0, t_0)$ by the same rational relation as those of a discrete solution with initial value $(\y_0,\bZZ_0, t_0)$ of the modified $q$-Painlev\'e $\mathrm{VI}$ equation \eqref{qPVIyZbis} with spectral data $\boldsymbol{\Theta}= 1+(q-1)\frac{\boldsymbol{\theta}}{2}$. 
Therefore, the sequence of rational functions $(\y_n,\bZZ_n,q^nt_0)_{n\in \N}$ \emph{is} the (positive part of) the solution with initial value $(\y_0,\bZZ_0, t_0)$, seen as a rational function of the variable $q$.\par  
The process in \eqref{eq:sigfromdelta} allowing to recover $\sigma_{q,t}^n$ from $1,\delta_{q,t}, \dots, \delta_{q,t}^n$ is inverse to the process in \eqref{eq:signtodeln} allowing to recover $\delta_{q,t}^n$ from $\sigma_{q,t}^{0}, \dots, \sigma_{q,t}^n$.  Therefore, the sequence $(\widetilde{a}_n,\widetilde{b}_n)_{n\in \N}$ constructed above coincides with  the $(a_n,b_n)_{n\in \N}$ defined in the statement of Theorem \ref{confDiscSol}.   This concludes the proof of Theorem \ref{confDiscSol}. 

\begin{rem}\label{remdiffsols}Note that with respect to the notation in Theorem \ref{confDiscSol}, for each $n\in \N$, as $q\to 1$, we have
$$\y_{n}(q)-y(q^{n}t_0)=O(q-1)\quad \quad \hbox{ and }\quad \quad \bZZ_{n}(q)-Z(q^{n}t_0)=O(q-1).$$
Indeed, we have $$\y_n(q)=\sum_{k=0}^n{n\choose k} (q-1)^k a_k(q)=a_0(q)+O(q-1)=y(t_0)+O(q-1)= y(q^n t_0)+O(q-1)\, .$$
 The argument for $\left(\bZZ_n(q)-Z(q^nt_0)\right)=O(q-1)$ is identical. 
\end{rem}
\begin{cor}Let $\boldsymbol{\theta}\in \C^4$. Let $\qDom \subset \C\setminus \{0,1\}$ be a subset with $1$ in its closure. Let $t_0\in \C\setminus \{0,1\}$. Let $(\y_0(q),\bZZ_0(q))$ be a pair of continuous functions 
$\qDom\to \C$ such that the limit 
$$(y_0,Z_0):=\lim_{\begin{smallmatrix} q\to 1 \\ q\in \qDom\end{smallmatrix}}(\y_0(q),\bZZ_0(q))$$
exists in $\C^2$ and satisfies $y_0\not \in \{0,1,t_0\}$. Then, there exists a family $(\y_n, \bZZ_n)_{n\in \N} \in \C(\y_0(q),q) \times \C(\bZZ_0(q),q)$ of pairs of continuous functions such that  for generic values of $q$, 
the sequence $(\y_n(q), \bZZ_n(q),q^nt_0)_{n\in \N}$ is the positive part of the discrete 
solution of  $q\widetilde{P}_{\mathrm{VI}}\left(1+(q-1)\frac{\boldsymbol{\theta}}{2}\right)$ with initial value $( \y_0 (q), \bZZ_0 (q), t_0)$. 
Moreover, for each $n\in \N$, as $q\to 1$, we have 
$$\y_n(q)-\y_0(q)=O(q-1)\, , \quad \quad \quad \bZZ_n(q)-\bZZ_0(q)=O(q-1)\,.$$
\end{cor}
\begin{proof} In the proof of Theorem  \ref{confDiscSol}, we may replace $(\y_0,\bZZ_0)$ by $(\y_0(q),\bZZ_0(q))$ in the definition of $(a_n(q),b_n(q))_{n\in \N}$. Note that the latter then is a sequence of pairs of rational functions, each evaluated in a pair of continuous functions in $q$ which, as $q\to 1$, admit a finite limit which is not in the polar locus of the restriction to $q=1$ of these rational functions.  
We conclude that for each $n\in \N$, the pair $(a_n(q),b_n(q))$ may be continued to a continuous function on $\qDom\cup \{1\}$ with finite value at $q=1$. Moreover, as before, we have the relation  
$$\begin{array}{rcl} (\y_n(q)\, , \bZZ_n(q)&=&\displaystyle \sum_{k=0}^n{n \choose k}(q-1)^k(a_k(q),b_k(q)) = (a_0(q),b_0(q))+O(q-1)\vspace{.2cm}\\
&=&(\y_0(q), \bZZ_{0}(q))+O(q-1)\, .\end{array}$$    \end{proof}

 \section{Okamoto's space of initial conditions}\label{sec:4}
 \subsection{Differential case}\label{sec:oka}
 Let us review the construction in \cite{okamoto1986studies} of a convenient space of initial conditions for the Painlev\'e ${\mathrm{VI}}$ equation, and recall why it proves the Painlev\'e property. 
 Let $\boldsymbol{\theta}=(\theta_0,\theta_1,\theta_t,\theta_\infty)\in \C^4$. Recall the Hamiltonian system (\ref{eq:HamSys}) associated to the Painlev\'e ${\mathrm{VI}}$ differential equation with spectral data $\boldsymbol{\theta}$:
$$
\left\{ \begin{array}{rcll}
y'(t)&=& &\frac{\partial H^{\boldsymbol{\theta}}_{\mathrm{VI}}}{\partial \ZZ}(y,\ZZ, t)\vspace{.2cm}
\\
\ZZ'(t)&=&-&\frac{\partial H^{\boldsymbol{\theta}}_{\mathrm{VI}}}{\partial y}(y,\ZZ, t)\, .\end{array}\right.$$
Let us fix a time $t_0\in \C\setminus \{0,1\}$. Recall from Section \ref{sec:diffP6}  that for any initial value $(y_0,Z_0)\in (\C\setminus \{0,1,t_0\})\times \C$, there exists a unique germ of holomorphic solution $(y(t),\ZZ(t))$ of  (\ref{eq:HamSys}) such that $(y(t_0),\ZZ(t_0))=(y_0,Z_0)$. Since we have $y_0\not \in \{0,1,t_0\}$ here, we may equivalently consider the space of initial conditions $(u_0,v_0)\in (\C\setminus \{0,1,t_0\})\times \C$, 
where we identify \begin{equation}\label{uvfromyZ}(u_0,v_0)=(y_0,y_0(y_0-1)(y_0-t_0)Z_0)\, .\end{equation}
We compactify this space of initial values to the second Hirzebruch surface $\mathbb{F}_2$, using the following coordinate charts of $\C^2$-spaces, endowed with their obvious rational transition maps:
\begin{align*}
&(u_0 ,v_0 )=:(u,v),&(u_1,v_1)&=\left(u,\frac{1}{v}\right),\\
&(u_2,v_2)=\left(\frac{1}{u},\frac{v}{u^2}\right), & (u_3,v_3)&=\left(\frac{1}{u},\frac{u^2}{v}\right).
\end{align*}
Here what we have added by the compactification is the union of the horizontal line $$ {\mathcal{H}}:=\{v_1=0\}\cup \{v_3
=0\}$$ and the four vertical lines given by $$\mathcal{D}_i:=\{u_0=i\}\cup \{u_1=i\}\quad \textrm{for}\quad i \in \{0,1,t_0\}\, , \quad  \mathcal{D}_\infty:=\{u_2=0\}\cup \{u_3=0\}\, .$$  Now the Hamiltonian system (\ref{eq:HamSys}) defines a meromorphic vector field on $ \mathbb{F}_2\times(\C\setminus \{0,1\})$. Explicitly, it is given with respect to the coordinates $(u,v, t)$ by 
$$  \left\{\begin{array}{ccl}
  t'&=&1\vspace{.2cm}\\
u ' &=& \frac{1}{t(t-1)}\left(2v +u (u -1)\right) \vspace{.2cm}\\
v ' &=& \frac{1}{4}\left( \frac{4v ^2-t^2\theta_{0} ^2}{t(t-1)u }+\frac{4v ^2- (t-1)^2\theta_1^2}{t(t-1)(u -1)}+\frac{4v ^2-t^2(t-1)^2\theta_t^2}{t(t-1)(u -t)}\right)+\frac{\theta_\infty(\theta_\infty-2)u (u -1)(u -t)}{4t(t-1)} \vspace{.2cm}\\
&+&\frac{1}{4}\left(- \frac{\theta_{0} ^2(u -1-t)}{t-1} +\frac{\theta_1^2(u +1-t)}{t}-\frac{(t(t-1)\theta_t^2-4v )(u -1+t)}{t(t-1)}  \right)\, . 
\end{array}\right.$$
 One realizes that the vector field is infinite on the set given for each fixed $t=t_0$ by $$ \mathcal{H}\cup \mathcal{D}_0\cup \mathcal{D}_1\cup \mathcal{D}_t\cup \mathcal{D}_\infty.$$ More precisely, it is infinite or undetermined (of the form ``$\frac{0}{0}$'') precisely there. These indeterminacy points will be called \emph{base points} in the following. If we assume that 
 \begin{equation}\label{cond8bp}\theta_0\neq 0\, ,\quad  \theta_1\neq 0\, , \quad \theta_t\neq 0\, , \quad \theta_\infty \neq 1\, , \end{equation}
 then there are precisely eight such base points. With respect to the four charts of $\mathbb{F}_2$, these base points, each possibly visible in several charts, are precisely the following:
 $$\begin{array}{ | c | c |  c  | c | c | }
\hline 
 \begin{array}{c} \vspace{.3cm} \end{array} & ( u_0 , v_0) & ( u_1, v_1) & ( u_2, v_2) & ( u_3, v_3) \\
\hline 
 \begin{array}{c} \vspace{.4cm} \end{array} \beta_0^\pm & \left(0, \pm\frac{t\theta_0}{2}\right)&\left(0, \pm\frac{2}{t\theta_0}\right) & &   \\
\hline 
 \begin{array}{c} \vspace{.4cm} \end{array} \beta_1^\pm &  \left(1, \pm\frac{(t-1)\theta_1}{2}\right) &   \left(1, \pm\frac{2}{(t-1)\theta_1}\right) &    \left(1, \pm\frac{(t-1)\theta_1}{2}\right)  &   \left(1, \pm\frac{2}{(t-1)\theta_1}\right)  \\
\hline 
 \begin{array}{c} \vspace{.4cm} \end{array} \beta_t^\pm  &  \left(t, \pm\frac{t(t-1)\theta_t}{2}\right)  & \left(t, \pm\frac{2}{t(t-1)\theta_t}\right)   &   \left(\frac{1}{t}, \pm\frac{(t-1)\theta_t}{2t}\right)   & \left(\frac{1}{t}, \pm\frac{2t}{(t-1)\theta_t}\right)   \\
\hline 
 \begin{array}{c} \vspace{.4cm} \end{array} \beta_\infty^+ & & & \left(0, \frac{\theta_\infty-2}{2}\right) &   \left(0, \frac{2}{\theta_\infty-2}\right)
\\ \hline
 \begin{array}{c} \vspace{.4cm} \end{array} \beta_\infty^- & & & \left(0, -\frac{\theta_\infty}{2} \right) &  \left(0, -\frac{2}{ {\theta}_\infty}\right)
\\ \hline
\end{array}$$
In the following discussion, we assume \eqref{cond8bp}. 
For fixed $t$, the Hirzebruch surface, as well as the configuration of particular lines and base points, are resumed in the following picture. Here ``$(n)$'' indicates ``self-intersection number equal to $n$''.
  
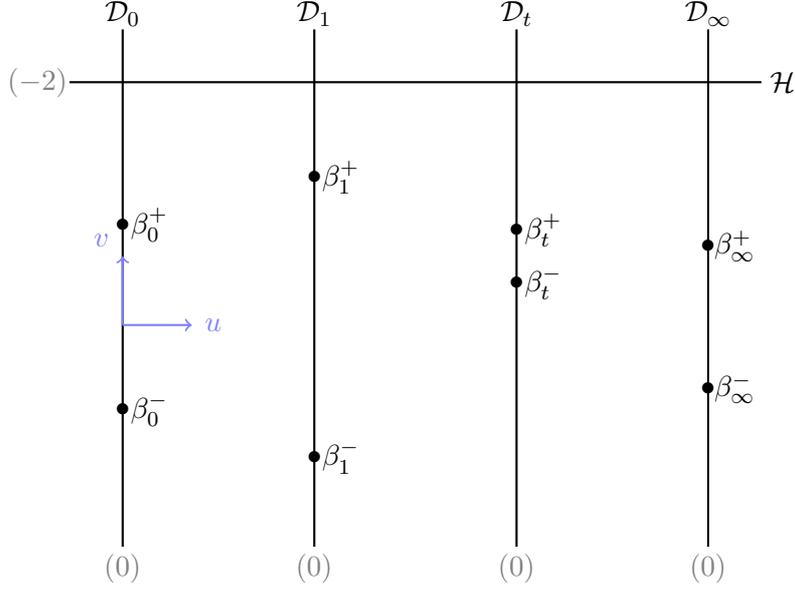
\begin{center}
\begin{tikzpicture}[scale=0.7]

\draw[thick] (-6.5,4.3)--(6.5,4.3);
  \draw (6.9,4.3) node {$\mathcal H$} ;
    \draw[gray] (-7.1,4.3) node {$(-2)$} ;

\draw[thick] (5.5,-4.5)--(5.5,5.3);
  \draw (5.5,5.6) node {$\mathcal{D}_\infty$} ;  \draw[gray] (5.5,-4.9) node {$(0)$} ;
  \draw[thick] (1.9,-4.5)--(1.9,5.3);
  \draw (1.9,5.6) node {$\mathcal{D}_t$} ; \draw[gray] (1.9,-4.9) node {$(0)$} ;
    \draw[thick] (-1.9,-4.5)--(-1.9,5.3);
  \draw (-1.9,5.6) node {$\mathcal{D}_1$} ; \draw[gray] (-1.9,-4.9) node {$(0)$} ;
      \draw[thick] (-5.5,-4.5)--(-5.5,5.3);
  \draw (-5.5,5.6) node {$\mathcal{D}_0$} ;  \draw[gray] (-5.5,-4.9) node {$(0)$} ;
 
    \draw (-5.5,1.6) node {$\bullet$} ;   \draw (-5,1.6) node {$\beta_0^+$} ;
     \draw (-5.5,-1.9) node {$\bullet$} ; \draw (-5,-1.9) node {$\beta_0^-$} ;

       \draw (-1.9,2.5) node {$\bullet$} ;      \draw (-1.4,2.5) node {$\beta_1^+$} ;
     \draw (-1.9,-2.8) node {$\bullet$} ; 	  \draw (-1.4,-2.8) node {$\beta_1^-$} ;

       \draw (1.9,0.5) node {$\bullet$} ;		    \draw (2.4,0.5) node {$\beta_t^-$} ;
     \draw (1.9,1.5) node {$\bullet$} ;			 \draw (2.4,1.5) node {$\beta_t^+$} ;
     
      \draw (5.5,1.2) node {$\bullet$} ;			  \draw (6,1.2) node {$\beta_\infty^+$} ;
     \draw (5.5,-1.5) node {$\bullet$} ;			    \draw (6,-1.5) node {$\beta_\infty^-$} ;

                \draw[blue!50, thick, ->] (-5.5,-0.3)--(-5.5,1);
            \draw[blue!50, thick, ->] (-5.5,-0.3)--(-4.2,-0.3);
               \draw[blue!50] (-3.8 , -0.3)  node  {$u$} ;
         \draw[blue!50] (-5.9 , 1.3)  node  {$v$} ;
         \end{tikzpicture}
\end{center}
For any fixed $t$, let us denote by  $\widehat{\mathbb{F}}_2^t$ the result of the above Hirzebruch surface ${\mathbb{F}}_2$ after blow up of the the eight base points. For each $i\in \{0,1,t,\infty\}$, we denote by $\mathcal{D}_i^{**}$ the strict transform of $\mathcal{D}_i$ after blow up of $\beta_i^\pm$, \emph{i.e.} the closure of $\mathcal{D}_i\setminus \{\beta_i^\pm\}$ in $\widehat{\mathbb{F}}_2^t$. Note that each $\mathcal{D}_i^{**}$ has self-intersection number $-2$.  The \emph{Okamoto space of initial values} at the time $t$ for the sixth Painlev\'e equation with spectral data $\boldsymbol{\theta}$ is by definition 
 $$\mathrm{Oka}_t:=\widehat{\mathbb{F}}_2^t\setminus \mathcal{I}^t\, , \quad \textrm{where} \quad \mathcal{I}^t:= \mathcal{D}_0^{**}\cup \mathcal{D}_1^{**}\cup \mathcal{D}_t^{**}\cup \mathcal{D}_\infty^{**}\cup \mathcal{H}\, .$$
 For example in order to blow up $\beta_0^-$, one replaces a neighborhood of $\beta_0^-$ containing none of the other seven base points, by the corresponding neighborhood in the spaces 
 $\C^2_{u_{01},v_{01}}$ and  $\C^2_{u_{02},v_{02}}$ related to $\C^2_{u,v}$ according to the following transition maps:
  $$\begin{array}{rclcrcl}(u_{01},v_{01})&=&\left(u, \frac{2v+t\theta_0}{2u} \right) ,
&\quad 
&(u_{02},v_{02})&=&\left(\frac{2u}{2v+t\theta_0}, v+\frac{t\theta_0}{2}\right).
\end{array}$$
Note that $(u_{02},v_{02})=\left(\frac{1}{v_{01}},u_{01}v_{01}\right)$ whenever $v_{01}\neq 0$. 
In these two new charts, what in $\C^2_{u,v}$ was the point $\beta_0^-$ now corresponds to the exceptional line, isomorphic to $\mathbb{P}^1$, given by $$\mathcal{E}_0^-:=\{u_{01}=0\}\cup \{v_{02}=0\}\, .$$ 
The complementary of $\mathcal{E}_0^-$ however is in biholomorphic correspondence with the corresponding open subset of $\C^2_{u,v}$. 
The vector field in these new charts is given as follows: $$
 \left\{\begin{array}{ccl}
u_{01}' &=& \frac{1}{t(t-1)}\left(2\left(u_{01}v_{01}-\frac{t\theta_0}{2}\right)+u_{01}(u_{01}-1)\right)  \vspace{.2cm}\\
v_{01}' 
 &=&-\frac{1}{t(t-1)}\left(v_{01}^2 -v_{01}+\frac{t\theta_0}{2}\right)  - \frac{t+1 }{t^2(t-1)}\left(u_{01}v_{01}-t\theta_0\right) v_{01} +\frac{v_{01}}{t}+\frac{1}{4}\left(-\frac{\theta_0^2}{t-1}+\frac{\theta_1^2}{t}-\theta_t^2  \right) 
\vspace{.2cm}\\&+&\frac{1}{4}\left(  \frac{\left(2u_{01}v_{01}- t\theta_0 \right)^2- (t-1)^2\theta_1^2}{t(t-1)(u_{01}-1)}+ \frac{\left(2u_{01}v_{01}- t\theta_0 \right)^2-t^2(t-1)^2\theta_t^2}{t^2(t-1)(u_{01}-t)}+ \frac{\theta_\infty(\theta_\infty-2)(u_{01}-1)(u_{01}-t)}{t(t-1)}\right) \, , 
\end{array}\right.
$$
$$
 \left\{\begin{array}{ccl}
u_{02}' &=&\frac{1}{t(t-1)}\left(1-u_{02}+\frac{t\theta_0}{2 }u_{02}^2\right)  + \frac{t+1 }{t^2(t-1)}\left(v_{02}-t\theta_0\right) u_{02} -\frac{u_{02}}{t}-\frac{u_{02}^2}{4}\left(-\frac{\theta_0^2}{t-1}+\frac{\theta_1^2}{t}-\theta_t^2  \right)
\vspace{.2cm}\\&-&\frac{u_{02}^2}{4}\left(  \frac{\left(2v_{02}- t\theta_0 \right)^2- (t-1)^2\theta_1^2}{t(t-1)(u_{02}v_{02}-1) }+ \frac{\left(2v_{02}- t\theta_0 \right)^2-t^2(t-1)^2\theta_t^2}{t^2(t-1)(u_{02}v_{02}-t)}+\frac{\theta_\infty(\theta_\infty-2)(u_{02}v_{02}-1)(u_{02}v_{02}-t)}{t(t-1)} \right)  \vspace{.2cm}\\
v_{02}'&=&  \frac{ v_{02}-t\theta_0 }{t(t-1)u_{02}}+\frac{1}{4}\left(\frac{\left(2v_{02}-t\theta_0\right)^2- (t-1)^2\theta_1^2}{t(t-1)(u_{02}v_{02}-1)}+\frac{\left(2v_{02}-t\theta_0\right)^2-t^2(t-1)^2\theta_t^2}{t(t-1)(u_{02}v_{02}-t)}\right) \vspace{.2cm}\\&+&\frac{\theta_\infty(\theta_\infty-2)u_{02}v_{02}(u_{02}v_{02}-1)(u_{02}v_{02}-t)}{4t(t-1)} \vspace{.2cm}\\
&+&\frac{1}{4}\left(- \frac{\theta_0^2(u_{02}v_{02}-1-t)}{t-1} +\frac{\theta_1^2(u_{02}v_{02}+1-t)}{t}-\frac{(t(t-1)\theta_t^2-2\left(2v_{02}- t\theta_0\right))(u_{02}v_{02}-1+t)}{t(t-1)}  \right) +\frac{\theta_0}{2}\, .
\end{array}
\right. 
$$

We see that there is no base point on $\mathcal{E}_0^-=\{u_{01}=0\}\cup \{v_{02}=0\}$. In other words, the blow up was sufficient to \emph{resolve} the base point $\beta_0^-$. Moreover, we see that the vector field is finite on all points of $\mathcal{E}_0^-$ except the one given by the intersection with the strict transform of $\mathcal{D}_0$, which is  visible in these charts as $\{u_{02}=0\}$. 
As shown in \cite{okamoto1986studies}, this actually holds true for all eight base points, \emph{i.e.} on $\mathrm{Oka}_t$, the vector field is everywhere finite and free of base points. 

As explained in \cite{loray2004theoremes}, this situation can be conveniently reformulated as follows. On $ \mathbb{F}_2\times (\C \setminus \{0,1\})$, the meromorphic vector field given by  (\ref{eq:HamSys}) defines a singular holomorphic foliation. The singular locus corresponds to the eight families (parametrized by $t\in \C\setminus \{0,1\}$) of base points. After blowing up the singular locus, the induced foliation on $$\bigcup_{t\in \C\setminus \{0,1\}}\widehat{\mathbb{F}}_2^t$$ is non-singular. Moreover, it is transversal to $\{t=cst\}$ on the complementary of 
$\bigcup_{t\in \C\setminus \{0,1\}} \mathcal{I}^t$, \emph{i.e.} on $$\mathrm{Oka}:=\bigcup_{t\in \C\setminus \{0,1\}}\mathrm{Oka}_t\, . $$  This implies that for any $t_0\in   \C\setminus \{0,1\}$ and any initial condition given by a point in $\mathrm{Oka}_{t_0}$, in turn given by a point in some chart $\C^2_{u_{ij},v_{ij}}$ of $\widehat{\mathbb{F}}_2^{t_0}$,  one obtains a unique germ of holomorphic integral curve of the form 
$(u_{ij}(t),v_{ij}(t),t)$. Translated back into the variables $(y,Z)$, this yields a meromorphic solution $y$ of the sixth Painlev\'e equation, associated to this initial condition. Moreover, since  $(u_{ij}(t),v_{ij}(t),t)$ parametrizes a germ of leaf of the Painlev\'e foliation on $\mathrm{Oka}$, and this foliation is transversal to $\{t=cst\} $, this parametrization of a germ of leaf can be analytically continued in $\mathrm{Oka}$ along any path  with starting point $t_0$. In other words, the  meromorphic solution $y$ of the Painlev\'e equation can be meromorphically continued along any path in $ \C\setminus \{0,1\}$ with starting point $t_0$.  As for the usual analytic continuation, this meromorphic continuation has no reason to be uniform. But 
we obtain a well-defined meromorphic function on every simply connected subset of $\C\setminus \{0,1\}$ containing $t_0$. 
 This phenomenon, which is also observed for the other five Painlev\'e equations, is also known as the \emph{Painlev\'e  property} of solutions of Painlev\'e  equations. 
\\
Let us illustrate the above by an example. Consider a germ of solution $y$ of $P_{\mathrm{VI}}$, associated to an initial condition on $\mathcal{E}_0^-\setminus \mathcal{D}_0^{**}$ at $t=t_0$. That is, we have $(u_{01},v_{01})(t_0)=(0,\alpha)$ for some $\alpha \in \C$. 
 Hence $y=u_{01}$ is holomorphic near $t_0$. From the explicit formula of the vector field, we readily calculate the first terms of the Taylor series expansion of  $y$:
$$y(t)= -\frac{\theta_0}{(t_0-1)}(t-t_0)-\frac{\theta_0(2\alpha-1-t_0)}{2t_0(t_0-1)^2}(t-t_0)^2+O(t-t_0)^3\, .$$
Note that $y$ has a simple zero at $t_0$, and its Taylor series expansion at $t_0$ up to order two is uniquely determined by $\alpha$.

  \subsection{A discrete analogue}\label{sec:qoka} Let us first review the construction in \cite{sakai2001rational} of a convenient $q$-analogue of the Okamoto space for the $q$-Painlev\'e $\mathrm{VI}$ equation, which proves the existence of discrete solutions, et then adapt this discussion to our modified $q$-Painlev\'e $\mathrm{VI}$ equation. 
  
  Let $q\in \C\setminus \{0,1\}$. 
 Let $\boldsymbol{\Theta}=(\Theta_0,\Theta_1,\Theta_t,\Theta_\infty)\in (\C^*)^4$ and denote, as usual, $\overline{\Theta}_i:=1/\Theta_i$ for each $i\in \{0,1,t,\infty\}$. 
 Consider the rational functions $f,g\in \C(\y,\zJS,t)$ in three complex variables given by 
  \begin{equation}\label{eqfg}   \left\{\begin{array}{rcl} g(\y,\zJS,t)  &:=&\frac{1}{q} \frac{\left(\y -t\Theta_t \right)\left(\y -t\overline{\Theta}_t \right)}{ \zJS \left(\y -\Theta_1\right)\left(\y -\overline{\Theta}_1 \right)}, \vspace{.2cm}\\
f(\y,\zJS,t) &:=&\frac{\left(g(\y,\zJS,t) - t\Theta_0\right)\left(g(\y,\zJS,t)- t\overline{\Theta}_0 \right)}{ \y \left(g(\y,\zJS,t) - \frac{{\Theta}_\infty}{q}  \right)\left(g(\y,\zJS,t) - \overline{\Theta}_\infty \right)}

= \frac{\left(\frac{1}{q} \frac{\left(\y -t\Theta_t \right)\left(\y -t\overline{\Theta}_t \right)}{ \zJS \left(\y -\Theta_1\right)\left(\y -\overline{\Theta}_1 \right)} - t\Theta_0\right)\left(\frac{1}{q} \frac{\left(\y -t\Theta_t \right)\left(\y -t\overline{\Theta}_t \right)}{ \zJS \left(\y -\Theta_1\right)\left(\y -\overline{\Theta}_1 \right)}- t\overline{\Theta}_0 \right)}{ \y \left(\frac{1}{q} \frac{\left(\y -t\Theta_t \right)\left(\y -t\overline{\Theta}_t \right)}{ \zJS \left(\y -\Theta_1\right)\left(\y -\overline{\Theta}_1 \right)} - \frac{{\Theta}_\infty}{q}  \right)\left(\frac{1}{q} \frac{\left(\y -t\Theta_t \right)\left(\y -t\overline{\Theta}_t \right)}{ \zJS \left(\y -\Theta_1\right)\left(\y -\overline{\Theta}_1 \right)} - \overline{\Theta}_\infty \right)}.
\end{array}\right. \end{equation}
 Note that with this notation, the sixth $q$-Painlev\'e equation \eqref{qPVI} with spectral data  $\boldsymbol{\Theta}$ writes 
 $$\sigma_{q,t} \y  =f(\y,\zJS,t)\, , \quad \sigma_{q,t} \zJS  =g(\y,\zJS,t).$$
 For each fixed $t=t_0\in \C^*$ such that \begin{equation}\label{eq:t_0Red}t\neq\Theta_0^{\pm 1} \Theta_\infty/q\, ,\quad t \neq\Theta_0^{\pm 1} \overline{\Theta}_\infty\,, \quad t\neq\Theta_t^{\pm 1} \Theta_1\, ,\quad t \neq\Theta_t^{\pm 1} \overline{\Theta}_1\,,\end{equation}
the expressions of $f(\y,\zJS,t_0)$ and $g(\y,\zJS,t_0)$ in $\eqref{eqfg}$ are reduced, \emph{i.e.} they admit no common factor in nominator and denominator. 
Now choose any $t=t_0\in \C^*$ satisfying \eqref{eq:t_0Red} and consider the rational map
\begin{equation}\label{defFrakSt}\mathfrak{S}_t:\left\{\begin{array}{ccc} \mathbb{P}^1\times \mathbb{P}^1& \dashrightarrow &\mathbb{P}^1\times \mathbb{P}^1\\
(\y,\zJS)&\mapsto& \left(f(\y,\zJS,t_0) ,  g(\y,\zJS,t_0)\right)\, .\end{array}\right.\end{equation}
Note that for any point $(\y_0,\zJS_0)\in \C^*\times \C^*\subset \mathbb{P}^1\times \mathbb{P}^1$, the image $\mathfrak{S}(\y_0,\zJS_0)$ is a well defined point in $\mathbb{P}^1\times \mathbb{P}^1$. However, on the complement of $\C^*\times \C^*$ in $\mathbb{P}^1\times \mathbb{P}^1$, there are some points for which the image under $\mathfrak{S}_t$ is undetermined, \emph{i.e.} at least one of the coordinates of the image contains, even after switching to homogeneous coordinates in the source, an expression of the form ``$\frac{0}{0}$''. These points will we called \emph{critical points} in the following. Let us assume \begin{equation}\label{eq:ThetaRed}\Theta_0^2\neq 1\, , \quad\Theta_1^2\neq 1\, , \quad \Theta_t^2\neq 1\, , \quad \Theta_\infty^2\neq q\, . 
 \end{equation}
 Then there are precisely eight critical points, which are given, for $t=t_0$, as follows:
   \begin{align*}
&\boldsymbol{\gamma}_0^-(t) : ~(\y,\zJS)=\left(0, t\Theta_0/q\right),&&\boldsymbol{\gamma}_0^+(t) : ~(\y,\zJS)=\left(0, t\overline{\Theta}_0/q\right),\vspace{.2cm}\\
&\boldsymbol{\gamma}_1^- (t): ~ (\y,\zJS)=\left(\overline{\Theta}_1, \infty \right),&&\boldsymbol{\gamma}_1^+(t) : ~(\y,\zJS)=\left(\Theta_1, \infty \right),\vspace{.2cm}\\
&\boldsymbol{\gamma}_t^-(t) : ~ (\y,\zJS)=\left(t{\Theta}_t, 0 \right),&&\boldsymbol{\gamma}_t^+(t) : ~(\y,\zJS)=\left(t\overline{\Theta}_t, 0 \right),\vspace{.2cm}\\
&\boldsymbol{\gamma}_\infty^-(t) : ~ (\y,\zJS)=\left(\infty, \overline{\Theta}_\infty \right),&&\boldsymbol{\gamma}_\infty^+(t) : ~(\y,\zJS)=\left(\infty,{\Theta}_\infty /q\right)\, . \end{align*}
Now let us denote, for each $t\in \C^*$ satisfying \eqref{eq:t_0Red}, by 
\begin{equation}\label{defFrakPt}\mathfrak{P}_t:=\mathrm{Bl}\left(\mathbb{P}^1\times \mathbb{P}^1\right)_{\boldsymbol{\gamma}_0^-(t),\boldsymbol{\gamma}_0^+(t),  \boldsymbol{\gamma}_1^-(t),\boldsymbol{\gamma}_1^+(t),\boldsymbol{\gamma}_t^-(t),\boldsymbol{\gamma}_t^+(t),\boldsymbol{\gamma}_\infty^-(t),\boldsymbol{\gamma}_\infty^+(t)}\end{equation}
the blow up of $\mathbb{P}^1\times \mathbb{P}^1$ at the eight points $\boldsymbol{\gamma}_i^\pm(t)$. Here we continue to assume \eqref{eq:ThetaRed}, so that these are indeed eight distinct points, the blow up is well defined, and does moreover not depend on the order in which the points are successively blown up. As a slight precision to \cite[Proposition 1]{sakai2001rational}, we obtain the following. 
\begin{prop}\label{PropSakai} Let $q\in \C\setminus \{0,1\} $. Let $\boldsymbol{\Theta}\in (\C^*)^4$ such that \eqref{eq:ThetaRed} holds. 
Denote \begin{equation}\label{eq:DefSq}S_q:= \left\{ \, \Theta_1^{\varepsilon_1}\Theta_t^{\varepsilon_t}\, ,  \,   \Theta_0^{\varepsilon_0}\Theta_\infty^{\varepsilon_\infty} ~\middle|~\varepsilon_0,\varepsilon_1,\varepsilon_t,\varepsilon_\infty \in \{-1,1\}  \right\} \cdot q^\Z\, .\end{equation}
Then for any $t\in \C^*\setminus S_q$, the map 
$$\mathfrak{S}_t:\mathfrak{P}_t \dashrightarrow \mathfrak{P}_{qt}$$
induced by \eqref{defFrakSt}, \emph{via} pre-composition and post-composition with the blow ups in \eqref{defFrakPt},  is biregular. That is, this rational map contains no indeterminacy points and is bijective. In particular, it is a biholomorphism. 
\end{prop}

\begin{proof}
We simply follow the proof in \cite{sakai2001rational}, where generic $t,\boldsymbol{\Theta}$, were considered, and make sure that it goes through for fixed  $t,\boldsymbol{\Theta}$ as in the statement. In order to make the following argumentation shorter, let us first consider an example. 

Consider the rational map $\varphi$ from $\mathbb{P}^1\times \mathbb{P}^1$ with standard coordinates $(\y, \zJS)$ to $\mathbb{P}^1\times \mathbb{P}^1$ with standard coordinates $\left(\hat{\y}, \hat{\zJS}\right)$ defined by 
$(\y, \zJS)\mapsto \left(\hat{\y}, \hat{\zJS}\right)=\left(\y, \zJS (\y-\Theta_1)\right)\,.$ This map is actually a so-called \emph{elementary transformation} with respect to the ruling $\mathbb{P}^1\times \mathbb{P}^1\to  \mathbb{P}^1$ given by $(\y, \zJS)\mapsto \y$. In the complement of the fiber $\{y=\Theta_1\}$ of this ruling, its defines a biholomorphism onto its image. Moreover, this map has an indeterminacy point at $(\y, \zJS)=({\Theta}_1,\infty)$. With the exception of this indeterminacy point, every point in the fibre $\{\y=\Theta_1\}$ is mapped to the point $(\hat \y, \hat \zJS)=({\Theta}_1,0)$. Conversely, the inverse rational map 
$ \left(\hat{\y}, \hat{\zJS}\right)\mapsto (\y, \zJS)= \left(\hat{\y}, \frac{\hat{\zJS}}{\hat{\y}-\Theta_1}\right)\,$ has an indeterminacy point at $(\hat \y, \hat \zJS)=({\Theta}_1,0)$ and maps the rest of the fiber $\{\hat \y=\Theta_1\}$ to the point $(\y, \zJS)=({\Theta}_1,\infty)$.  However, as one can easily check, the rational map  obtained by considering the composition $$\mathrm{Bl}(\mathbb{P}^1\times \mathbb{P}^1)_{(\y, \zJS)=({\Theta}_1,\infty)}\longrightarrow \mathbb{P}^1\times \mathbb{P}^1 \stackrel{\varphi}{\dashrightarrow} \mathbb{P}^1\times \mathbb{P}^1 \dashrightarrow \mathrm{Bl}(\mathbb{P}^1\times \mathbb{P}^1)_{(\hat \y, \hat \zJS)=({\Theta}_1,0)}\,, $$  is biregular. So in summary, the elementary transformation $\varphi$ blows up the point $(\y, \zJS)=({\Theta}_1,\infty)$ and contracts the strict transform of the line $\{y=\Theta_1\}$, and becomes biregular when pre- and post-composed with the blow-ups of $(\y, \zJS)=({\Theta}_1,\infty)$ and $(\hat \y, \hat \zJS)=({\Theta}_1,0)$ respectively. 

The key is now to use elementary transformations in order to decompose the map $\mathfrak{S}_t$ into a sequence of biregular isomorphisms. 
 \begin{itemize}
  \item[$\bullet$] Consider the rational map from $\mathbb{P}^1\times \mathbb{P}^1$ with standard coordinates $(\y, \zJS)$ to $\mathbb{P}^1\times \mathbb{P}^1$ with standard coordinates $\left(\y, \widehat{\zJS}\right)$ defined by 
$(\y, \zJS)\mapsto \left(\y, \widehat{\zJS}\right)=\left(\y, \zJS\frac{(\y-\Theta_1)(\y-\overline{\Theta}_1)}{(\y-t\Theta_t)(\y-t\overline{\Theta}_t)}\right)\, .$ This map can be seen as the composition of four (commuting) elementary transformations with respect to the ruling $(\y, \zJS)\mapsto \y$. Note that by  assumption, the set $\{\Theta_1, \overline{\Theta}_1,t\Theta_t, t\overline{\Theta}_t\}$ has cardinality four, which implies that no two of these elementary transformations cancel each other out. 
Therefore, the considered rational map induces a biregular isomorphism from $\mathfrak{P}_t$ to the surface $\mathfrak{P}_t^{(1)}$, where $\mathfrak{P}_t^{(1)}$ denotes the blow up of $\mathbb{P}^1\times \mathbb{P}^1$ with standard coordinates $\left(\y, \widehat{\zJS}\right)$ at the eight points given, with respect to these coordinates, by 
$$\begin{array}{lllll}\left(0, \frac{{\Theta}_0}{qt}\right),&    \left(0,\frac{\overline{\Theta}_0}{qt}\right),&  \left(\overline{\Theta}_1,0\right),&   \left({\Theta}_1,0\right),&  \\ \left(t\Theta_t,\infty \right),&    \left(t\overline{\Theta}_t,\infty \right),&   \left(\infty, \overline{\Theta}_\infty\right),&   \left(\infty, {\Theta}_\infty/q\right). \end{array} $$
 \item[$\bullet$] Consider the biregular map from $\mathbb{P}^1\times \mathbb{P}^1$ with standard coordinates  $\left(\y, \widehat{\zJS}\right)$ to $\mathbb{P}^1\times \mathbb{P}^1$ with standard coordinates $\left(\y, \widetilde{\zJS}\right)$ defined by 
$ \left(\y, \widehat{\zJS}\right)\mapsto \left(\y, \widetilde{\zJS}\right)=\left(\y, \frac{1}{q\widehat{\zJS}}\right)\, .$ This map induces a biregular isomorphism from $\mathfrak{P}_t^{(1)}$ to the surface $\mathfrak{P}_t^{(2)}$, where $\mathfrak{P}_t^{(2)}$ denotes the blow up of $\mathbb{P}^1\times \mathbb{P}^1$ with standard coordinates $\left(\y, \widetilde{\zJS}\right)$ at the eight points given, with respect to these coordinates, by 
$$\begin{array}{lllll}\left(0, t\overline{\Theta}_0 \right),&    \left(0,t\Theta_0 \right),&  \left(\overline{\Theta}_1,\infty \right),&   \left({\Theta}_1,\infty\right),  \\ \left(t\Theta_t,0 \right),  &  \left(t\overline{\Theta}_t,0 \right),&   \left(\infty,  {\Theta}_\infty/q\right),&   \left(\infty,  \overline{\Theta}_\infty \right). \end{array} $$
  \item[$\bullet$] Consider the rational map from $\mathbb{P}^1\times \mathbb{P}^1$ with standard coordinates $(\y, \widetilde{\zJS})$ to $\mathbb{P}^1\times \mathbb{P}^1$ with standard coordinates $\left(\widehat{\y}, \widetilde{\zJS}\right)$ defined by 
$(\y, \widetilde{\zJS})\mapsto \left(\widehat{\y}, \widetilde{\zJS}\right)=\left(\y\frac{( \widetilde{\zJS}-\Theta_\infty/q)( \widetilde{\zJS}-\overline{\Theta}_\infty)}{( \widetilde{\zJS}-t\Theta_0)( \widetilde{\zJS}-t\overline{\Theta}_0)}\, , \widetilde{\zJS}\right)\, .$ This map can be seen as the composition of four (commuting) elementary transformations with respect to the ruling $(\y, \widetilde{\zJS})\mapsto \widetilde{\zJS}$. Note that by  assumption, the set $\{t\Theta_0, t\overline{\Theta}_0,\Theta_\infty/q, \overline{\Theta}_\infty\}$ has cardinality four, which implies that   the considered rational map induces a biregular isomorphism from $\mathfrak{P}_t^{(2)}$ to the surface $\mathfrak{P}_t^{(3)}$, where $\mathfrak{P}_t^{(3)}$ denotes the blow up of $\mathbb{P}^1\times \mathbb{P}^1$ with standard coordinates $\left(\widehat{\y}, \widetilde{\zJS}\right)$ at the eight points given, with respect to these coordinates, by 
$$\begin{array}{lllll}\left(\infty, t\overline{\Theta}_0 \right),&    \left(\infty,t\Theta_0 \right),&  \left(\overline{\Theta}_1,\infty \right),&   \left({\Theta}_1,\infty\right),  \\ \left(\frac{\Theta_t}{qt},0 \right),&   \left(\frac{\overline{\Theta}_t}{qt},0 \right),&   \left(0,  {\Theta}_\infty/q\right),&   \left(0,  \overline{\Theta}_\infty \right). \end{array} $$
 \item[$\bullet$] Consider the biregular map from $\mathbb{P}^1\times \mathbb{P}^1$ with standard coordinates  $\left(\widehat{\y}, \widetilde{\zJS}\right)$ to $\mathbb{P}^1\times \mathbb{P}^1$ with standard coordinates $\left(\widetilde{\y}, \widetilde{\zJS}\right)$ defined by 
$\left(\widehat{\y}, \widetilde{\zJS}\right)\mapsto \left(\widetilde{\y}, \widetilde{\zJS}\right)=\left(\frac{1}{\widehat{\y}}, \widetilde{\zJS}\right)\, .$ This map induces a biregular isomorphism from $\mathfrak{P}_t^{(3)}$ to the surface $\mathfrak{P}_t^{(4)}$, where $\mathfrak{P}_t^{(4)}$ denotes the blow up of $\mathbb{P}^1\times \mathbb{P}^1$ with standard coordinates $\left(\widetilde{\y}, \widetilde{\zJS}\right)$ at the eight points given, with respect to these coordinates, by 
$$\begin{array}{lllll}\left(0, t\overline{\Theta}_0 \right),&    \left(0,t\Theta_0 \right),&  \left( {\Theta}_1,\infty \right),&   \left(\overline{\Theta}_1,\infty\right),  \\ \left(qt\overline{\Theta}_t,0 \right),&    \left(qt{\Theta}_t,0 \right),&   \left(\infty,  {\Theta}_\infty/q\right),&   \left(\infty,  \overline{\Theta}_\infty \right). \end{array} $$
\end{itemize}
It now suffices to see that $\mathfrak{P}_t^{(4)}=\mathfrak{P}_{qt}$ and that the biregular map $\mathfrak{P}_{t}\stackrel{\sim}{\to}  \mathfrak{P}_t^{(4)}$, obtained by composing all of the above, coincides with $\mathfrak{S}_t$. 
\end{proof}
\begin{rem}\label{RemSols}
 The above proposition implies that if  \eqref{eq:ThetaRed} holds, then for each $t\in \C^*\setminus S_q$, for each $n\in \N$, we have biregular maps 
$$\mathfrak{S}_t^{(n)}:=\mathfrak{S}_{q^nt}\circ \dots \circ \mathfrak{S}_{qt}\circ \mathfrak{S}_t : \mathfrak{P}_{t}\to \mathfrak{P}_{q^nt}\,, \quad \mathfrak{S}_t^{(-n)}:=\left(\mathfrak{S}_{q^{-n}t}^{(n)}\right)^{-1}: \mathfrak{P}_{t}\to  \mathfrak{P}_{q^{-n}t}\, .$$ In particular, if $t_0\in \C^*\setminus S_q$, and $(\y_0,\zJS_0)\in \C^*\times \C^*\subset  \mathfrak{P}_{t_0}$, then for each $n \in \Z$, we obtain a well-defined element
$$(\widetilde{\y}_n,\widetilde{\zJS}_n):=\mathfrak{S}_{t_0}^{(n)}(\y_0,\zJS_0)\in \mathfrak{P}_{q^n t_0}$$ and a well-defined element $(\y_n, \zJS_n)\in \mathbb{P}^1\times \mathbb{P}^1$ obtained by projecting \emph{via} the natural regular map
$\mathfrak{P}_{q^n t_0}\to \mathbb{P}^1\times \mathbb{P}^1$. But the fact that the association $(\y_0,\zJS_0)\to (\y_n, \zJS_n)$ is well-defined for each $n \in \Z$ is simply a reformulation of the statement of Proposition \ref{theo2}. So the above proposition proves the latter.
\end{rem}

 The $q$-analogue of an Okamoto space for fixed $t\in \C^*\setminus S_q$ will be a certain Zariski-open subset of $\mathfrak{P}_t$. In order to define it, let us go back to $\mathbb{F}_0:=\mathbb{P}^1\times \mathbb{P}^1$ with standard coordinates $(\y,\zJS)$, endowed with the eight distinct points $\boldsymbol{\gamma}_i^\pm$ and identify some particular components of this space, namely, the following vertical and horizontal lines:
 \begin{center}$\boldsymbol{\mathcal{H}}_0:\{\zJS=0\}\, , \quad \boldsymbol{\mathcal{H}}_\infty:\{\zJS=\infty\} \, , \quad \boldsymbol{\mathcal{V}}_0:\{\y=0\}\, , \quad \boldsymbol{\mathcal{V}}_\infty:\{\y=\infty\}\,, $\vspace{.2cm}\\
$ \boldsymbol{\mathcal{V}}_t^+:\{\y=t\overline{\Theta}_t\}\, , \quad \boldsymbol{\mathcal{V}}_t^-:\{\y=t\Theta_t\}\, .$\end{center}
This configuration of points and lines is illustrated in the following picture. Here the colors can be ignored for now, their use will become clear later, see Remark \ref{remcolor}. 

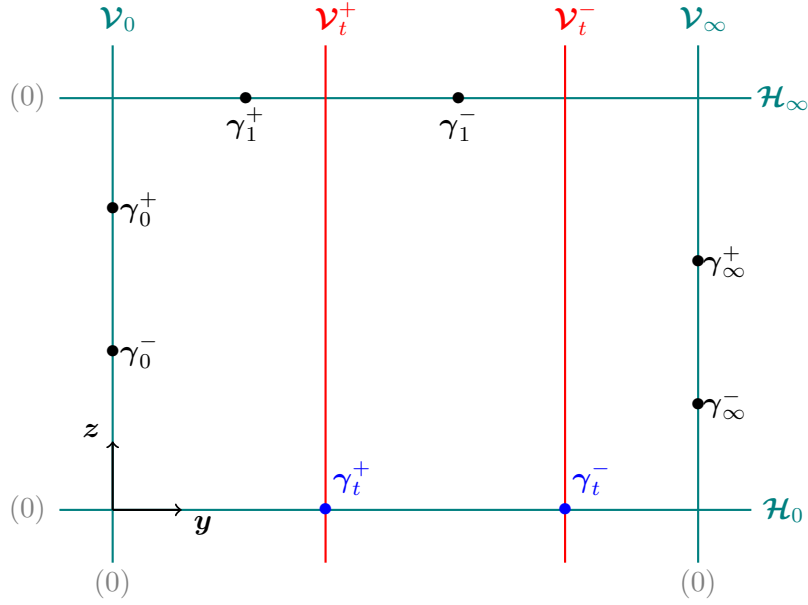
\begin{center}
\begin{tikzpicture}[scale=0.7]

\draw[teal,thick] (-6.5,4.3)--(6.5,4.3); \draw[teal] (7.1,4.3) node {$\boldsymbol{\mathcal{H}}_\infty$};
\draw[teal, thick] (-6.5,-3.5)--(6.5,-3.5);\draw[teal] (7.1,-3.5) node {$\boldsymbol{\mathcal{H}}_0$};
\draw[teal, thick] (5.5,-4.5)--(5.5,5.3);\draw[teal] (5.6,5.8) node {$\boldsymbol{\mathcal{V}}_\infty$};
\draw[teal, thick] (-5.5,-4.5)--(-5.5,5.3);\draw[teal] (-5.4,5.8) node {$\boldsymbol{\mathcal{V}}_0$};
\draw[red, thick] (-1.5,-4.5)--(-1.5,5.3); \draw[red] (-1.3,5.8) node {$\boldsymbol{\mathcal{V}}_t^+$};
\draw[red, thick] (3,-4.5)--(3,5.3); \draw[red] (3.2,5.8) node {$\boldsymbol{\mathcal{V}}_t^-$};

     \draw (-3,4.3) node {$\bullet$} ;   \draw(-3,3.8) node {$\boldsymbol{\gamma}_1^+$} ;
      \draw (1,4.3) node {$\bullet$} ;   \draw(1,3.8) node {$\boldsymbol{\gamma}_1^-$} ;
     \draw[blue] (-1.5,-3.5) node {$\bullet$} ;   \draw[blue] (-1,-2.9) node {$\boldsymbol{\gamma}_t^+$} ;
      \draw[blue] (3,-3.5) node {$\bullet$} ;   \draw[blue](3.5,-2.9) node {$\boldsymbol{\gamma}_t^-$} ;      
    \draw (-5.5,2.2) node {$\bullet$} ;   \draw (-5,2.2) node {$\boldsymbol{\gamma}_0^+$} ;
     \draw (-5.5,-0.5) node {$\bullet$} ; \draw (-5,-0.5) node {$\boldsymbol{\gamma}_0^-$} ;

      \draw (5.5,1.2) node {$\bullet$} ;			  \draw (6,1.2) node {$\boldsymbol{\gamma}_\infty^+$} ;
     \draw (5.5,-1.5) node {$\bullet$} ;			    \draw (6,-1.5) node {$\boldsymbol{\gamma}_\infty^-$} ;

                \draw[ thick, ->] (-5.5,-3.5)--(-5.5,-2.2); 
            \draw[ thick, ->] (-5.5,-3.5)--(-4.2,-3.5); 
               \draw  (-3.8 ,-3.8)  node  {$\y$} ; 
         \draw (-5.9 , -2)  node  {$\zJS$} ; 
            \draw[gray] (-7.1,4.3) node {$(0)$} ;
  \draw[gray] (5.5,-4.9) node {$(0)$} ;
 \draw[gray] (-7.1,-3.5) node {$(0)$} ;
 \draw[gray] (-5.5,-4.9) node {$(0)$} ;

         \end{tikzpicture}
\end{center}

Recall that $\mathfrak{P}_t$ is the blow up of $\mathbb{F}_0$ in the eight points $\boldsymbol{\gamma}_i^\pm$ with $i\in \{0,1,t,\infty\}$. 
In $\mathfrak{P}_t$, we denote by $$\boldsymbol{\mathcal{H}}_0^{**}, \boldsymbol{\mathcal{H}}_\infty^{**}, \boldsymbol{\mathcal{V}}_0^{**}, \boldsymbol{\mathcal{V}}_\infty^{**}, \boldsymbol{\mathcal{V}}_t^{+*}, \boldsymbol{\mathcal{V}}_t^{-*}$$
the strict transforms of the corresponding projective lines in $\mathbb{F}_0$. We denote 
 $$\mathcal{\boldsymbol{J}}^t:=\boldsymbol{\mathcal{H}}_0^{**}\cup \boldsymbol{\mathcal{H}}_\infty^{**}\cup  \boldsymbol{\mathcal{V}}_0^{**}\cup \boldsymbol{\mathcal{V}}_\infty^{**}\, .$$
Under the map $\mathfrak{S}_t:\mathfrak{P}_t\to \mathfrak{P}_{qt}$, the set $\mathcal{\boldsymbol{J}}^t$ is mapped to the set $\mathcal{\boldsymbol{J}}^{qt}$. So in some sense, the set $\mathcal{\boldsymbol{J}}^t$ can be seen as invariant under 
the $q$-Painlev\'e map $\mathfrak{S}_t$. Moreover, a discrete solution $(\y_n,\zJS_n,q^n t_0)_{n \in \Z}$ given by an initial condition in $\mathcal{\boldsymbol{J}}^t$ is not very interesting in the sense that both 
$(\y_n )_{n \in \Z}$ and $(\zJS_n )_{n \in \Z}$ simply oscillate between $0$ and $\infty$. The Okamoto space $\qOka{t}$ for fixed $t\in \C^{*}\setminus S_q$ as introduced in \cite{sakai2001rational}  is by definition the complement of  $\mathcal{\boldsymbol{J}}^t$ in $\mathfrak{P}_t$. So we define 
$$\qOka{t}:=\mathfrak{P}_t\setminus \mathcal{\boldsymbol{J}}^t\, .$$
Note that the strict transforms $\boldsymbol{\mathcal{V}}_t^{\pm*}$ of the vertical lines $\boldsymbol{\mathcal{V}}_t$ in $\mathbb{F}_0$ are not contained in $\mathcal{\boldsymbol{J}}^t$. They do however play a particular role in the relation to the  construction of the \emph{modified} Okamoto space $\qOkaMod{t}$ that we will now define.  The letter will be better suited for the confluence problem.

Let us recall that $\bZZ=\frac{\frac{(\y-t{\Theta}_t)(\y-t\overline{\Theta}_t)}{q(\y-1)(\y-t)\zJS}-1}{(q-1)\y}$. Motivated by \eqref{uvfromyZ}, we apply the change of variable 
 \begin{equation}\label{varchangeUgrasVgras}(\boldsymbol{u},\boldsymbol{v})=(\y, \y(\y-1)(\y-t)\bZZ)=\left(\y ,\frac{(\y-t\Theta_{t})(y-t/\Theta_{t})}{q(q-1)\zJS }-\frac{(\y-1)(\y-t)}{(q-1)} \right),
 \end{equation} to the modified $q$-Painlev\'e $\mathrm{VI}$ equation \eqref{qPVIyZbis} with spectral data $\boldsymbol{\Theta}$.  
There are well-defined rational functions $\widetilde{f}, \widetilde{g}\in \C(\boldsymbol{u},\boldsymbol{v},t)$ such that with respect to these variables, \eqref{qPVIyZbis} is of the form
$$ \sigma_{q,t}\boldsymbol{u} =\widetilde{f}(\boldsymbol{u},\boldsymbol{v},t)\, , \quad \sigma_{q,t}\boldsymbol{v}  =\widetilde{g}(\boldsymbol{u},\boldsymbol{v},t) \, .$$
More precisely, we have 
 
$$ \left\{\begin{array}{rcl}
 \widetilde{f}(\boldsymbol{u},\boldsymbol{v},t)  &=&  \frac{\left( \frac{(\boldsymbol{u}-1)(\boldsymbol{u}-t)+(q-1)\boldsymbol{v}}{\left(\boldsymbol{u} -\Theta_1\right)\left(\boldsymbol{u} -\frac{1}{\Theta_1}\right)}-t\Theta_0 \right)\left( \frac{(\boldsymbol{u}-1)(\boldsymbol{u}-t)+(q-1)\boldsymbol{v}}{\boldsymbol{u} \left(\boldsymbol{u} -\Theta_1\right)\left(\boldsymbol{u} -\frac{1}{\Theta_1}\right)}-t\frac{1}{\Theta_0}\right)}{\left( \frac{(\boldsymbol{u}-1)(\boldsymbol{u}-t)+(q-1)\boldsymbol{v}}{\left(\boldsymbol{u} -\Theta_1\right)\left(\boldsymbol{u} -\frac{1}{\Theta_1}\right)}-\frac{1}{\Theta_\infty} \right)\left( \frac{(\boldsymbol{u}-1)(\boldsymbol{u}-t)+(q-1)\boldsymbol{v}}{\left(\boldsymbol{u} -\Theta_1\right)\left(\boldsymbol{u} -\frac{1}{\Theta_1}\right)}-\frac{\Theta_\infty}{q}\right)}\vspace{.2cm}\\
\widetilde{g}(\boldsymbol{u},\boldsymbol{v},t)&=& 
 \frac{\left(\boldsymbol{u}  -\Theta_1\right)\left(\boldsymbol{u}  -\frac{1}{\Theta_1}\right)\left( \widetilde{f}(\boldsymbol{u},\boldsymbol{v},t)-qt\Theta_t\right)\left( \widetilde{f}(\boldsymbol{u},\boldsymbol{v},t) -qt\frac{1}{{\Theta}_t}\right)}{ q(q-1)(\boldsymbol{u}-1)(\boldsymbol{u}-t)+q(q-1)^2\boldsymbol{v} }  -
 \frac{\left(  \widetilde{f}(\boldsymbol{u},\boldsymbol{v},t) -1\right)\left( \widetilde{f}(\boldsymbol{u},\boldsymbol{v},t)-qt\right)}{(q-1) }
\, .
\end{array}\right.$$
 We consider the second Hirzebruch surface ${\mathbb{F}}_2$ with $\C^2$-charts $\C^2_{\boldsymbol{u}_0, \boldsymbol{v}_0},\C^2_{\boldsymbol{u}_1, \boldsymbol{v}_1}, \C^2_{\boldsymbol{u}_2, \boldsymbol{v}_2},\C^2_{\boldsymbol{u}_3, \boldsymbol{v}_3}$ glued together along their $\C^*\times \C^*$-subsets according to the transition maps given by \begin{equation}\label{HirzCoord}\begin{array}{llll}
&\displaystyle (\boldsymbol{u}_0 ,\boldsymbol{v}_0 )=(\boldsymbol{u},\boldsymbol{v}),&&\displaystyle(\boldsymbol{u}_1,\boldsymbol{v}_1)\displaystyle=\left(\boldsymbol{u},\frac{1}{\boldsymbol{v}}\right),\vspace{.2cm}\\
&\displaystyle(\boldsymbol{u}_2,\boldsymbol{v}_2)=\left(\frac{1}{\boldsymbol{u}},\frac{\boldsymbol{v}}{\boldsymbol{u}^2}\right),& &\displaystyle (\boldsymbol{u}_3,\boldsymbol{v}_3)\displaystyle=\left(\frac{1}{\boldsymbol{u}},\frac{\boldsymbol{u}^2}{\boldsymbol{v}}\right).
\end{array}\end{equation} Here we will consider $ (\boldsymbol{u},\boldsymbol{v})$ as the standard coordinates, with respect to which we will define rational maps such as the following.
For each fixed $t=t_0\in \C^*$ satisfying \eqref{eq:t_0Red}, we obtain a   rational map
\begin{equation}\label{eq:DefSfrakTilde}\widetilde{\mathfrak{S}}_t : \left\{\begin{array}{ccc}  {\mathbb{F}}_2&\dashrightarrow &{\mathbb{F}}_2\vspace{.2cm}\\ (\boldsymbol{u},\boldsymbol{v})&\mapsto& \left(\widetilde{f}(\boldsymbol{u},\boldsymbol{v}, t),\widetilde{g}(\boldsymbol{u},\boldsymbol{v}, t)\right)\, .\end{array}\right.\end{equation}

 \begin{lem}\label{lem:CritvalsOkatilde} Let $\boldsymbol{\Theta}\in (\C^*)^4$ such that \eqref{eq:ThetaRed} holds. Let $t\in \C^*\setminus S_q$, where $S_q$ is defined in \eqref{eq:DefSq}. The indeterminacy points of the rational map $\widetilde{\mathfrak{S}}_t$
 defined in \eqref{eq:DefSfrakTilde} are precisely the following  eight points (in the source ${\mathbb{F}}_2$).  
 \begin{align*}
&\boldsymbol{\beta}_0^-(t) : ~(\boldsymbol{u} ,\boldsymbol{v})=\left(0, \frac{t(\overline{\Theta}_0-1)}{q-1}\right),&&\boldsymbol{\beta}_0^+(t) : ~(\boldsymbol{u} ,\boldsymbol{v})=\left(0, \frac{t(\Theta_0-1)}{q-1}\right),&\\
&\boldsymbol{\beta}_1^-(t) : ~ (\boldsymbol{u} ,\boldsymbol{v})=\left(\overline{\Theta}_1, \frac{(\overline{\Theta}_1-1)(t-\overline{\Theta}_1)}{q-1}\right),&&\boldsymbol{\beta}_1^+(t) : ~(\boldsymbol{u} ,\boldsymbol{v})=\left(\Theta_1, \frac{(\Theta_1-1)(t-\Theta_1)}{q-1}\right),&\\
&\boldsymbol{\beta}_t^-(t) : ~ (\boldsymbol{u} ,\boldsymbol{v})=\left(t\Theta_t, -\frac{t(\Theta_t-1)(t\Theta_t-1)}{q-1}\right) ,&&\boldsymbol{\beta}_t^+(t) : ~(\boldsymbol{u} ,\boldsymbol{v})=\left(t\overline{\Theta}_t, -\frac{t(\overline{\Theta}_t-1)(t\overline{\Theta}_t-1)}{q-1}\right),\\
&\boldsymbol{\beta}_\infty^-(t) : ~ (\boldsymbol{u}_2 ,\boldsymbol{v}_2)=\left(0, \frac{\overline{\Theta}_\infty-1 }{q-1}\right),&&\boldsymbol{\beta}_\infty^+(t) : ~(\boldsymbol{u}_2 ,\boldsymbol{v}_2)=\left(0, \frac{{\Theta}_\infty-q }{q(q-1)}\right) .\\
\end{align*}
\end{lem}
\begin{proof}
 The statement can easily be verified by direct computation. Note however that this lemma can also be deduced, with much less computation, from Proposition \ref{Prop:RelDeuxOkas} below. 
\end{proof}
   
    In addition to the eight indeterminacy points $\boldsymbol{\beta}_i^\pm$ for $i\in \{0,1,t,\infty\}$, we identify the following particular projective lines in $\mathbb{F}_2$:
 \begin{center} $\boldsymbol{\mathcal{H}}:\{\boldsymbol{v}_1=0\}\cup \{\boldsymbol{v}_3=0\}\, , \quad \boldsymbol{\mathcal{D}}_0:=\{\boldsymbol{u}=0\}\cup \{\boldsymbol{u}_1=0\}\, , \quad \boldsymbol{\mathcal{D}}_\infty:=\{\boldsymbol{u}_2=0\}\cup \{\boldsymbol{u}_3=0\}\, ,$ \vspace{.2cm}\\
 $\boldsymbol{\mathcal{D}}_1^+:=\{\boldsymbol{u}={\Theta}_1\}\cup \{\boldsymbol{u}_1=  {\Theta}_1\}\, , \quad \boldsymbol{\mathcal{D}}_1^-:=\{\boldsymbol{u}= \overline{\Theta}_1\}\cup \{\boldsymbol{u}_1=\overline{\Theta}_1\}\, ,$\vspace{.2cm}\\
 $\boldsymbol{\mathcal{D}}_t^+:=\{\boldsymbol{u}=t\overline{\Theta}_t\}\cup \{\boldsymbol{u}_1=t\overline{\Theta}_t\}\, , \quad \boldsymbol{\mathcal{D}}_t^-:=\{\boldsymbol{u}=t {\Theta}_t\}\cup \{\boldsymbol{u}_1=t {\Theta}_t\}\, .$\end{center}
 Moreover, we introduce the following curve  (that corresponds to $\boldsymbol{\mathcal{H}}_0:\{\zJS=0\}$):
 $$ \boldsymbol{\mathcal{C}}:=\{(\boldsymbol{u}-1)(\boldsymbol{u}-t)=(1-q)\boldsymbol{v}\}\, \cup\,  \{(\boldsymbol{u}_1-1)(\boldsymbol{u}_1-t)\boldsymbol{v}_1=1-q \}\, \cup\,  
 \{(1-\boldsymbol{u}_3)(1-t\boldsymbol{u}_3)\boldsymbol{v}_3=1-q \}\, .$$ 
 The configuration of these points, lines and the curve $ \boldsymbol{\mathcal{C}}$ in $\mathbb{F}_2$ is illustrated in the following figure. Here the grey numbers indicate the self-intersection number of the corresponding curve. 
 The use of colors will became clear later, see Remark \ref{remcolor}.
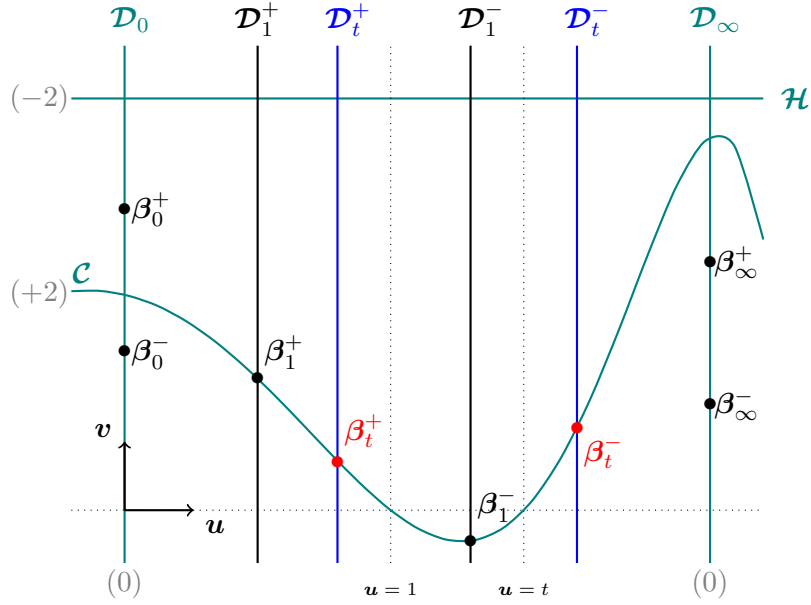
 \begin{center}
\begin{tikzpicture}[scale=0.7]

\draw[teal,thick] (-6.5,4.3)--(6.5,4.3); \draw[teal] (7.1,4.3) node {$\boldsymbol{\mathcal{H}} $};
\draw[dotted] (-6.5,-3.5)--(6.5,-3.5);
\draw[teal, thick] (5.5,-4.5)--(5.5,5.3);\draw[teal] (5.6,5.8) node {$\boldsymbol{\mathcal{D}}_\infty$};
\draw[teal, thick] (-5.5,-4.5)--(-5.5,5.3);\draw[teal] (-5.4,5.8) node {$\boldsymbol{\mathcal{D}}_0$};
\draw[blue, thick] (-1.5,-4.5)--(-1.5,5.3); \draw[blue] (-1.3,5.8) node {$\boldsymbol{\mathcal{D}}_t^+$};
\draw[blue, thick] (3,-4.5)--(3,5.3); \draw[blue] (3.2,5.8) node {$\boldsymbol{\mathcal{D}}_t^-$};
\draw[  thick] (-3,-4.5)--(-3,5.3); \draw  (-3,5.8) node {$\boldsymbol{\mathcal{D}}_1^+$};
\draw[  thick] (1,-4.5)--(1,5.3); \draw  (1.2,5.8) node {$\boldsymbol{\mathcal{D}}_1^-$};
\draw[dotted] (2,-4.5)--(2,5.3); \draw (2,-5) node {\tiny{$\boldsymbol{u}=t$}};
\draw[dotted] (-0.5,-4.5)--(-0.5,5.3); \draw (-0.5,-5) node {\tiny{$\boldsymbol{u}=1$}};

\draw [teal, thick] plot  [smooth, domain=-6.5:6.5] (\x,{-3.5-(\x+12)*0.004*exp(0.2*\x)*(\x-2)*(\x+0.5)*(\x-7.1)});
\draw [teal] (-6.3,1) node {$\boldsymbol{\mathcal{C}}$};

     \draw (-3,-1) node {$\bullet$} ;   \draw(-2.5,-0.5) node {$\boldsymbol{\beta}_1^+$} ;
      \draw (1,-4.1) node {$\bullet$} ;   \draw(1.5,-3.4) node {$\boldsymbol{\beta}_1^-$} ;
     \draw[red] (-1.5,-2.6) node {$\bullet$} ;   \draw[red] (-1 ,-2) node {$\boldsymbol{\beta}_t^+$} ;
      \draw[red] (3,-1.95) node {$\bullet$} ;   \draw[red](3.5,-2.4) node {$\boldsymbol{\beta}_t^-$} ;      
    \draw (-5.5,2.2) node {$\bullet$} ;   \draw (-5,2.2) node {$\boldsymbol{\beta}_0^+$} ;
     \draw (-5.5,-0.5) node {$\bullet$} ; \draw (-5,-0.5) node {$\boldsymbol{\beta}_0^-$} ;

      \draw (5.5,1.2) node {$\bullet$} ;			  \draw (6,1.2) node {$\boldsymbol{\beta}_\infty^+$} ;
     \draw (5.5,-1.5) node {$\bullet$} ;			    \draw (6,-1.5) node {$\boldsymbol{\beta}_\infty^-$} ;

                \draw[ thick, ->] (-5.5,-3.5)--(-5.5,-2.2); 
            \draw[ thick, ->] (-5.5,-3.5)--(-4.2,-3.5); 
               \draw  (-3.8 ,-3.8)  node  {$\boldsymbol{u}$} ; 
         \draw (-5.9 , -2)  node  {$\boldsymbol{v}$} ; 

             \draw[gray] (-7.1,4.3) node {$(-2)$} ;
  \draw[gray] (5.5,-4.9) node {$(0)$} ;
 \draw[gray] (-7.1,.6) node {$(+2)$} ;
 \draw[gray] (-5.5,-4.9) node {$(0)$} ;

         \end{tikzpicture}
\end{center}
Note that the points $\boldsymbol{\beta}_1^\pm$, respectively  $\boldsymbol{\beta}_t^\pm$, are precisely the intersection between $\boldsymbol{\mathcal{C}}$ and $\boldsymbol{\mathcal{D}}_1^\pm$, respectively $\boldsymbol{\mathcal{C}}$ and $\boldsymbol{\mathcal{D}}_t^\pm$. Note further that $\boldsymbol{\mathcal{C}}\cap \boldsymbol{\mathcal{H}}=\emptyset$,  that $\boldsymbol{\beta}_0^\pm \not \in \boldsymbol{\mathcal{H}}\cup \boldsymbol{\mathcal{C}}$ because $\Theta_0, \overline{\Theta}_0\neq 0$ and that 
$\boldsymbol{\beta}_\infty^\pm \not \in \boldsymbol{\mathcal{H}}\cup \boldsymbol{\mathcal{C}}$ because $\Theta_\infty, \overline{\Theta}_\infty\neq 0$.

Now let us denote, for each $t\in \C^*$ satisfying \eqref{eq:t_0Red}, by 
$$\widetilde{\mathfrak{P}}_t:=\mathrm{Bl}\left(\mathbb{P}^1\times \mathbb{P}^1\right)_{\boldsymbol{\beta}_0^-(t),\boldsymbol{\beta}_0^+(t),  \boldsymbol{\beta}_1^-(t),\boldsymbol{\beta}_1^+(t),\boldsymbol{\beta}_t^-(t),\boldsymbol{\beta}_t^+(t),\boldsymbol{\beta}_\infty^-(t),\boldsymbol{\beta}_\infty^+(t)}$$
the blow up of $\mathbb{F}_2$ at the eight points $\boldsymbol{\beta}_i^\pm(t)$. Here we continue to assume \eqref{eq:ThetaRed}.  
In $\widetilde{\mathfrak{P}}_t$, we denote by $$\boldsymbol{\mathcal{H}}, \boldsymbol{\mathcal{D}}_0^{**}, \boldsymbol{\mathcal{D}}_\infty^{**}, \boldsymbol{\mathcal{D}}_1^{+*}, \boldsymbol{\mathcal{D}}_1^{-*}, \boldsymbol{\mathcal{D}}_t^{+*}, \boldsymbol{\mathcal{D}}_t^{-*}, \boldsymbol{\mathcal{C}}^{****}$$
the strict transforms of the corresponding projective lines/curves in $\mathbb{F}_2$. We define  
$$q\mathrm{-}\widetilde{\mathrm{Oka}}_t:=\widetilde{\mathfrak{P}}_t\setminus \mathcal{\boldsymbol{I}}^t\, , \textrm{ where } \mathcal{\boldsymbol{I}}^t:=\boldsymbol{\mathcal{H}} \cup    \boldsymbol{\mathcal{D}}_0^{**}\cup \boldsymbol{\mathcal{D}}_\infty^{**}\cup \boldsymbol{\mathcal{C}}^{****}\, .$$
As we shall see, this is an {alternative} $q$-Okamoto space of initial values of $qP_{\mathrm{VI}}$, and $\qOkaMod{t}$ is convenient for the study of confluence. 
Before formulating the equivalence of $q\textrm{-}{\mathrm{Oka}}_t$ and $\qOkaMod{t}$, let us give a name to the exceptional curves. We denote, for each $i\in \{0,1,t,\infty\}$, by $\boldsymbol{\mathcal{F}}_i^{\pm}$ the exceptional lines in ${\mathfrak{P}}_t$ corresponding to   blow up of $\boldsymbol{\gamma}_t^\pm$ and by $\boldsymbol{\mathcal{E}}_i^{\pm}$ the exceptional lines in $\widetilde{\mathfrak{P}}_t$ corresponding to   blow up of $\boldsymbol{\beta}_i^\pm$.  

\begin{prop}\label{Prop:RelDeuxOkas}Let $q\in \C\setminus \{0,1\}$. 
Let $\boldsymbol{\Theta}\in (\C^*)^4$ such that \eqref{eq:ThetaRed} holds. Let $t\in \C^*\setminus S_q$, where $S_q$ is defined in \eqref{eq:DefSq}.  
 Consider the birational map given, with respect to the standard coordinates and the above notation, by 
$$\varphi : \left\{\begin{array}{rcl}{\mathfrak{P}}_t&\dashrightarrow &\widetilde{\mathfrak{P}}_t\vspace{.2cm}\\ (\y,\zJS)&\mapsto &(\boldsymbol{u},\boldsymbol{v})= \left(\y ,  \frac{ (\y -t\Theta_t)\left(\y -t\overline{\Theta}_t\right)-q(\y-1)(\y-t)  \zJS}{q(q-1) \zJS} \right)\, .\end{array}\right.$$
This map is biregular and induces bijections $$\boldsymbol{\mathcal{V}}_t^{\pm *}\simeq  \boldsymbol{\mathcal{E}}_t^{\pm}\, , \quad  \boldsymbol{\mathcal{F}}_t^{\pm}\simeq  \boldsymbol{\mathcal{D}}_t^{\pm*}\quad \textrm{and} \quad  \boldsymbol{\mathcal{F}}_i^{\pm}\simeq \boldsymbol{\mathcal{E}}_i^{\pm}\quad \forall i\in \{0,1,\infty\}\, .$$
Moreover, it induces a bijection $\boldsymbol{J}^t \stackrel{\sim}{\to}\boldsymbol{I}^t$ and therefore provides an isomorphism 
$$ \qOka{t}\stackrel{\sim}{\longrightarrow} \qOkaMod{t}\, .$$
\end{prop}
\begin{proof}
First, consider the rational map $\mathbb{F}_0\dashrightarrow  \mathbb{F}_2$ given, with respect to the standard coordinates, by the same formula as $\varphi$. We will abusively denote it again by $\varphi$.  Note that 
$\varphi$ preserves the fibers of the rulings $ \mathbb{F}_2\to  \mathbb{P}^1$ and $ \mathbb{F}_0\to  \mathbb{P}^1$ given by $(\boldsymbol{u},\boldsymbol{v}) \mapsto \boldsymbol{u}$ and $(\y,\zJS)\mapsto \y$. Hence we may  restrict and corestrict $\varphi$ to a map $\varphi^\circ : \mathbb{F}_0\setminus \boldsymbol{\mathcal{V}}_t^\pm \dashrightarrow  \mathbb{F}_2\setminus \boldsymbol{\mathcal{D}}_t^\pm$.  It is however immediate to check that $\varphi^\circ$ is  regular and a bijection, with 
inverse map given by
$$\psi :\left\{\begin{array}{rcl} \mathbb{F}_2\setminus \boldsymbol{\mathcal{D}}_t^\pm&\to& \mathbb{F}_0\setminus \boldsymbol{\mathcal{V}}_t^\pm\vspace{.2cm}\\  (\boldsymbol{u},\boldsymbol{v}) &\mapsto &(\y,\zJS)= \left(\boldsymbol{u} ,  \frac{1}{q}\cdot\frac{ (\boldsymbol{u} -t\Theta_t)\left(\boldsymbol{u} -t\overline{\Theta}_t\right)}{ (\boldsymbol{u}-1)(\boldsymbol{u}-t)+(q-1)\boldsymbol{v}  }\right)\, .\end{array}\right.$$
Moreover, it is immediate to check that $\varphi^\circ$ maps the points $\boldsymbol{\gamma}_i^\pm$ to the points $\boldsymbol{\beta}_i^\pm$ for each $i\in \{0,1,\infty\}$, and that 
it induces bijections $$\boldsymbol{\mathcal{V}}_0\simeq \boldsymbol{\mathcal{D}}_0\, , \quad \boldsymbol{\mathcal{V}}_\infty \simeq \boldsymbol{\mathcal{D}}_\infty \, , \quad   \boldsymbol{\mathcal{H}}_0\setminus \{\boldsymbol{\gamma}_t^\pm\} \simeq \boldsymbol{\mathcal{H}}_\infty \setminus \boldsymbol{\mathcal{D}}_t^\pm,  \quad \,  \boldsymbol{\mathcal{H}}  \setminus  \boldsymbol{\mathcal{V}}_t^\pm \simeq \boldsymbol{\mathcal{C}} \setminus \{\boldsymbol{\beta}_t^\pm\}\, .$$

To conclude, we use the same argument as in the proof of Proposition \ref{PropSakai}.
Namely, in a neighborhood of the fibers $\boldsymbol{\mathcal{V}}_t^\pm$ and $\boldsymbol{\mathcal{D}}_t^\pm$, $\varphi$ is an elementary transformation blowing up the point $\boldsymbol{\gamma}_t^\pm$ and contracting the strict transform of the fiber $\boldsymbol{\mathcal{V}}_t^\pm$ onto the point $\boldsymbol{\beta}_t^\pm$. Therefore, the induced map $\varphi:\mathrm{Bl}(\mathbb{F}_0)_{\boldsymbol{\gamma}_t^\pm}\to \mathrm{Bl}(\mathbb{F}_2)_{\boldsymbol{\beta}_t^\pm}$ is biregular. The result follows.
 \end{proof}
\begin{rem}\label{remcolor} As shown in the above proof, the rational map $\varphi : \mathbb{F}_2\dashrightarrow \mathbb{F}_0$ corresponding to the change of variable \eqref{varchangeUgrasVgras} composed with the change of variable \eqref{eqzJSbZZ} (which relates $q{P}_{\mathrm{VI}}(\boldsymbol{\Theta})$ to the modified $q\widetilde{P}_{\mathrm{VI}}(\boldsymbol{\Theta})$), respects the scheme of colors in the diagrams representing the particular lines in $\mathbb{F}_2$ and $\mathbb{F}_0$.
\end{rem}

\subsection{Confluence}\label{sec:ConfOkaSpaces} In this section, we will see that the differential Okamoto space can be obtained from the second version of the $q$-difference one by a limit process. More precisely, we will show that they smoothly fit together into a family of Okamoto-spaces, parametrized by a neighborhood of $q=1$ in $\C$. 

%\subsubsection{Confluence of Okamoto spaces}

Let $\boldsymbol{\theta}=(\theta_0,\theta_1,\theta_t,\theta_\infty)\in \C^*\times \C^*\times \C^*\times \left(\C\setminus \{1\}\right)$, and consider, as in Section \ref{sec:confPainl}, the quadrupel of rational functions
\begin{equation}\label{Thetaconfb}\boldsymbol{\Theta}(q)=1+\frac{q-1}{2}\boldsymbol{\theta}.\end{equation}
Let $t\in \C^*$. Let us consider the second Hirzebruch surface $\mathbb{F}_2$ with coordinates $(\boldsymbol{u}, \boldsymbol{v}), (\boldsymbol{u}_1, \boldsymbol{v}_1), (\boldsymbol{u}_2, \boldsymbol{v}_2) , (\boldsymbol{u}_3, \boldsymbol{v}_3)$ as in \eqref{HirzCoord}. For each $i\in \{0,1,t,\infty\}$, we may define meromorphic functions in the variable $q\in \C$, holomorphic in a neighborhood of $\{q=1\}$, of the form
   $$\boldsymbol{\beta}_i^\pm:\C\to \mathbb{F}_2\, , $$
   given, for each fixed $q\in \C$, as follows.
     \begin{align*}
&\boldsymbol{\beta}_0^-(q) : ~(\boldsymbol{u} ,\boldsymbol{v})=\left(0,  \frac{-t\theta_0/2}{\Theta_0(q)}\right)  ,&&\boldsymbol{\beta}_0^+(q) : ~(\boldsymbol{u} ,\boldsymbol{v})=\left(0, \frac{t\theta_0}{2}\right),&\\
&\boldsymbol{\beta}_1^-(q) : ~ (\boldsymbol{u} ,\boldsymbol{v})=\left(\frac{1}{\Theta_1(q)}, \frac{-\theta_1/2\left(t-\frac{1}{\Theta_1(q)}\right)}{\Theta_1(q)} \right),&&\boldsymbol{\beta}_1^+(q) : ~(\boldsymbol{u} ,\boldsymbol{v})=\left(\Theta_1(q), \frac{\theta_1 (t-\Theta_1(q))}{2}\right),&\\
&\boldsymbol{\beta}_t^-(q) : ~ (\boldsymbol{u} ,\boldsymbol{v})=\left(t\Theta_t(q), - \frac{t\theta_t}{2}\left(t \Theta_t(q)-1 \right)\right) ,&&\boldsymbol{\beta}_t^+(q) : ~(\boldsymbol{u} ,\boldsymbol{v})=\left(\frac{t}{ {\Theta}_t(q)},  \frac{t(t-{\Theta}_t(q))\theta_t/2}{ {\Theta}_t(q)^2}\right),\\
&\boldsymbol{\beta}_\infty^-(q) : ~ (\boldsymbol{u}_2 ,\boldsymbol{v}_2)=\left(0,  \frac{-\theta_\infty/2}{\Theta_\infty(q)} \right),&&\boldsymbol{\beta}_\infty^+(q) : ~(\boldsymbol{u}_2 ,\boldsymbol{v}_2)=\left(0, \frac{ \theta_\infty/2-1  }{q }\right) .\\
 \end{align*}
Note that one the one hand, for generic values of $q$ and $t$, these correspond in the confluence setting \eqref{Thetaconfb} to the values of the $\boldsymbol{\beta}_i^\pm$ in Lemma \ref{lem:CritvalsOkatilde}. On the other hand, for $q=1$ and $t\neq 1$, we have $\boldsymbol{\beta}_i^\pm(1)=\beta_i^\pm$ with $\beta_i^\pm$ as in Section~\ref{sec:oka}. 
  Now we may see these functions as the parametrized curves $\left\{\left(q,\boldsymbol{\beta}_i^\pm\right)~|~q\in \C^*\right)\}$ in the product 
  $$\C^* \times \mathbb{F}_2\, .$$
  We will abusively call them the curves $\boldsymbol{\beta}_i^\pm$. Moreover, in this product space, we may identify the following planes $$\boldsymbol{\mathcal{H}}:\{\boldsymbol{v}_1=0\}\cup \{\boldsymbol{v}_3=0\}\, , \quad \boldsymbol{\mathcal{D}}_0:=\{\boldsymbol{u}=0\}\cup \{\boldsymbol{u}_1=0\}\, , \quad \boldsymbol{\mathcal{D}}_\infty:=\{\boldsymbol{u}_2=0\}\cup \{\boldsymbol{u}_3=0\}$$ and the surface
   $$ \boldsymbol{\mathcal{C}}:=\{(\boldsymbol{u}-1)(\boldsymbol{u}-t)=(1-q)\boldsymbol{v}\}\, \cup\,  \{(\boldsymbol{u}_1-1)(\boldsymbol{u}_1-t)\boldsymbol{v}_1=1-q \}\, \cup\,  
 \{(1-\boldsymbol{u}_3)(1-t\boldsymbol{u}_3)\boldsymbol{v}_3=1-q \}\, .$$ 
Note that the curve defined by the restriction of $\boldsymbol{\mathcal{C}}$ to $\{q=1\}$ is degenerate: it has three irreducible components, given, with respect to the notation in Section \ref{sec:oka}, by $\mathcal{H}$, $\mathcal{D}_1$ and $\mathcal{D}_t$ respectively. Now denote by $$\Omega_t:=\mathrm{Bl}\left(\C^*\times \mathbb{F}_2\right)_{\boldsymbol{\beta}_i^\pm} ~\setminus \boldsymbol{J}\quad \textrm{where} \quad \boldsymbol{J}:\boldsymbol{\mathcal{H}}\cup \boldsymbol{\mathcal{C}}^{****}\cup \boldsymbol{\mathcal{D}}_0^{**} \cup \boldsymbol{\mathcal{D}}_\infty^{**}$$ the blow up of this product space along the eight curves $\boldsymbol{\beta}_i^\pm$ minus the strict transforms of the mentioned particular surfaces.\footnote{Strictly speaking, here one has to choose an order for the eight curves to be blown up in order to obtain a well-defined result in restriction to those $q\in \C^*$ where the curves intersect. We will however neglect these values of $q$ anyway afterwards, because they are not close to $1$. }
Then we have $$\left\{\begin{array}{rclll}\Omega_t|_{q=1}&=&\mathrm{Oka}_t & \textrm{for all } t\in \C\setminus \{0,1\} \vspace{.2cm}\\
  \Omega_t|_{q=q_0}&=&\qOkaMod{t}(q_0) & \textrm{for all } (t,q_0)\in \C^*\times \left(\C\setminus \{0,1\}\right) \textrm{ such that \eqref{eq:ThetaRed} holds and }t\not \in S_{q_0}\, .\end{array}\right.  $$
  Recall that $S_{q_0}$ was defined in \eqref{eq:DefSq}. Let us see how these two conditions of validity fit together in the family 
  $$\boldsymbol{\Omega}:=\bigcup_{t\in \C\setminus \{0,1\}}  \Omega_t\, .$$
  The condition \eqref{eq:ThetaRed} is vacuous for $q_0$ sufficiently close to $1$, because we assumed \eqref{cond8bp}. However, in the parameter space $\C^*\times  \C^* $ with coordinates $(t,q)$ of  $\boldsymbol{\Omega}$, the set $\{(t,q)~|~t\in  S_{q}, q\neq 1\}$ decomposes as an infinite union of curves of the form $\left\{ \, \Theta_1^{\varepsilon_1}\Theta_t^{\varepsilon_t}q^k ~\middle|~\varepsilon_0,\varepsilon_1 \in \{-1,1\},k\in \Z \right\} $ and $\left\{ \Theta_0^{\varepsilon_0}\Theta_\infty^{\varepsilon_\infty}q^k ~\middle|~\varepsilon_t,\varepsilon_\infty \in \{-1,1\},k\in \Z \right\} $. By \eqref{Thetaconfb}, the adherence of each of these curves at $\{q=1\}$   is given by $\{q=1,t=1\}$. 
  
  \begin{rem} The behavior of the set $S_{q}$ when $q$ goes to $1$ may be wild. In order to fix this, analouglsy to \cite{sauloy2000regular}, we need to make $q$ goes to $1$ following a $q$-spiral. More precisely, let $t\in \C\setminus \{0,1\}$  and fix $|q_{0}|>1$ with $t\notin q_{0}^{\R}$. We have $q_{0}^{\varepsilon}\to 1$, when $\varepsilon>0 $ is a real number going to $0$ and for  $\varepsilon>0 $ sufficiently close to $0$, we find $t\notin S_{q_{0}^{\varepsilon}}$.
     \end{rem}

\section{Appendix: The relation between two notions of $q$-isomonodromy}\label{app:qFundsols}
Some authors  interpret the pseudo-constancy of the Birkhoff connection matrix as a suitable discrete analogue for the isomonodromy of families of Fuchsian systems. We will explain here how this is related to our notion of $q$-isomonodromy (see Section \ref{sec1}). This Birkhoff connection matrix is defined \emph{via} certain fundamental solutions of the family of $q$-Fuchsian systems parameterized by $t\in \mathfrak{D}$. Therefore, we shall first recall from \cite{sauloy2000regular}  the construction of fundamental solutions   (see also  \cite{praagman1986fundamental,ramis2009local,dreyfus2014building} for constructions in some more general settings). Note that these fundamental solutions  will be meromorphic matrix functions on $\C^*\times \mathfrak{D}$. In particular, they are uniform in $t$, which is one of the reasons why the definition of monodromy in the differential case should not be translated literally to the $q$-difference setting.

Let $q$  be a complex number with $|q|>1$. Let $\mathfrak{D}$ be open connected  subset of $\C^*$.
 Let \begin{equation}\label{eqq1ter}
\sigma_{q,x} Y(x,t)=\mathfrak{A}(x,t)Y(x,t),\quad  \textrm{ with }\quad   \mathfrak{A}(x,t)=\mathfrak{A}_{0}(t)+x\frac{\mathfrak{A}_{1}(t)}{x-1}+x\frac{\mathfrak{A}_{t}(t)}{t(x-t)}\, 
\end{equation}
be a family of $q$-Fuchsian systems  as in Definition \ref{qfuchsian} with \underline{non-resonant} spectral data $(\boldsymbol{\Theta}, \overline{\boldsymbol{\Theta}})$. 
Note that in particular, we assume that that for each $i\in \{0,1,t\}$, we have $\mathfrak{A}_{i} \in \mathrm{M}_2(\mathcal{O}(\mathfrak{D}))$, \emph{\emph{i.e.}} these matrices have holomorphic entries. 
Moreover, by the requirements of Definition \ref{qfuchsian},  $$\mathrm{Spec}(\mathfrak{A}_{0}(t))=\left\{ \overline{\Theta}_0,\Theta_0\right\}\quad \textrm{and} \quad \mathfrak{A}_{\infty}=\begin{pmatrix}\overline{\Theta}_\infty& 0\\0& \Theta_\infty \end{pmatrix}\, ,$$
where $\mathfrak{A}_{\infty}=\mathfrak{A}_{0}+\mathfrak{A}_{1}+\frac{\mathfrak{A}_{t}}{t}$. Note that we may have $\overline{\Theta}_0= \Theta_0$ so that the matrix $\mathfrak{A}_0(t)$ may be not diagonalisable.
Let us choose $P\in \mathrm{GL}_2(\mathcal{M}(\mathfrak{D}))$ such that  
 \begin{equation}\label{eqDiagA0} P(t)\mathfrak{A}_0(t)P(t)^{-1}=J_{0}\,,
 \end{equation}
 where $J_{0}$ is in Jordan normal form, with eigenvelues $\overline{\Theta}_0, \Theta_0$.
 \begin{lem} \label{lemL0L8} For each $i\in \{0,\infty\}$, there exists a matrix $L_i(x) \in \mathrm{GL}_2(\mathcal{O}(\C^*))$ satisfying the $q$-difference equation 
 $$\sigma_{q,x}L_{0}=J_{0}L_{0},\quad \sigma_{q,x}L_{\infty} = \begin{pmatrix}\overline{\Theta}_{\infty}&0\\0&\Theta_{\infty}\end{pmatrix}L_{\infty} .$$
 \end{lem}
 \begin{proof}
As explained for example in \cite[Page 1024]{sauloy2000regular}, for any  $a\in \C$, there exists a meromorphic function $\mathfrak{e}_{q,a}(x)$ on $\C^*$ satisfying the $q$-difference equation 
 $$\sigma_{q,x} \mathfrak{e}_{q,a} =a\mathfrak{e}_{q,a} \, .$$ Indeed, one may set $\mathfrak{e}_{q,a}(x):=\frac{\T_{q}(x)}{\T_{q}(x/a)}$, where $\T_{q}(x)=\displaystyle \sum_{n \in \Z} q^{\frac{-n(n+1)}{2}}x^{n}$ is the Jacobi theta function satisfying $\sigma_{q,x}\T_{q}(x)=x\T_{q}(x)$.
 We may now choose $L_{\infty}(x):=\mathrm{diag}\left( \mathfrak{e}_{q,\overline{\Theta}_{\infty}}(x)\, ,  \mathfrak{e}_{q, {\Theta}_{\infty}}(x)\right)\, .$
 If $J_{0}$ is diagonal, the construction of $L_0$ is analogue. Let us assume that $J_{0}=\begin{pmatrix}\Theta_{0}&1\\0&\Theta_{0}\end{pmatrix}$ is not diagonal. Let us introduce the $q$-logarithm $\ell_{q}(x)=\frac{x\partial_{x}\T_{q}(x) }{\T_{q}(x)}$, that satisfies $\sigma_{q,x}(\ell_{q})=\ell_{q}+1$. Then, we may take $$L_{0}(x)=\begin{pmatrix}\mathfrak{e}_{q,\Theta_{0}}(x)&\frac{\mathfrak{e}_{q,\Theta_{0}}(x)\ell_{q}(x)}{\Theta_{0}}\\0&\mathfrak{e}_{q,\Theta_{0}}(x)\end{pmatrix}.$$
 \end{proof}
\begin{rem} Of course the choice of the matrices $L_i$ in the above lemma is not unique. For instance, some authors prefer to replace $\mathfrak{e}_{q,a}(x)$ and $\ell_{q}(x)$ respectively by $$a^{\frac{\ln (x)}{\ln(q)}}\,\quad  \hbox{and}\quad  \frac{\ln(x)}{\ln(q)} , $$ which satisfies the same $q$-difference equation, and yields matrices $L_i$ whose entries are defined no longer on $\C^*$, but on the Riemann surface of the complex logarithm. 
\end{rem}
 \begin{prop}\label{prop:qsolfond} Let $L_0,L_\infty$ be as in Lemma \ref{lemL0L8}. 
 Let $P\in \mathrm{GL}_2(\mathcal{M}(\mathfrak{D}))$ such that \eqref{eqDiagA0} holds. 
  There exists a unique pair $( \mathfrak{H}_{0},  \mathfrak{H}_{\infty})$ of meromorphic matrix functions satisfying the following:
 \begin{itemize}
  \item[$\bullet$] $ \mathfrak{H}_{0}(x,t), \mathfrak{H}_{\infty}\left(x^{-1},t\right) \in \mathrm{GL}_{2}(\mathcal{M}(\C\times \mathfrak{D}))$,\vspace{.2cm}
  \item[$\bullet$]$\mathfrak{H}_{0}\left(0,t\right)=\mathfrak{H}_{\infty}\left(\infty,t\right)=\Itwo $,\vspace{.2cm}
    \item[$\bullet$]  the matrix functions  $\mathfrak{U}_{0}$, $\mathfrak{U}_{\infty}\in  \mathrm{GL}_{2}(\mathcal{M}(\C^{*}\times \mathfrak{D}))$ defined by 
    $$\mathfrak{U}_{0}(x,t):= \mathfrak{H}_{0}(x,t)P(t)^{-1}L_{0}(x)\, \quad \quad 
 \mathfrak{U}_{\infty}(x,t):= \mathfrak{H}_{\infty}(x,t)L_{ \infty}(x) $$
are both solutions of \eqref{eqq1ter}.
 \end{itemize}
 \end{prop}
\begin{proof} We will closely follow  \cite[p. 1034]{sauloy2000regular}, where an analogous result for fixed $t$ has been established, but we also need to take the $t$-dependency into account. We focus on the  existence  of $ \mathfrak{H}_{0}$ as in the statement; the construction of $ \mathfrak{H}_{\infty}$ is analogous. 
The change of variable $Y= \mathfrak{H}_{0}P^{-1} L_0$ leads us to the $q$-difference equation 
\begin{equation}\label{focusQdiff}\sigma_{q,x}( \mathfrak{H}_{0}(x,t)) \mathfrak{A}_0(t)=\mathfrak{A}(x,t) \mathfrak{H}_{0}(x,t)\, .\end{equation} 
It suffices to show that for each $t_0\in \mathfrak{D}$, there exists a unique germ of holomorphic solution  $\mathfrak{H}_{0}$ of  \eqref{focusQdiff} with $\mathfrak{H}_{0}(0,t)=\Itwo $  defined in a neighborhood $U\times \Delta$ of $(x,t)=(0,t_0)\in \C \times \mathfrak{D}$. Indeed,   the functional equation \eqref{focusQdiff} then allows to extend this holomorphic solution to a meromorphic solution on $\C \times \Delta$. By uniqueness, we obtain a unique meromorphic solution on  $\C \times \mathfrak{D}$. 

Let $t_0\in \mathfrak{D}$ and let $\Delta \subset  \mathfrak{D}$ be a sufficiently small disc with center $t_0$. 
Recall that $\mathfrak{D} \subset \C^*$, so that we may assume there exists some $R\in ]0,1[$ such that $|t|>R$ for each $t\in \Delta$. 
Let $\mathbb{D}(0;R)=\{x\in \C| |x|<R\}$.
Moreover, we may assume that on $\mathbb{D}(0;R)\times \Delta$, the matrix function $\mathfrak{A}(x,t)$ is holomorphic and given as the sum of a normally convergent series $\sum_{i,j\geq 0} A_{ij} x^i(t-t_0)^j$ with $A_{ij}\in \mathrm{M}_2(\C)$. In particular, on this product we may write $\mathfrak{A}(x,t) =\sum_{i,j\geq 0} A_{i}(t) x^i$ with  $A_{i}(t)\in \mathrm{M}_2(\mathcal{O}_b(\Delta))$,   such that the power series 
$$\alpha(x):= \sum_{i \geq 0} \alpha_i    x^i \, , \quad \quad \textrm{where} \quad \quad \alpha_i:= ||A_{i  }(t)||_\infty:=\mathrm{sup}_{t\in \Delta} ||A_i(t)||,$$
converges on $\mathbb{D}(0;R)$. Here $||\cdot ||$ is some submultiplicative norm on $\mathrm{M}_2(\C)$ and $\mathcal{O}_b(\Delta)$ denotes the ring of uniformly bounded holomorphic functions on $\Delta$. 
 For $\lambda \in \C^*$, let us consider the map 
$$
\Psi_{\lambda }:\left\{\begin{array}{ccc}
\mathrm{M}_{2}(\mathcal{O}_b(\Delta))&\rightarrow&\mathrm{M}_{2}(\mathcal{O}_b(\Delta))\\
X&\mapsto& \lambda X\mathfrak{A}_0 (t)  - \mathfrak{A}_0 (t) X.
\end{array}\right.
$$
As we may see for instance in \cite[p. 1033]{sauloy2000regular}, the set of eigenvalues of $\Psi_{\lambda }$ is given by 
$$\mathrm{Spec}(\Psi_{\lambda })=\left\{ \lambda \Theta_0- {\Theta}_0,\lambda \Theta_0-\overline{\Theta}_0, \lambda \overline{\Theta}_0-\Theta_0, \lambda \overline{\Theta}_0-\overline{\Theta}_0\right\}. $$

In particular, by the non-resonancy assumption, for every $n\in \N_{>0}$, the endomorphism $\Psi_{q^n}$ is invertible. 
If we  write   $ \mathfrak{H}_{0}(x,t)=\sum_{i=0}^{\infty}H_{i}(t)x^{i}$, with $H_{0}(t)=\mathrm{I}_{2}$, then equation \eqref{focusQdiff} is formally   equivalent to 
$$ \forall n\in \N_{> 0}, \quad    H_{n}(t)=\Psi_{q^{n}}^{-1}\left(\sum_{i=1}^{n}A_i (t) H_{n-i}(t)\right).$$
In particular, $H_{n}(t)$ is uniquely determined from the lower order terms. The endomorphism $\Psi_{q^{n}}$ is equivalent, when $n\to \infty$,  to the endomorphism $X\mapsto q^n X\mathfrak{A}_0 (t) \, .$
As in  \cite[p. 1034]{sauloy2000regular}, one can deduce that, there exists a bound $\beta$ such that
$$\forall n\in \N, \quad \forall X\in \mathrm{M}_{2}(\mathcal{O}_b(\Delta))\, , \quad \quad ||\Psi_{q^{n}}^{-1}(X)||_\infty \leq \beta ||X||_\infty\, .$$
 By induction, one deduces that for each $n\in \N$, the value of $||H_n(t)||_\infty$ is less or equal to the $n$-th coefficient in the power series expansion of 
 $$\frac{||H_0(t)||_\infty}{1-\beta \sum_{k=1}^\infty \alpha_k x^k}\, .$$ Since this power series has positive radius of convergence, we obtain that $ \mathfrak{H}_{0}(x,t)$ is holomorphic on $U\times \Delta$, where $U$ is a neighborhood of $x=0$ in $\C$. 
     \end{proof}
 
 \begin{rem}\label{rem1}
We may also solve order one equations having only meromorphic coefficients. More precisely,  let $\C(\{x\})$ be the field  of germs of meromorphic functions at $x=0$. Let $0\neq c\in \C(\{x\})$, let $v$ be its valuation, and let $c_0\in \C^*$ such that $c=c_0 x^{v}+\dots$. By \cite[p. 1034]{sauloy2000regular}, there exists $0\neq \mathfrak{m}\in \C(\{x\})$ solution of  $\sigma_{q,x}\mathfrak{m}=cc_{0}^{-1}x^{-v}\mathfrak{m}$. Consider the Jacobi theta function $\vartheta_{q}$ and $\mathfrak{e}_{q,c_{0}}$, that are defined in the proof of Lemma~\ref{lemL0L8}. Then, $\mathfrak{e}_{q,c_{0}}\vartheta_{q}^{v}\mathfrak{m}$ is solution of $\sigma_{q,x}(\mathfrak{e}_{q,c_{0}}\vartheta_{q}^{v}\mathfrak{m})=c\mathfrak{e}_{q,c_{0}}\vartheta_{q}^{v}\mathfrak{m}$ and is meromorphic on a punctured neighborhood of $0$ in $\C^*$.
\end{rem}
 
 \begin{prop}\label{PropAnnex2} Let $\mathfrak{U}_\infty$ be as in Proposition \ref{prop:qsolfond}. The following are equivalent. 
 \begin{enumerate}
\item The family \eqref{eqq1ter} is $q$-Schlesinger isomonodromic. \vspace{.2cm} 
\item The matrix $\mathfrak{B}_\infty := \sigma_{q,t} \mathfrak{U}_\infty \cdot  \mathfrak{U}_\infty^{-1} \in \mathrm{GL}_2(\mathcal{M}(\C^*\times \mathfrak{D} ))$ is rational in $x$:
$$\mathfrak{B}_\infty\in \mathrm{GL}_2(\mathcal{M}(\mathfrak{D})(x))\, .$$
 \end{enumerate}
 \end{prop}

\begin{proof} By definition, we have $\mathfrak{B}_\infty=\sigma_{q,t}\mathfrak{H}_\infty \cdot \sigma_{q,t}L_\infty \cdot L_\infty^{-1}\cdot \mathfrak{H}_\infty^{-1}.$ Since $L_\infty$ does not depend on $t$, we actually have 
$\mathfrak{B}_\infty=\sigma_{q,t}\mathfrak{H}_\infty \cdot  \mathfrak{H}_\infty^{-1} \in \mathrm{GL}_2(\mathcal{M}\left((\mathbb{P}^1\setminus \{0\})\times \mathfrak{D}\right)$ and $\mathfrak{B}_\infty(\infty,t)=\Itwo $. 
Moreover, using  $\sigma_{q,x}\sigma_{q,t} \mathfrak{U}_\infty = \sigma_{q,t}\sigma_{q,x} \mathfrak{U}_\infty$ and the definition of  $\mathfrak{B}_\infty$, we find that $X=\mathfrak{B}_\infty$ solves the $q$-difference equation 
\begin{equation}\label{qLaxX} \sigma_{q,t}\mathfrak{A}\cdot X=\sigma_{q,x} X \cdot \mathfrak{A}\, .\end{equation}
It follows from  \cite[Section 1.1.3]{sauloy2000regular} and the non-resonancy condition on the eigenvalues of $\mathfrak{A}_\infty$, that this $q$-difference equation \eqref{qLaxX} admits a unique formal solution  $X=\sum_{n\geq 0} \xi_n(t) x^{-n} \in \mathrm{GL}_2\left(\mathcal{M}(\mathfrak{D})[[x^{-1}]]\right)$ with $\xi_0=\Itwo $. Hence  the power series expansion of $\mathfrak{B}_\infty$ at $x=\infty$ coincides with this unique formal solution. \\
On the other hand, by Proposition  \ref{qLaxPair}, the (non-resonant) family \eqref{eqq1ter} is $q$-Schlesinger isomonodromic if and only if there exists a solution $\mathfrak{B} \in \mathrm{GL}_2(\mathcal{M}(\mathfrak{D})(x))$ of the $q$-Lax equation \eqref{qLaxX} which satisfies $\mathfrak{B}(\infty,t)=\Itwo $. By uniqueness of the formal solution of \eqref{qLaxX}, the equality holds $\mathfrak{B}=\mathfrak{B}_{\infty}$.  The result follows. 
\end{proof}
Given $\mathfrak{U}_0,\mathfrak{U}_\infty$  as in Proposition \ref{prop:qsolfond}, we may define the \emph{Birkhoff connection matrix} $ {\mathfrak{P}}(x,t)$  by 
\begin{equation}\label{BirkhoffDef}  {\mathfrak{P}} :=\mathfrak{U}_{\infty }^{-1}\cdot \mathfrak{U}_{0} \in \mathrm{GL}_2(\mathcal{M}(\C^*\times \mathfrak{D} ))\, .
\end{equation} 
\begin{prop}\label{PropAnnex3}
Let $\mathfrak{U}_0,\mathfrak{U}_\infty,{\mathfrak{P}}$ be as above. The following are equivalent. \vspace{.2cm} 
\begin{enumerate}
\item The {Birkhoff connection matrix} ${\mathfrak{P}}$ is pseudo constant, \emph{\emph{i.e.}} $\sigma_{q,t}\mathfrak{P}=\mathfrak{P}.$\vspace{.2cm} 
\item For $\mathfrak{B}_0, \mathfrak{B}_\infty \in \mathrm{GL}_2(\mathcal{M}(\C^*\times \mathfrak{D} ))$ defined by $\mathfrak{B}_i := \sigma_{q,t} \mathfrak{U}_i \cdot  \mathfrak{U}_i^{-1}$, we have $\mathfrak{B}_0= \mathfrak{B}_\infty\, .$ 
 \end{enumerate}
 \end{prop}
 \begin{proof} 
 We have $$\begin{array}{rclclcl}\displaystyle \mathfrak{P}^{-1}\cdot \sigma_{q,t}\mathfrak{P}&=& \displaystyle \mathfrak{U}_0^{-1}\cdot \mathfrak{U}_\infty \cdot \sigma_{q,t} \mathfrak{U}_\infty^{-1}\cdot \sigma_{q,t}\mathfrak{U}_0&=&\displaystyle \mathfrak{U}_0^{-1}\cdot \mathfrak{B}_\infty^{-1}\cdot \sigma_{q,t}\mathfrak{U}_0
 &=&\displaystyle \mathfrak{U}_0^{-1}\cdot \mathfrak{B}_\infty^{-1}\cdot \mathfrak{B}_0\cdot \mathfrak{U}_0\, .\end{array} $$
 The result follows. 
  \end{proof}
  
  \begin{cor}
If the {Birkhoff connection matrix} ${\mathfrak{P}}$ given in \eqref{BirkhoffDef}  is pseudo constant, then the family  \eqref{eqq1ter} is $q$-Schlesinger isomonodromic. \end{cor}
\begin{proof} Recall from the proof of Proposition \ref{PropAnnex2} that $\mathfrak{B}_\infty=\sigma_{q,t}\mathfrak{H}_\infty \cdot  \mathfrak{H}_\infty^{-1}$. Similarly, we obtain 
$\mathfrak{B}_0=\sigma_{q,t}\mathfrak{H}_0 \cdot \sigma_{q,t}P^{-1}\cdot P\cdot \mathfrak{H}_0^{-1}$, where $P$ is as in \eqref{eqDiagA0}. Note that $\mathfrak{B}_\infty(\infty, t)=\Itwo $ and 
$\mathfrak{B}_0(0, t)=P(qt)^{-1}P(t)$. By assumption and Proposition \ref{PropAnnex3}, we have $\mathfrak{B}_\infty=\mathfrak{B}_0$. It follows that $\mathfrak{B}_\infty\in \mathrm{GL}_2\left( \mathcal{M}(\C^*\times \mathfrak{D})\right)$ can be meromorphically continued to $x=0$ and $x=\infty$. Hence $\mathfrak{B}_\infty$ is rational in $x$. We conclude by Proposition \ref{PropAnnex2}.
\end{proof}

The above corollary establishes the sought relation between the two notions of $q$-isomonodromy (in the non-resonant case with $|q|>1$):   pseudo-constancy of the Birkhoff connection matrix is a  stronger  requirement than 
$q$-isomonodromy as in Section \ref{sec1}. Note that the  Birkhoff connection matrix in \eqref{BirkhoffDef}  is not canonically defined: it depends on the choices of $L_0, L_\infty$ and $P$. However, for any choice, its pseudo-constancy implies $q$-Schlesinger isomonodromy. As a final remark, we indicate that for example in \cite{dreyfus2017isomonodromic,jimbo1996q}, another type (again non-canonical) of Birkhoff connection matrix has been considered, namely 
$$\widetilde{\mathfrak{P}}:=\mathfrak{U}_{\infty }^{-1}\cdot \widetilde{\mathfrak{U}}_{0}\, , $$ with $\widetilde{\mathfrak{U}}_{0}:= {\mathfrak{U}}_{0}\cdot P$. 
It has been shown in
 \cite[Proposition 1]{dreyfus2017isomonodromic}, see also \cite[Theorem 3]{jimbo1996q}, that the analogue of Proposition \ref{PropAnnex3} for $\widetilde{\mathfrak{U}}_{0},\mathfrak{U}_{\infty },\widetilde{\mathfrak{P}}$ holds under the additional assumption that $\mathfrak{A}_0(t)$ is either constant or proportional to $t$. The choice of ${\mathfrak{U}}_{0}$ in  the above exposition was made to circumvent this assumption.  
 
    \bibliographystyle{alpha}
\bibliography{biblio}

\begin{thebibliography}{BHHW18}

\bibitem[BG10]{bruno2010asymptotic}
A.~Bruno and I.~Goryuchkina.
\newblock Asymptotic expansions of solutions of the sixth {P}ainlev{\'e}
  equation.
\newblock {\em Transactions of the Moscow Mathematical Society}, 71:1--104,
  2010.

\bibitem[BHHW18]{bachmayr2016differential}
Annette Bachmayr, David Harbater, Julia Hartmann, and Michael Wibmer.
\newblock Differential embedding problems over complex function fields.
\newblock {\em Documenta Mathematica}, 23:241--291, 2018.

\bibitem[Bol97]{bolibruch1997isomonodromic}
Andre~A Bolibruch.
\newblock On isomonodromic deformations of {F}uchsian systems.
\newblock {\em Journal of dynamical and control systems}, 3(4):589--604, 1997.

\bibitem[Dre14]{dreyfus2014building}
Thomas Dreyfus.
\newblock Building meromorphic solutions of q-difference equations using a
  {B}orel--{L}aplace summation.
\newblock {\em International Mathematics Research Notices},
  2015(15):6562--6587, 2014.

\bibitem[Dre15]{dreyfus2015confluence}
Thomas Dreyfus.
\newblock Confluence of meromorphic solutions of q-difference equations.
\newblock In {\em Annales de l'institut Fourier}, volume~65, pages 431--507,
  2015.

\bibitem[Dre17]{dreyfus2017isomonodromic}
Thomas Dreyfus.
\newblock Isomonodromic deformation of q-difference equations and confluence.
\newblock {\em Proceedings of the American Mathematical Society},
  145(3):1109--1120, 2017.

\bibitem[DVZ09]{di2009q}
Lucia Di~Vizio and Changgui Zhang.
\newblock On $ q $-summation and confluence ($ q $-sommation et confluence).
\newblock In {\em Annales de l'institut Fourier}, volume~59, pages 347--392,
  2009.

\bibitem[Fuc07]{fuchs1907lineare}
Richard Fuchs.
\newblock {\"U}ber lineare homogene {D}ifferentialgleichungen zweiter {O}rdnung
  mit drei im {E}ndlichen gelegenen wesentlich singul{\"a}ren {S}tellen.
\newblock {\em Mathematische Annalen}, 63(3):301--321, 1907.

\bibitem[HL04]{hinkkanen2004meromorphic}
Aimo Hinkkanen and Ilpo Laine.
\newblock The meromorphic nature of the sixth {P}ainlev{\'e} transcendents.
\newblock {\em Journal d'Analyse Mathematique}, 94(1):319--342, 2004.

\bibitem[H{\"o}r85]{hormander1985analysis}
Lars H{\"o}rmander.
\newblock The analysis of linear partial differential operators. {III}, volume
  274 of {G}rundlehren der {M}athematischen {W}issenschaften, 1985.

\bibitem[JK94]{joshi1994direct}
Nalini Joshi and Martin~D Kruskal.
\newblock A direct proof that solutions of the six {P}ainlev{\'e} equations
  have no movable singularities except poles.
\newblock {\em Studies in Applied Mathematics}, 93(3):187--207, 1994.

\bibitem[JM81]{jimbo1981monodromy}
Michio Jimbo and Tetsuji Miwa.
\newblock Monodromy perserving deformation of linear ordinary differential
  equations with rational coefficients. {II}.
\newblock {\em Physica D: Nonlinear Phenomena}, 2(3):407--448, 1981.

\bibitem[JS96]{jimbo1996q}
Michio Jimbo and Hidetaka Sakai.
\newblock A q-analog of the sixth {P}ainlev{\'e} equation.
\newblock {\em Letters in Mathematical Physics}, 38(2):145--154, 1996.

\bibitem[Kha96]{MyBook}
Dmitry Khavinson.
\newblock {\em Holomorphic partial differential equations and classical
  potential theory}.
\newblock Departamento de an{\'a}lisis matem{\'a}tico. Universidad de la
  Laguna, 1996.

\bibitem[Lor05]{loray2004theoremes}
Frank Loray.
\newblock Sur les th{\'e}oremes {I} et {II} de {P}ainlev{\'e}.
\newblock In {\em Geometry and Dynamics, International conference in honnour of
  the 60th Anniversary of Alberto Verjovsky, 2005, Cuernavaca, Mexico.}, pages
  165--190, 2005.

\bibitem[Lor16]{loray2016isomonodromic}
Frank Loray.
\newblock Isomonodromic deformation of {L}am{\'e} connections, {P}ainlev{\'e}
  {VI} equation and {O}kamoto symmetry.
\newblock {\em Izvestiya: Mathematics}, 80(1):113, 2016.

\bibitem[Man10]{mano2010asymptotic}
Toshiyuki Mano.
\newblock Asymptotic behaviour around a boundary point of the q-{P}ainlev{\'e}
  {VI} equation and its connection problem.
\newblock {\em Nonlinearity}, 23(7):1585, 2010.

\bibitem[Ohy09]{Ohyama}
Yousuke Ohyama.
\newblock Analytic solutions to the sixth q-{P}ainlev{\'e} equation around the
  origin.
\newblock {\em RIMS Kôkyûroku Bessatsu}, 01 2009.

\bibitem[Oka86]{okamoto1986studies}
Kazuo Okamoto.
\newblock Studies on the {P}ainlev{\'e} equations.
\newblock {\em Annali di Matematica pura ed applicata}, 146(1):337--381, 1986.

\bibitem[Pra86]{praagman1986fundamental}
Cornelis Praagman.
\newblock Fundamental solutions for meromorphic linear difference equations in
  the complex plane, and related problems.
\newblock {\em J. Reine Angew. Math}, 369:101--109, 1986.

\bibitem[Ram92]{ramis1992growth}
Jean-Pierre Ramis.
\newblock About the growth of entire functions solutions of linear algebraic $
  q $-difference equations.
\newblock In {\em Annales de la Facult{\'e} des sciences de Toulouse:
  Math{\'e}matiques}, volume~1, pages 53--94, 1992.

\bibitem[RSZ13]{ramis2009local}
J-P Ramis, Jacques Sauloy, and Changgui Zhang.
\newblock Local analytic classification of q-difference equations.
\newblock {\em Ast\'erisque}, 355:vi+151 pp, 2013.

\bibitem[Sak01]{sakai2001rational}
Hidetaka Sakai.
\newblock Rational surfaces associated with affine root systems and geometry of
  the {P}ainlev{\'e} equations.
\newblock {\em Communications in Mathematical Physics}, 220(1):165--229, 2001.

\bibitem[Sau00]{sauloy2000regular}
Jacques Sauloy.
\newblock Regular singular linear q-difference systems: classification,
  connection matrix and monodromy.
\newblock volume~50, pages 1021--+, 2000.

\bibitem[Zha02]{zhang2002sommation}
Changgui Zhang.
\newblock Une sommation discr{\`e}te pour des {\'e}quations aux
  q-diff{\'e}rences lin{\'e}aires et {\`a} coefficients analytiques:
  th{\'e}orie g{\'e}n{\'e}rale et exemples.
\newblock In {\em Differential equations and the Stokes phenomenon}, pages
  309--329. World Scientific, 2002.

\end{thebibliography}
\end{document}